\documentclass[11pt,reqno]{amsart}
\usepackage{hyperref}

\allowdisplaybreaks[4]

\usepackage{amsmath}
\usepackage{amssymb}
\usepackage{mathrsfs}
\usepackage{amsthm}
\usepackage{bm}
\usepackage{booktabs}
\usepackage{float}
\usepackage{graphicx,epstopdf}
\usepackage{subfig}
\usepackage{enumerate}
\usepackage{enumitem}
\usepackage{extarrows}
\usepackage{geometry}
\usepackage{cases}
\usepackage{tikz}
\usepackage{subcaption}
\usetikzlibrary{calc}
\usepackage{subfiles}
\usepackage{graphicx}
\usepackage{algorithm}
\usepackage{algpseudocode}
\usepackage{empheq}
\usepackage[numbers,sort&compress]{natbib}
\usepackage{color}
\theoremstyle{plain}
\newtheorem{lemma}{Lemma}[section]
\newtheorem{theorem}[lemma]{Theorem}
\newtheorem{dfn}[lemma]{Definition}
\newtheorem{corollary}[lemma]{Corollary}
\newtheorem{remark}[lemma]{Remark}
\newtheorem{prop}[lemma]{Proposition}
\newtheorem{assumption}{Assumption}
\newtheorem{claim}{Claim}
\newtheorem{example}{Example}
\numberwithin{equation}{section}

\begin{document}
	\bibliographystyle{plainnat}
	\title[Strong Feller property for SPDEs on graphs]{Quantifying the effect of graph structure on strong Feller property of SPDEs}
	\author{Jianbo Cui}
	\address{Department of Applied Mathematics, The Hong Kong Polytechnic
		University, Hung Hom, Kowloon, Hong Kong, SAR, China.}
	\email{jianbo.cui@polyu.edu.hk}
    \author{Tonghe Dang}
    \address{Department of Applied Mathematics, The Hong Kong Polytechnic
		University, Hung Hom, Kowloon, Hong Kong, SAR, China.}
	\email{tonghe.dang@polyu.edu.hk}
	\author{Jialin Hong}
	\address{State Key Laboratory of Mathematical Sciences, Academy of Mathematics and Systems Science, Chinese Academy of
		Sciences, Beijing 100190, China, and School of Mathematical Sciences, University of Chinese Academy of Sciences, Beijing 100049, China.
	}
	\email{hjl@lsec.cc.ac.cn}
	\author{Zhengkai Wang}
	\address{State Key Laboratory of Mathematical Sciences, Academy of Mathematics and Systems Science, Chinese Academy of
			Sciences, Beijing 100190, China, and School of Mathematical Sciences, University of Chinese Academy of Sciences, Beijing 100049, China.}
		\email{wangzhengkai@lsec.cc.ac.cn(Corresponding author)}
		
		\thanks{This research is partially supported by MOST National Key R\&D Program No. 2024YFA1015900, the Hong Kong Research Grant Council GRF grant  15302823, NSFC grant (No. 12301526, No. 12471386, No. 12461160278), NSFC/RGC Joint Research Scheme N PolyU5141/24, and the CAS AMSS-PolyU Joint Laboratory of Applied Mathematics. 
        }
		\begin{abstract}
	This paper investigates how the structure of the underlying graph influences the behavior of stochastic partial differential equations (SPDEs) on finite tree graphs, where each edge is driven by space-time white noise. We first introduce a novel graph-based null decomposition approach to analyzing the strong Feller property of the Markov semigroup generated by SPDEs on tree graphs. By examining the positions of zero entries in eigenfunctions of the graph Laplacian operator, we establish a sharp upper bound on the number of noise-free edges that ensures both the strong Feller property and irreducibility. Interestingly, we find that the addition of noise to any single edge is sufficient for chain graphs, whereas for star graphs, at most one edge can remain noise-free without compromising the system’s properties. Furthermore, under a dissipative condition, we prove the existence and exponential ergodicity of a unique invariant measure.
		\end{abstract}
		
	\maketitle
    \section{Introduction}\label{sec:intro} 
    Partial differential equations (PDEs) on graphs have become an increasingly active area of research, motivated  by wide applications in diverse fields such as image processing \cite{elmoataz2016morphological} and traffic dynamics \cite{li2017convex}. Substantial progress has been made for deterministic models, including gradient flows \cite{chow2012fokker,chow2018entropy}, Hamiltonian flows \cite{cui2022time,cui2023wasserstein}, Schr\"odinger equations \cite{chow2019discrete}, Hamilton--Jacobi equations \cite{cui2025finite,cui_dang2025HJB} as well as wave and heat equations \cite{fijavz2007variational,arendt2014diffusion}. In realistic applications, systems are often subject to random influences from external perturbations or intrinsic noises, which introduce uncertainty and make stochastic modeling indispensable. This has led to growing interest in incorporating stochastic effects into PDEs posed on graph structures. In recent years, SPDEs on graphs have been actively studied, with representative advances for parabolic SPDEs \cite{cerrai2017spdes,kovacs2021stochasticdiffusion,cui2025large} and stochastic Schr\"odinger equation \cite{cui2023optimal}. In particular, the regularity of Markov semigroup corresponding to the SPDEs and long‑time dynamical properties of solutions have drawn increasing attention in recent decades \cite{albeverio2013invariant,kovacs2023parabolic,fkirine2025strong}.
    
    In this paper, we are concerned with the SPDE on the tree $\Gamma$:
   \makeatletter
   \begingroup
   \def\tagform@#1{\maketag@@@{\small(#1)}}
   \begin{subequations}\label{eq:parabolic}
   	\begin{empheq}[left=\empheqlbrace]{flalign}
   		&\tfrac{\partial{u}_j}{\partial t}(t,x_j) = \tfrac{\partial^2{u}_j}{\partial x^2}(t,x_j)
   		+ b_j\bigl(u_j(t,x_j)\bigr)
   		+ Q_j \tfrac{\partial W_j}{\partial t}(t,x_j),
   		&& t>0,\, x_j\in e_j,\label{eq:edge-eq} \\
   		&u_j(t,v_i) = u_l(t,v_i),
   		&& t\ge0,\, \forall\, e_j,e_l \in E_{v_i}, \label{eq:continuity} \\
   		&\textstyle\sum_{j=1}^{m} \phi_{ij}\, \tfrac{\partial{u}_j}{\partial x}(t,v_i) = 0,
   		&& t\ge0,\, 1\le i\le n, \label{eq:kirchhoff} \\
   		&u_j(0,x_j) = u_j^0(x_j),
   		&& x_j \in e_j,\, 1\le j\le m. \label{eq:initial}
   	\end{empheq}
   \end{subequations}
   \endgroup
   \makeatother
    Here, $\Gamma=(V(\Gamma),E(\Gamma))$ is a finite connected tree, consisting of a vertex set $V(\Gamma)=\{v_1,\dots,v_n\}$ and an edge set $E(\Gamma)=\{e_1,\dots,e_m\}$, where $m\in\mathbb{N}^+$ and $n=m+1$. For any vertex $v\in V(\Gamma)$, let $E_v$ denote the set of edges incident to $v$. For $j=1,\ldots,m$, we normalize and
    parameterize every edge on the interval [0,1], i.e., $e_j=[0,1]$, the drift coefficient $b_j:\mathbb{R}\rightarrow\mathbb{R}$ is a globally Lipschitz continuous function, $\tfrac{\partial W_j}{\partial t}$ represents independent space-time white noise defined on a complete filtered probability space $\left(\Omega,\mathcal{F},(\mathcal{F}_t)_{t\ge0},\mathbb{P}\right)$, the coefficient $Q_j$ denotes an operator-valued function (see \eqref{Q}), and $u^0_j$ is the initial datum. The classical homogeneous Neumann--Kirchhoff condition \eqref{eq:kirchhoff} is imposed to characterize the flux conservation at the vertices. Here $(\phi_{ij})_{n\times m}$ is the incidence matrix of an oriented version of $\Gamma$; see Subsection~\ref{notation} for details. Equation \eqref{eq:edge-eq}, known as the stochastic reaction‑diffusion equation on networks, arises in diverse applied contexts, including synchronization phenomena in dynamical systems \cite{bessaih2025synchronization}, neurophysiological modeling \cite{bonaccorsi2008stochastic}, and models for the motion of molecular motors \cite{cerrai2017spdes} on graph structures.
     
 A fundamental property in the long-time dynamical analysis of stochastic systems is the strong Feller property of the Markov semigroup. This property provides a regularizing effect by mapping bounded measurable functions into the space of bounded continuous functions, thereby linking measure theoretic properties of a process to the corresponding topological properties \cite{hairer2018strong,hairer2006ergodicity}. It not only serves as a cornerstone for establishing the existence and regularity of transition densities, but also plays a crucial role in analyzing the uniqueness of invariant measures and the ergodicity of the underlying Markov semigroup. Together with irreducibility (see Definition~\ref{def}), the strong Feller property guarantees that any existing equilibrium state (or invariant measure) is unique. In applied contexts, these properties are of great importance in the study of stochastic models arising from real-world applications. Typical examples include stochastic climate models and turbulence models \cite{majda2001mathematical,hairer2006ergodicity}, where these properties are closely related to predictability, statistical equilibrium, and uncertainty quantification \cite{batou2013calculation}. For SPDEs on Euclidean domains, the strong Feller property, irreducibility and invariant measure have been extensively studied, and we refer to the monograph \cite{da1996ergodicity} and references therein for details.
Despite of fruitful results in the Euclidean case, not much is known for the regularizing properties of Markov semigroup and long-time dynamics for SPDEs on graphs. For example, in the case of the heat equation with Brownian motion perturbations at all vertices, \cite{kovacs2023parabolic} investigated the nonlinear case and proved that the associated solution is Markov and Feller. \cite{fkirine2025strong} further studied the linear equation on a tree and proved the strong Feller property when noise is presented in all of the boundary vertices except one. It also provided the proof of the existence and uniqueness of an invariant measure for the linear equation. For SPDE \eqref{eq:parabolic} on graph with trace class noise, \cite{albeverio2013invariant} studied the existence and uniqueness of an invariant measure under dissipative polynomially bounded nonlinearity by Yosida approximation and splitting methods. 

Compared with the Euclidean case, the study of the strong Feller property of Markov semigroup for SPDEs on graphs faces several unique challenges due to the interplay between continuous intervals and discrete vertices under stochastic forcing. First, the graph structure combines Euclidean topology (each edge as a continuous interval) with discrete incidence geometry (vertex adjacency), introducing multiscale effects absent in Euclidean domains. For example, the spectral properties (e.g., spectral gap) of the graph Laplacian depend intrinsically on the graph’s incidence structure, making the analytical framework more difficult than in Euclidean settings. Second, the dynamics of equations are defined on edges but coupled through vertex conditions, which introduces structural complexity that complicates the analysis of regularity and properties of solution. Moreover, the stochastic forcing may act on part of the edges or vertices, leading to an incomplete coverage of noise on graph. The interaction of this partial noise with the graph geometry and vertex coupling essentially affects the propagation of randomness, the regularizing properties of associated Markov semigroup, and the long-time dynamics. Consequently, quantifying how the graph structure affects the regularizing properties of the associated Markov semigroup remains unclear.


In this paper, we study the SPDE \eqref{eq:parabolic} driven by edge‑based space‑time white noise which is allowed to act on partial set of edges, with a focus on the strong Feller property, irreducibility, and invariant measure. To overcome the structural complexity arising from the interaction between edges and vertex conditions, we incorporate the graph's adjacency matrix into the analysis of the Laplacian operator, and reformulate the classical condition on the inclusive relation of operators (see \cite[Theorem 7.2.1]{da1996ergodicity}) for the strong Feller property of the linear SPDE as a spectrum problem for the corresponding graph Laplacian operator. To systematically handle the structural complexity of general trees, we propose an approach based on a null decomposition of the tree graph, which reduces the original global problem to its minimal indivisible components, referred to as $S$-atoms. A key insight from our analysis is that the positions of zero entries in the graph Laplacian eigenfunctions reveal how noise-free edges affect the regularizing influence of noise, enabling us to quantify the impact of partial noise propagation on the tree. For each $S$-atom, we analyze its support and core sets, which yields the maximal number of noise‑free edges that can be permitted while maintaining the strong Feller property. By aggregating the results over all $S$-atoms in the decomposition, we ultimately obtain a sharp upper bound of the number of noise-free edges to quantify the influence of graph structure on maintaining the strong Feller property for linear equation. For the nonlinear equation, we employ the Girsanov theorem to perform a measure transformation, which establishes the equivalence of distributions and thereby yields the same result for the strong Feller property as in the linear case; we refer to Theorem \ref{main theorem} for details.
 
 For the irreducibility of SPDE \eqref{eq:parabolic} on a tree, we show that the necessary and sufficient condition in the linear case coincides with that for the strong Feller property. Furthermore, by presenting long-time regularity estimates in a Sobolev space with a positive index and the  asymptotic attractiveness of the solution, we obtain the existence and exponential ergodicity of a unique invariant measure for \eqref{eq:parabolic} with a dissipative nonlinearity; see Theorem \ref{invariantmeasureth}. Finally, we apply our main theoretical results (Theorems \ref{main theorem} and \ref{invariantmeasureth}) to three fundamental tree structures: chain graphs, star graphs, and an $S$-atom. We surprisingly find that for the chain graph, the Markov semigroups of SPDE remain strong Feller and irreducible as long as at least one edge is noisy. For the star graph, these properties are maintained precisely when all edges are noisy with at most one exception, which appears consistent with \cite[Example 4.8]{fkirine2025strong} where the noise is imposed on the vertices (see Remark \ref{vertexnoise}). For the considered $S$-atom, the sharp upper bound on noise‑free edges that still guarantees strong Feller and irreducible semigroups can also be obtained, and the positions of noise-free edges are determined by eigenfunctions of the corresponding graph Laplacian. However, for a general graph, it is still difficult to identify which edges can be left noise-free. Moreover, in the three cases discussed above, the solution of the SPDE with a dissipative nonlinearity admits a unique invariant measure.
    
This paper is organized as follows. In Section~\ref{motivation}, we present some motivating observations of strong Feller property for linear SPDEs on chain graph and star graph. Section~\ref{Preliminaries} introduces some preliminaries including notations, definitions as well as the well-posedness of the SPDE on graph. We also present some definitions about tree and introduce a null decomposition of trees. In Section~\ref{main}, we present the main results regarding the strong Feller property, irreducibility, and invariant measure of the Markov semigroup associated with SPDE on trees. Numerical validations are provided in Section~\ref{numerical}. Section~\ref{proof} is devoted to the proofs of the main results. Appendix \ref{appendix} provides supporting material, including some proofs and several additional numerical experiments.
    \section{Motivating observations}\label{motivation} 
    In this section, we present some interesting findings on
    the behaviors of 
    linear SPDEs on trees, i.e., \eqref{eq:parabolic} with $b_j=0$ for $j=1,2,\ldots,m$.
To gain insight into the strong Feller property of the stochastic system, we perform numerical simulations on two prototypical graphs: a chain graph and a star graph (see Fig.~\ref{SF_graphs}). We approximate the solution by a full discretization whose spatial direction is the spectral Galerkin method and temporal direction is the accelerated exponential Euler method \cite{djurdjevac2024higher}. Recall that the strong Feller property intuitively implies a smoothing effect: the Markov semigroup should map bounded measurable functions to bounded continuous functions. In the numerical experiments, we quantify the smoothing effect of the semigroup by the difference $|\mathcal{R}_T\phi(X_0+\mathcal{E})-\mathcal{R}_T\phi(X_0)\big|$, which is approximated using the Monte Carlo method 
\begin{flalign}\label{lineardiffernece}
|\mathcal{R}_T\phi(X_0+\mathcal{E})-\mathcal{R}_T\phi(X_0)\big|&\approx\frac{1}{M_{traj}}\sum_{m=1}^{M_{traj}}|\phi\bigl(v^{\tau,N}(T,X_0+\mathcal{E})\bigl)\big|,
\end{flalign} where $\mathcal R_T$ is the Markov semigroup defined in \eqref{linear semigroup}, $M_{traj}$ is the number of trajectory, $\epsilon \in \{ 10^{-\tfrac{4k}{7}} : k = 0,1,\dots,7\}$, $\mathcal{E}=\epsilon\sum_{l=1}^{N}\Psi^{1,l}$  with $\{\Psi^{1,l}\}_{l=1}^N$ being a family of eigenfunctions of graph Laplacian operator, sign function $\phi(h)=\text{sgn}(\langle h,\Psi^{1,1}\rangle)\in\mathcal{B}_b(H)$, $v^{\tau,N}(T,X_0+\mathcal{E})$ denotes the numerical solution to the linear SPDE at time $T$ with initial value $X_0+\mathcal{E}$, and $N,\tau$ are respectively the dimension of the spectral Galerkin projection and the temporal stepsize. Note that  by the definition of strong Feller property, if the right‑hand side of \eqref{lineardiffernece} converges to zero as $\mathcal{E}$ converges to zero, then $\mathcal{R}_T$ is strong Feller; otherwise, it is not strong Feller.
    	
  To specify the noisy and noise-free edges, we partition the edge set $E(\Gamma)$ into two disjoint subsets: $E(\Gamma) = \mathcal{Y}(\Gamma) \cup \mathcal{Z}(\Gamma)$, where $\mathcal{Y}(\Gamma)$ consists of noise-induced edges and $\mathcal{Z}(\Gamma)$ comprises noise-free edges. Accordingly, the operator $Q_j,\, j=1,2,\dots,m$ in \eqref{eq:edge-eq} can be described as
  \begin{flalign}\label{Q}
  	Q_j=\begin{cases}
  		I_d,\ &e_j\in\mathcal{Y}(\Gamma),\\
  		0,\ &e_j\in\mathcal{Z}(\Gamma),
  	\end{cases}
  \end{flalign}where $I_d$ is an identity operator on $L^2(0,1)$.

We consider different noise configurations on chain graph and star graph represented by different lines in Figs.~\ref{SF_graphs}(a) and \ref{SF_graphs}(b). For the chain graph displayed in Fig.~\ref{SF_graphs}(a), we observe that all lines except the one of zero noisy edge converge to zero as $\mathcal{E}$ converges to zero, which indicates that the Markov semigroup corresponding to the linear equation exhibits the strong Feller property provided noise acts on at least one edge. While the star graph case needs a high requirement on the amount of noise-induced edge. For example, for the star graph with three edges presented in Fig.~\ref{SF_graphs}(b), we find that only the lines of two noisy edges and the three noisy edges converge to zero, as $\mathcal{E}$ converges to zero. This suggests that the strong Feller property can be retained if noise is present on all edges with at most one exception (i.e., at most one edge is noise‑free). These observations illustrate that the graph structure governs the strong Feller property of Markov semigroup for SPDEs on these graphs.

	\begin{figure}[htbp]
		\centering
		
		\begin{minipage}[t]{0.42\textwidth}
			\centering
			\begin{tikzpicture}[scale=1.1,line width=0.8pt]
				\foreach \i in {1,2,3,4,5} {
					\filldraw (\i,0) circle (1.5pt);
					\ifnum\i<5
					\draw (\i,0) -- (\i+1,0);
					\fi
				}
				\foreach \i in {1,...,5}
				\node[below] at (\i,0) {$\i$};
			\end{tikzpicture}
		\end{minipage}
		\hspace{7mm}
		\begin{minipage}[t]{0.42\textwidth}
			\centering
			\begin{tikzpicture}[scale=1.1,line width=0.8pt]
				\filldraw (0,0) circle (1.5pt);
				\node[below=2pt]{$c$};
				\foreach \a/\l in {90/1,210/2,330/3} {
					\draw (0,0) -- (\a:1);
					\filldraw (\a:1) circle (1.5pt);
					\node at (\a:1.25) {$\l$};
				}
			\end{tikzpicture}
		\end{minipage}
		
		\vspace{0.3cm}
		
		\begin{minipage}[t]{0.4\textwidth}
			\centering
			\includegraphics[width=\linewidth]{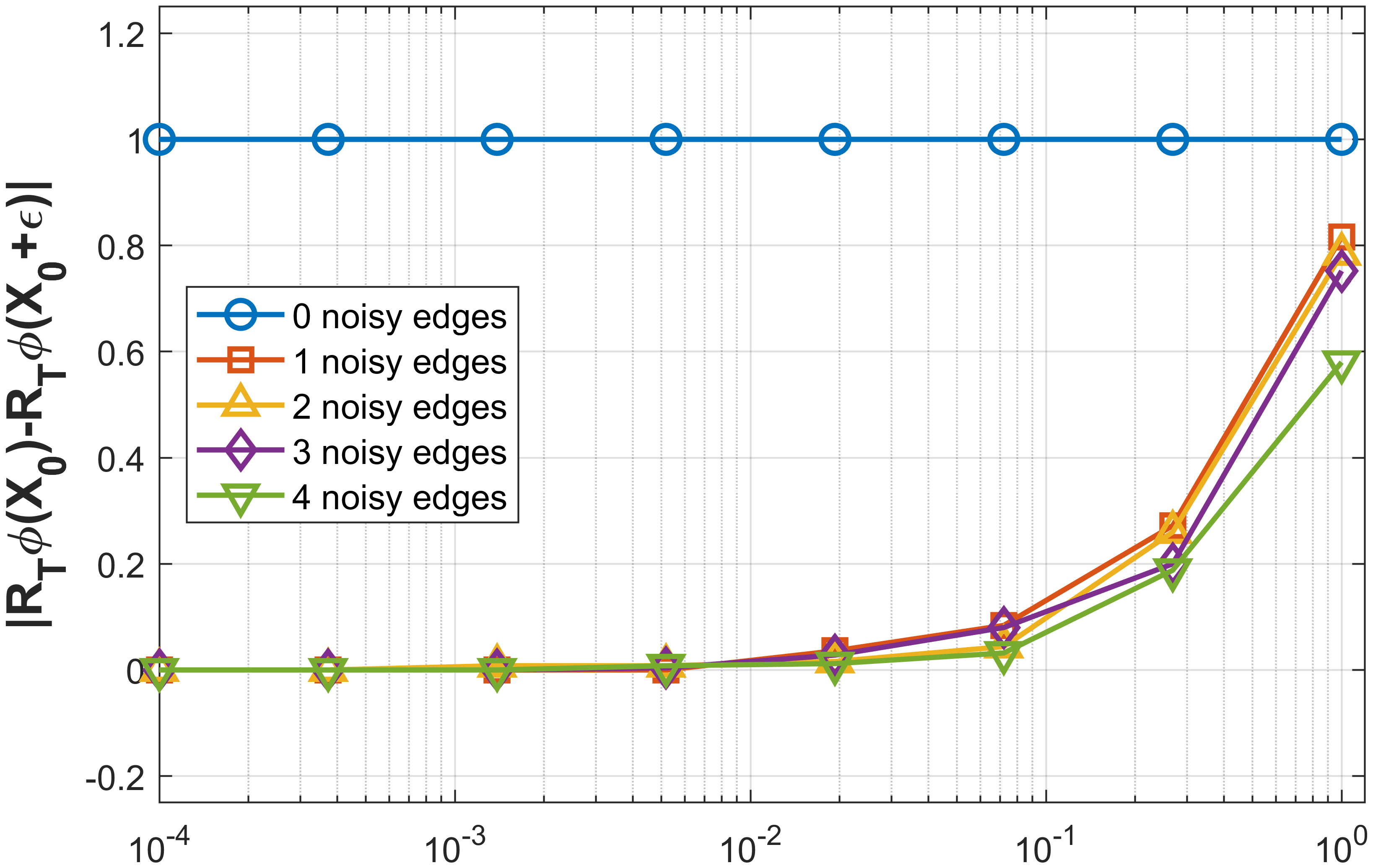}
			\captionsetup{
				singlelinecheck=false,
				justification=centering
			}
			\caption*{(a) Chain graph with $4$ edges}
		\end{minipage}
		\hspace{12mm}
		\begin{minipage}[t]{0.4\textwidth}
			\centering
			\includegraphics[width=\linewidth]{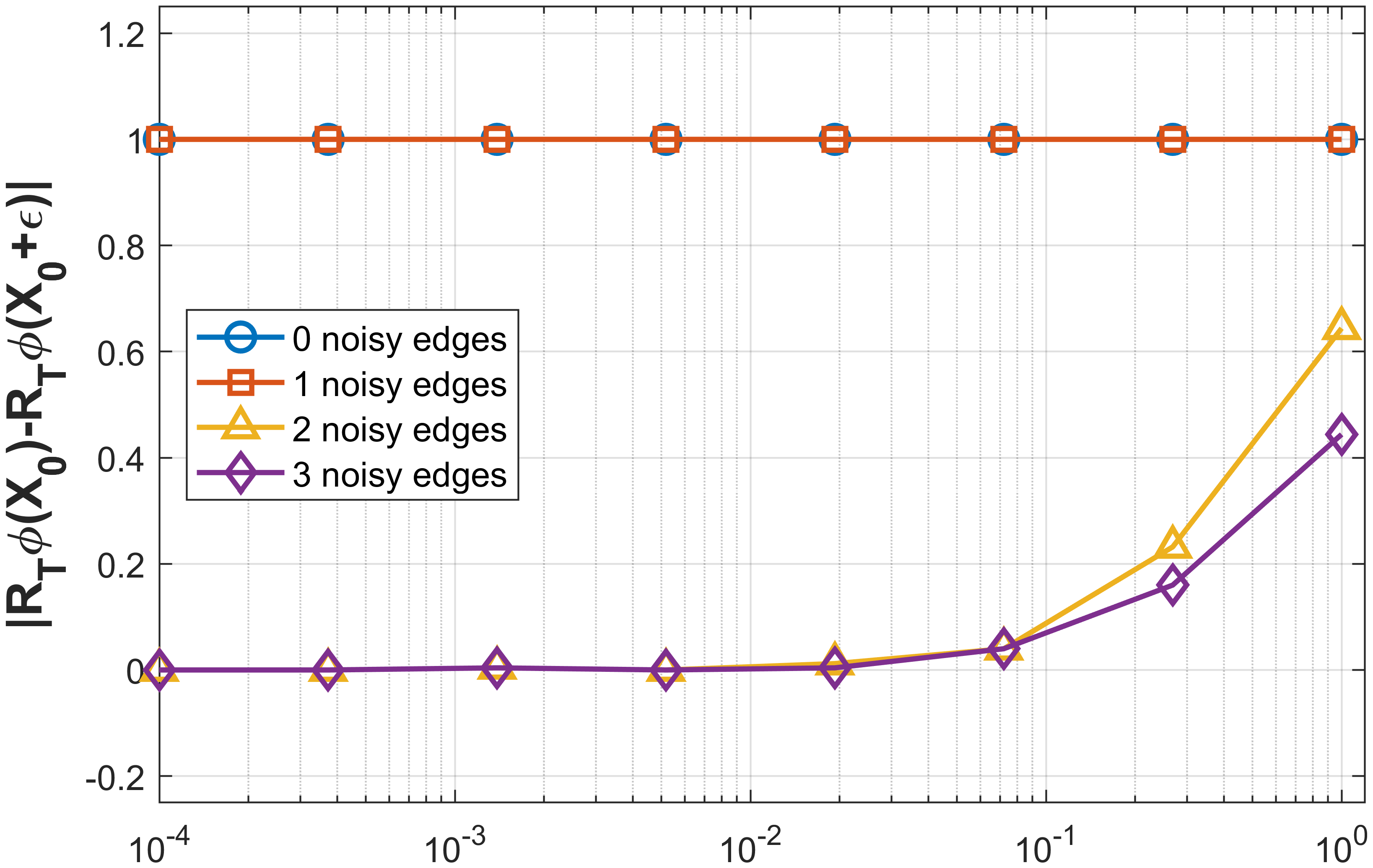}
			\captionsetup{
				singlelinecheck=false,
				justification=centering
			}
			\caption*{(b) Star graph with $3$ edges}
		\end{minipage}
		
		\caption{
			Strong Feller property of \eqref{eq:parabolic} with $b_j=0$ for all $j$ on different graphs.
			Common parameters: $X_0=0,\,N=2^6,\, \tau=2^{-5},\, T=2^{-1},\ M_{{traj}}=500,\,\epsilon \in \{ 10^{-\tfrac{4k}{7}} : k = 0,1,\dots,7\},\, \mathcal{E}=\epsilon\sum_{l=1}^{N}\Psi^{1,l}.$ For chain graph, 
$\Psi^{1,l}(\mathbf{x})=[\sqrt{2}\sin(\sqrt{\mu_l}(1-x_1))+\sin(\sqrt{\mu_l}x_1),
    \sin(\sqrt{\mu_l}(1-x_2),\sin(\sqrt{\mu_l}x_3),
    \sin(\sqrt{\mu_l}(1-x_4))-\sqrt{2}\sin(\sqrt{\mu_l}x_4)]^\top$ with $\mu_l=(\tfrac{\pi}{4}+2(l-1)\pi)^2$ and $\mathbf{x}=[x_1,x_2,x_3,x_4]^\top$. For star graph, $\Psi^{1,l}(\mathbf{x})=[\cos(\sqrt{\mu_l}x_1),-\cos(\sqrt{\mu_l}x_2),0]^\top$ with $\mu_l=(\tfrac{1}{2}+(l-1))^2\pi^2$ and $\mathbf{x}=[x_1,x_2,x_3]^\top$.  }
		\label{SF_graphs}
	\end{figure}


    Furthermore, we make a comparison with stochastic differential equation on Euclidean domain. Consider the following example on $\mathbb{R}^2$ (see \cite{hairer2006ergodicity})
    \begin{flalign*}
    	\begin{cases}
    		dx(t)=-x(t)dt+dw(t),\ &x(0)=x_0,\\
    		dy(t)=-y(t)dt,\ &y(0)=y_0,
    	\end{cases}
    \end{flalign*}
where $t>0$ and $w(t)$ is a standard one-dimensional Brownian motion. Denote the corresponding Markov semigroup by $\mathcal{G}_t$. Note that for the test function $\psi(x,y) = \operatorname{sgn}(y) \in \mathcal{B}_b(\mathbb{R}^2)$, we have $\mathcal{G}_t\psi = \psi$ for all $t \ge 0$ and $\psi$ is discontinuous. Consequently, the semigroup $\mathcal{G}_t$ fails to be strong Feller for any $t \ge 0$.
This suggests that for a decoupled system, a lack of noise of some component can destroy the strong Feller property of the system, because noise can not be propagated into that component (here the $y$ direction). However, on a graph, the Neumann--Kirchhoff conditions at the vertices act as ``mixing" junctions. The noise-induced smoothing effect from some edges can ``flow" into adjacent noise-free edges.

These observations motivate us to investigate how the graph structure influences the strong Feller property of SPDEs on graphs. Moreover, we find that the graph structure also impacts other properties of the associated Markov semigroup, such as irreducibility and the existence and uniqueness of invariant measures (see Theorems \ref{main theorem} and \ref{invariantmeasureth}).
%
 
    \section{Preliminaries}\label{Preliminaries}
	We begin this section by setting up the necessary notation for a tree $\Gamma$, function spaces and operators on $\Gamma$, and then presenting the well-posedness of parabolic SPDE \eqref{eq:parabolic} on $\Gamma$ and introducing definitions of strong Feller property, irreducibility, invariant measure and exponential ergodicity of Markov semigroup. Furthermore, we introduce a null decomposition of trees, which will be used to analyze the effect of graph structure on properties of  Markov semigroups in the subsequent sections. 
	\subsection{Notation and assumptions}\label{notation}
	 We begin by introducing basic concepts and notations from graph theory. Two distinct vertices are called adjacent if they are connected by an edge; two edges are adjacent if they share a common vertex. A vertex and an edge are said to be incident if the edge is attached to that vertex.  These concepts are algebraically represented by two fundamental matrices. For graph $\Gamma$ with vertex set $V(\Gamma)=\{v_1,v_2,\ldots,v_n\}$ and edge set $E(\Gamma)=\{e_1,e_2,\ldots,e_m\}$, the adjacency matrix  $A(\Gamma)=(a_{ij})$ is an $n\times n$ matrix with $a_{ij}=1$ if vertices $v_i$ and $v_j$ are adjacent, and $0$ otherwise. The incidence matrix $M(\Gamma)=(m_{ij})$ is an $n\times m$ matrix with $m_{ij}=1$ if vertex $v_i$ is incident to edge $e_j$, and $0$ otherwise. 
	 	
	 The degree of vertex $v$, written as $\text{deg}(v)$, is the number of edges incident to $v$. $N(v)$ is the set of adjacent vertices of $v$ in $V(\Gamma)$ and $N[v]:=N(v)\cup\{v\}$. We call $N(v)$ (resp. $N[v]$) the neighbor (resp. closed neighbor) of the vertex $v$. For a vertex subset $S\subset V(\Gamma)$, we denote $N(S):={\cup}_{v\in S}N(v)$ and $N[S]:={\cup}_{v\in S}N[v]$. The subgraph induced by $S$ is written as $\Gamma\langle S\rangle$, and the connected components of $\Gamma$ are its maximal connected subgraphs. 
	 
	 A trail is a sequence of distinct edges joining a sequence of vertices. A circuit is a non‑empty trail whose first and last vertices coincide. A cycle is a non‑empty trail in which only the first and last vertices are equal. The tree is an important component of graph, which is connected and contains no cycles. The null space of $\Gamma$, denoted by $\mathcal{N}(\Gamma)$  is the null space of its adjacency matrix $A(\Gamma)$.
	 
	 A set $\mathcal{O}\subset V(\Gamma)$ is an independent set if no two vertices in $\mathcal{O}$ are adjacent. The independence number $\alpha(\Gamma)$ is the cardinality of a maximum independent set in $\Gamma$. A matching $\tilde{M}$ in $\Gamma$ is a set of pairwise non-adjacent edges (i.e., no two edges in $\tilde{M}$ share a common vertex). The matching number $\nu(\Gamma)$ is the cardinality of a maximum matching. In addition, for any finite set $\mathcal{O}$, we denote by $|\mathcal{O}|$ the number of its elements.
	 
	Then we introduce the function spaces and operators on tree $\Gamma$. Define the Hilbert space on $\Gamma$
	\begin{flalign*}
		H:=L^2(\Gamma)=\prod_{j=1}^{m}L^2(0,1;dx_j)
	\end{flalign*} 
	endowed with the inner product 
	\begin{flalign*}
		\langle u,v\rangle:=\sum_{j=1}^{m}\int_0^1u_j(x_j)v_j(x_j)dx_j,\ u=\begin{bmatrix}
			u_1\\u_2\\\vdots\\u_m
		\end{bmatrix},\ v=\begin{bmatrix}
			v_1\\v_2\\\vdots\\v_m
		\end{bmatrix}\in H,
	\end{flalign*}
and the norm $\|u\|_H=\sqrt{\langle u,u\rangle}$. And define the space of continuous functions on graph:
	\begin{flalign*}
		\mathcal{C}(\Gamma):=\{u\in\left(\mathcal{C}([0,1];\mathbb R)\right)^m:u_j(v_i)=u_l(v_i),\ \forall\,j,l=1,\ldots,m\ \text{with }  e_j,e_l\in E_{v_i}, i=1,2,\ldots,n\}.
	\end{flalign*}
	Let $W^{k,p}(0,1),\, k,p\in\mathbb{N}^+$ be the usual Sobolev space on $(0,1)$. On $H$, define the $m\times m$ dimensional Laplacian operator 
	\begin{flalign}\label{laplacian}
		\Delta=\begin{bmatrix}
			\frac{d^2}{dx^2}      &  & 0      \\
			& \ddots &  \\
			0      &  & \frac{d^2}{dx^2}
		\end{bmatrix}_{m\times m}
	\end{flalign}
	with the homogeneous Neumann--Kirchhoff boundary condition \eqref{eq:kirchhoff}. The domain of the Laplacian operator is given by \begin{flalign}\label{domain}
		\mathcal{D}(\Delta)=\left\{h\in\left(W^{2,2}(0,1)\right)^m\cap \mathcal{C}(\Gamma):\Phi^+h'(0)-\Phi^-h'(1)=0\right\},
	\end{flalign} 
where the $n\times m$ dimensional matrices $\Phi^+:=(\phi_{ij}^+)$ and $\Phi^-:=(\phi_{ij}^-)$ are respectively defined by 
\begin{flalign*}
	\phi_{ij}^+:=\begin{cases}
		1,\ \text{if}\ e_j(0)=v_i,\\
		0,\ \text{otherwise},
	\end{cases}\ 
	\text{and}\ \phi_{ij}^-:=\begin{cases}
		1,\ \text{if}\ e_j(1)=v_i,\\
		0,\ \text{otherwise},
	\end{cases}
\end{flalign*}
for $i=1,2,\ldots,n$ and $j=1,2,\ldots,m$. Here we denote by $e_j(0)$ and $e_j(1)$ the $0$ endpoint and the $1$ endpoint of the edge $e_j$, respectively.
By convention, denote $\Phi:=\Phi^+-\Phi^-$. It is known that the operator $(\Delta,\mathcal{D}(\Delta))$ is self-adjoint and non-positive operator. Several useful properties on the spectrum, spectral expansion and semigroup for the operator $(\Delta,\mathcal{D}(\Delta))$ are given in the following lemma; see \cite[Theorem 2.1]{10.1007/BFb0072753}, \cite[Theorem 3.1.1]{berkolaiko2013introduction} and \cite[Chapter 3]{fijavz2007variational} for details.
	\begin{lemma}\label{operatorproperty}
		For the Laplacian operator \eqref{laplacian} with domain \eqref{domain}, the following properties hold.
		\begin{enumerate}[label=(\Roman*)]
			\item The
			spectrum $\sigma(\Delta)$ consists of isolated eigenvalues 
			and  $\sigma(\Delta)=\sigma_1\cup\sigma_2$,  where
			\begin{enumerate}[label=(\roman*)]
				\item $\sigma_1 = \{-k^2\pi^2 \;|\; k \in \mathbb{N}\}$, 
				with multiplicity $r_k=1$;
				
				\item\label{sigma2} $\sigma_2=\{\lambda_\mu\in(-\infty,0)\;|\;\mu=\cos(\sqrt{-\lambda_\mu})\in \sigma(\tilde{A}(\Gamma))\cap(-1,1)\}$, where $\tilde{A}(\Gamma)=D^{-1}(\Gamma)A(\Gamma)$ denotes the $n\times n$ row-normalized adjacency matrix of $\Gamma$ and $\sigma(\tilde{A}(\Gamma))$ denotes the eigenvalue set of $\tilde{A}(\Gamma)$. Here $D(\Gamma)$ is an $n\times n$ dimensional diagonal matrix defined by $D(\Gamma)=\mathrm{diag}\bigl(\mathrm{deg}(v_i)\bigr).$
			\end{enumerate}
		\item For the eigenfunction corresponding to the eigenvalue $\lambda\in\sigma(\Delta)$, we have the following expressions. Let $\mathbf{x}=[x_1,x_2,\ldots,x_m]^\top\in[0,1]^m$.
		
		(i) $\lambda=0\in\sigma_1:$ 
		\begin{flalign*}
			\phi^{1,0}(\mathbf{x})=\bigl[\phi^{1,0}_{1}(x_1),\phi^{1,0}_{2}(x_2),\ldots,\phi^{1,0}_{m}(x_m)\bigr]^\top,
		\end{flalign*}
	where  for $j=1,2,\ldots,m,\,\phi^{1,0}_{j}(x_j)=a_{j}x_j+b_{j},\,a_j,\,b_j$ are some real constants such that $\phi^{1,0}$ satisfies the continuity condition \eqref{eq:continuity}, the homogeneous Neumann--Kirchhoff condition  \eqref{eq:kirchhoff} and $\|\phi^{1,0}\|_H=1$.
	
		(ii) $\lambda=-k^2\pi^2\in\sigma_1,\, k\in\mathbb{N}^+:$
		\begin{flalign*}
			\phi^{1,k}(\mathbf{x})=\bigl[\phi^{1,k}_{1}(x_1),\phi^{1,k}_{2}(x_2),\ldots,\phi^{1,k}_{m}(x_m)\bigr]^\top,
		\end{flalign*}
		where for $j=1,2,\ldots,m,\,\phi^{1,k}_{j}(x_j)=c^k_{1,j}\cos(k\pi x_j)+c^k_{2,j}\sin(k\pi x_j),\,c^k_{1,j},\, c^k_{2,j}$ are some real constants such that $\phi^{1,k}$ satisfies the continuity condition \eqref{eq:continuity}, the homogeneous Neumann--Kirchhoff condition  \eqref{eq:kirchhoff} and $\|\phi^{1,k}\|_H=1$.
	
	(iii) Any nonzero $\mu\in\sigma(\tilde{A}(\Gamma))\cap(-1,1)$ generates a set of eigenvalues $\{\lambda_\mu^l\in(-\infty,0)\;|\;\lambda_\mu^l=-(\text{arccos}\mu+2l\pi)^2,l\in\mathbb{Z}\}\subset\sigma_2$. In particular if $0\in\sigma(\tilde{A}(\Gamma))\cap(-1,1)$, the corresponding set takes the form $\{\lambda_0^l\in(-\infty,0)\;|\;\lambda_0^l=-(\frac{1}{2}+l)^2\pi^2,l\in\mathbb{Z}\}\subset\sigma_2$. Then for $\lambda=\lambda_\mu^l\in\sigma_2,\, l\in\mathbb{Z}$,
	\begin{align*}
		\phi^{2,l}(\mathbf{x})=\bigl[\phi^{2,l}_{1}(x_1),\phi^{2,l}_{2}(x_2),\ldots,\phi^{2,l}_{m}(x_j)\bigr]^\top,
	\end{align*}
where for $j=1,2,\ldots,m$,
\begin{flalign*}
\phi^{2,l}_{j}(x_j)=\tfrac{1}{\sin\sqrt{-\lambda}}\left(\phi^{2,l}_{j}(0)\sin(\sqrt{-\lambda}(1-x_j))+\phi^{2,l}_{j}(1)\sin(\sqrt{-\lambda}x_j)\right).
\end{flalign*}
Here $\phi^{2,l}_{j}(0)$ and $\phi^{2,l}_{j}(1)$ denote the endpoint values of the function on each edge $e_j$ such that $\phi^{2,l}$ satisfies the continuity condition \eqref{eq:continuity}, the homogeneous Neumann--Kirchhoff condition  \eqref{eq:kirchhoff} and $\|\phi^{2,l}\|_H=1$.
	\item The semigroup $(P_t)_{t\ge0}$ generated by $\Delta$ forms a strongly continuous contraction semigroup on $H$. The space $H$ admits a complete orthonormal basis $\{\phi_\lambda: \lambda\in\sigma_1\cup\sigma_2\}$ which is obtained by orthonormalizing the set $\{\phi^{1,k}\cup\phi^{2,l}:k\in\mathbb{N},l\in\mathbb{Z}\}$. Using this basis, we have the expansion:
\begin{flalign}\label{heatsemigroup}
	P_t u = \sum_{\lambda\in\sigma_1\cup\sigma_2} \langle u,\phi_\lambda\rangle e^{\lambda t} \phi_\lambda,\, t\ge0.
\end{flalign}
	In particular, when $t=0$, we have $u = \sum_{\lambda\in\sigma_1\cup\sigma_2} \langle u,\phi_\lambda\rangle \phi_\lambda$.
\end{enumerate}
\end{lemma}

	Arranging all eigenvalues of the operator $A := -\Delta$ in ascending order  yields a sequence $\{\mu_k\}_{k=0}^\infty$ satisfying \cite{ariturk2016eigenvalue}
	\begin{flalign}\label{eigenvalue}
		0=\mu_0<\mu_1\leq\mu_2\leq\cdots\leq\mu_k\le\cdots\ \text{with}\ \mu_k\sim k^2. 
	\end{flalign}
And the eigenpair set of $(\Delta,\mathcal{D}(\Delta))$ is then expressed as $\{(-\mu_k,\phi_k)\}_{k=0}^\infty$. In addition, for any $\alpha>0$, the space $\mathbb{H}^\alpha$ consists of all expansions
$u=\sum_{k=1}^{\infty} u_k \phi_k,\, u_k\in\mathbb{R}$, with
$
\|u\|_{\mathbb{H}^\alpha}^2
:=\sum_{k=1}^{\infty} \mu_k^{\alpha} u_k^2<\infty,
$
equipped with the inner product $(\cdot,\cdot)_\alpha := \langle A^{\frac{\alpha}{2}}\cdot, A^{\frac{\alpha}{2}}\cdot\rangle$.  It is a Hilbert space and coincides with $\mathcal D(A^{\frac{\alpha}{2}})$ on the orthogonal complement of  $(\ker(A))^\perp=\text{span}\{\phi^{1,0}\}$, where $\text{ker}(A)$ denotes the kernel space of the operator $A$.
 
 
If we set $X(t,X_0)(\mathbf{x})=\bigl[u_1(t,x_1),\ldots,u_m(t,x_m)\bigl]^\top,\mathbf{x}=[x_1,x_2,\dots,x_m]^\top\in[0,1]^m$, then we can rewrite \eqref{eq:parabolic} into an infinite-dimensional stochastic evolution equation:
	\begin{equation}\label{SEE}
		\begin{cases}
			 dX(t,X_0)=-A X(t,X_0)dt+B(X(t,X_0))dt+QdW(t),\,t>0,\\
			X(0,X_0)=X_0\in H,
		\end{cases}
	\end{equation}
where $X_0=\bigl[u^0_1,u^0_2,\dots,u^0_m \bigr]^\top$ with $u_j^0\in L^2(0,1),\, j=1,2,\dots,m$, $Q=\text{diag}\left(Q_j\right)_{m\times m}$ with $Q_j$ defined in \eqref{Q}, $W(t)=\bigl[W_1(t),W_2(t),\ldots,W_m(t)\bigl]^\top$ is an $H$-valued cylindrical Wiener process with $W_j(t)$ being independent cylindrical Wiener process on $L^2(0,1)$ and $B:H\rightarrow H$ is a Nemytskii operator satisfying the following assumption.
	
    \begin{assumption}\label{b1}
		The nonlinear term $B$ is defined by: for $h\in H$
		\begin{flalign*}
			B(h)(\mathbf{x}):=\bigl[b_1(h_1(x_1)),b_2(h_2(x_2)),\ldots,b_m(h_m(x_m))\bigl]^\top,
		\end{flalign*}
		where for some constant $K>0$, $|b_j'(\zeta)| \leq K$ for all $\zeta\in \mathbb{R}$, $j = 1,\dots,m$.
	\end{assumption}
Based on the contraction mapping approach (see e.g. \cite{lord2014introduction}), one can show the well-posedness of the mild solution to the semi‑linear SPDE \eqref{SEE} on graph, stated below.
\begin{prop}
	Let Assumption \ref{b1} hold and $X_0\in L^p(\Omega,H)$ for some $p\ge1$. Then for each $T>0$, \eqref{SEE} admits a unique mild solution $X\in L^p\left(\Omega,\mathcal{C}([0,T],H)\right)$ given by
	\begin{flalign}\label{mildsolution}
		X(t,X_0)=P_tX_0+\int_0^tP_{t-s}B(X(s,X_0))ds+\int_0^tP_{t-s}QdW(s),\ t\in[0,T].
	\end{flalign}
\end{prop}

The linear equation corresponding to \eqref{SEE} reads as
\begin{flalign}\label{linear SEE}
		dv(t,X_0)=-A v(t,X_0)dt+QdW(t).
\end{flalign}
Let $\mathcal{B}_b(H)$ and $\mathcal{C}_b(H)$ be, respectively, the spaces of bounded measurable functions and bounded continuous functions on $H$. We now define two Markov semigroups
\begin{flalign}
	\mathcal{S}_t \psi(X_0) &:= \mathbb{E}\bigl( \psi(X(t,X_0)) \bigr), \, t\ge0,\, \psi \in \mathcal{B}_b(H), \label{nonlinear semigroup} 
	\\
	\mathcal{R}_t \psi(X_0) &:= \mathbb{E}\bigl( \psi(v(t,X_0)) \bigl), \, t\ge0,\, \psi \in \mathcal{B}_b(H), \label{linear semigroup}
\end{flalign}
which are generated by nonlinear equation \eqref{SEE}, and linear equation \eqref{linear SEE}, respectively. The definitions of strong Feller property, irreducibility, invariant measure and exponential ergodicity of Markov semigroup are stated below; see e.g. \cite{da1996ergodicity,friesen2023}.

\begin{dfn}\label{def}
Let $(\mathcal{A}_t)_{t\ge0}$ be either $(\mathcal{S}_t)_{t\ge0}$ or $(\mathcal{R}_t)_{t\ge0}$.
	\begin{itemize}
		\item[(i)] $(\mathcal{A}_t)_{t\ge0}$ is called strong Feller at time $t_0>0$ if $
		\mathcal{A}_{t_0} \psi \in \mathcal{C}_b(H)$ for every $\psi \in \mathcal{B}_b(H).$
		\item[(ii)] $(\mathcal{A}_t)_{t\ge0}$ is called irreducible at time $t_0>0$ if $
		\mathcal{A}_{t_0}(X_0,O)= \mathcal{A}_{t_0} \mathbf{1}_O(X_0) > 0$ holds for every nonempty open set $O \subset H$ and every $X_0 \in H$, where $\mathcal{A}_{t_0}(X_0,O)$ denotes the transition probability from initial value $X_0\in H$ to the set $O\subset H$ at time $t_0$ and $\mathbf{1}_O(\cdot)$ is the indicator function of set $O$.
		\item[(iii)] A probability measure $\mu \in \mathcal{P}(H)$ is an invariant measure for $(\mathcal{A}_t)_{t\ge0}$ if
		\[
		\int_H (\mathcal{A}_t \psi)(u) \, \mu(du) = \int_H \psi(u) \, \mu(du),
		\ \forall\, \psi \in \mathcal{B}_b(H),\, t \ge 0,
		\]where $\mathcal{P}(H)$ is the set of probability measures on $H$.
		\item[(iv)] Let $(\mathcal{A}_t)_{t \ge 0}$ admit a unique invariant measure $\mu$, and let $d$ be a metric on the space of probability measures $\mathcal{P}(H)$. If there exist a constant $\alpha > 0$ and a function $\mathcal{K} : H \to (0,+\infty)$ such that for every $X_0 \in H$ and all $t \ge 0$
		\[
		d\bigl(\mathcal{A}_t(X_0, \cdot), \mu \bigr) \le\mathcal{K}(X_0)e^{-\alpha t},
		\]
		then the invariant measure $\mu$ is called exponentially ergodic with respect to the metric $d$.
	\end{itemize}
\end{dfn}
\subsection{Short introduction to decomposition of trees}\label{decomposition}
 The purpose of this subsection is to decompose a complex tree in terms of its null space, which will be used to analyze the effect on strong Feller property and irreducibility of graph structure in Sections~\ref{main} and \ref{proof}. We begin with some basic definitions. 
\begin{dfn}
	\cite[Definitions 2.1 and 4.2]{jaume2018null} Given a vector $Y\in\mathbb{R}^n$, the \textbf{support} of $Y$ is \begin{flalign*}
		\mathbf{Supp}(Y):=\{v\in V(\Gamma):Y(v)\ne0\}.
	\end{flalign*}
	Let $S$ be a subset of $\mathbb{R}^n$. Then the support of $S$ is
	\begin{flalign*}
		\mathbf{Supp}(S):=\underset{Y\in S}{\cup}\mathbf{Supp}(Y).
	\end{flalign*}
	The support of a tree $\Gamma$, denoted by  $\mathbf{Supp}(\Gamma)$ is the set $\mathbf{Supp}(\mathcal{N}(\Gamma))$.
	The core of $\Gamma$, denoted by $\mathbf{Core}(\Gamma)$ is the set
	\begin{flalign*}
		\mathbf{Core}(\Gamma):=N(\mathbf{Supp}(\Gamma)).
	\end{flalign*}
Recall that $\mathcal{N}(\Gamma)$ is the null space of $\Gamma$ and $N(\mathbf{Supp}(\Gamma))$ denotes the neighbor of the set $\mathbf{Supp}(\Gamma)$.
\end{dfn}
Next, we present the decomposition of arbitrary tree.
\begin{dfn}\label{decom}
The \textbf{S-Set} of tree $\Gamma$, denoted by $\mathcal{F}_S(\Gamma)$, is defined by
	\begin{flalign*}
		\mathcal{F}_S(\Gamma):=\left\{S:\ S\ \text{is a connected component of } \Gamma\langle N[\mathbf{Supp}(\Gamma)]\rangle \right\}.
	\end{flalign*}
	The \textbf{N-Set} of tree $\Gamma$, denoted by $\mathcal{F}_N(\Gamma)$,  is defined to be the set of connected
	components of the remaining graph:
	\begin{flalign*}
		\mathcal{F}_N(\Gamma):=\left\{N:\ N\ is\ a\ connected\ component\ of\ \Gamma\backslash\mathcal{F}_S(\Gamma) \right\}.
	\end{flalign*}
	The pair of sets $\bigl(\mathcal{F}_S(\Gamma),\mathcal{F}_N(\Gamma)\bigl)$ is called the \textbf{null decomposition} of $\Gamma$. Every $S\in\mathcal{F}_S(\Gamma)$ is called an $\textbf{S-tree}$ and every $N\in\mathcal{F}_N(\Gamma)$ is called an $\textbf{N-tree}$. 
	
	For an $S$-tree $S$, the  \textbf{A-set}, denoted by $\mathcal{F}_A(S)$, is the set of all connected components that remain after taking away all the edges between core vertices. In particular, if $\mathcal{F}_A(S)=\{S\}$, i.e. $S$ has no edges between core vertices, then $S$ is called an \textbf{S-atom}. 
\end{dfn}
The $A$-set of an $S$-tree $S$ is a set of $S$-atoms, which are the minimal element in the analysis of effect on strong Feller property of graph structure.
Some useful properties of the decomposition of a tree $\Gamma$ are collected in the following lemma (see \cite[Lemma 4.3]{jaume2018null}).
\begin{lemma}\label{suppcore}
	For a tree $\Gamma$, the sets $\mathbf{Supp}(\Gamma)$, $\mathbf{Core}(\Gamma)$ and $V(\mathcal{F}_N(\Gamma))$ form a weak partition of $V(\Gamma)$, i.e., \begin{align*}
		V(\Gamma)=\mathbf{Supp}(\Gamma)\cup\mathbf{Core}(\Gamma)\cup V(\mathcal{F}_N(\Gamma))
	\end{align*}and some of the sets on the right hand side are allowed to be empty. Here $\mathbf{Supp}(\Gamma)$ and $\mathbf{Core}(\Gamma)$ can further be partitioned as \begin{flalign*}
		\mathbf{Supp}(\Gamma)=\underset{S\in\mathcal{F}_S(\Gamma)}{\cup}\mathbf{Supp}(S),\quad \mathbf{Core}(\Gamma)=\underset{S\in\mathcal{F}_S(\Gamma)}{\cup}\mathbf{Core}(S).
	\end{flalign*}
\end{lemma} 
We denote by $\textbf{ConnE}(\Gamma)$ the set of edges that are neither in $\mathcal{F}_S(\Gamma)$ nor in $\mathcal{F}_N(\Gamma)$, i.e.  $\textbf{ConnE}(\Gamma):=E(\Gamma)\backslash\left(E(\mathcal{F}_S(\Gamma))\cup E(\mathcal{F}_N(\Gamma))\right).$ The set of edges in $E(S)$ which do not belong to any $S$-atom of $\mathcal{F}_A(S)$ is denoted by $\textbf{BondE}(S):=E(S)\backslash E(\mathcal{F}_A(S))$.
\begin{remark}
	Note that for every $e\in\textbf{ConnE}(\Gamma)$, the endpoints belong to $ V(\mathcal{F}_N(\Gamma))$ and $ V(\mathcal{F}_S(\Gamma))$, respectively. Without loss of generality, we suppose that $e(0)\in V(\mathcal{F}_N(\Gamma))$ and $e(1)\in V(\mathcal{F}_S(\Gamma))$. Then we claim that  $e(0)\notin$\textbf{Core}$(\Gamma)\cup\textbf{Supp}(\Gamma)$ and $e(1)\in\textbf{Core}(\Gamma)$. The former holds clearly due to Lemma \ref{suppcore}. For the latter, we suppose by contradiction that $e(1)\in\textbf{Supp}(\Gamma)$. Then by the definition of $\mathcal{F}_S(\Gamma)$ and the fact that $e(0)\in N(e(1))$, $e(0)$ must be a vertex of an $S$-tree, which leads to a contradiction. Thus the claim is proved.
	
	Moreover, for any $e\in\textbf{BondE}(S)$, both endpoints $e(0),e(1)\in\textbf{Core}(S)$.
\end{remark}
To end this part, we present an example to explain these definitions on trees.
\begin{example}\label{Treegamma}
	Consider the tree $\Gamma$ in Fig.~\ref{tree gamma} with vertex set $V(\Gamma)=\{1,2,3,4,5,6\}$ and edge set $E(\Gamma)=\{e_{12},e_{13},e_{14},e_{15},e_{56}\}$.
    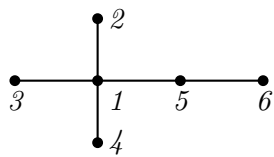
\begin{figure}[htbp]
		\centering
		\begin{tikzpicture}[scale=1.1,line width=0.8pt]
			
			\filldraw (0,0) circle (1.5pt) node[below right]{1};
			
			\draw (0,0) -- (-1,0) node[below] {3};
			\filldraw (-1,0) circle (1.5pt);
			
			\draw (0,0) -- (0,0.75) node[right] {2};
			\filldraw (0,0.75) circle  (1.5pt);
			
			\draw (0,0) -- (0,-0.75) node[right] {4};
			\filldraw (0,-0.75) circle  (1.5pt);
			
			\draw (0,0)--(1,0) node[below]{5};
			\filldraw (1,0) circle  (1.5pt);
			
			\draw (1,0)--(2,0) node[below]{6};
			\filldraw (2,0) circle  (1.5pt);

		\end{tikzpicture}
		\caption{An illustrative example}
		\label{tree gamma}
	\end{figure}
	The adjacency matrix $A(\Gamma)\in\mathbb{R}^{6\times 6}$ and incidence matrix $M(\Gamma)\in\mathbb{R}^{6\times 5}$ are 
	\[
	A(\Gamma) = 
	\begin{bmatrix}
		0 & 1 & 1 & 1 & 1 & 0 \\
		1 & 0 & 0 & 0 & 0 & 0 \\
		1 & 0 & 0 & 0 & 0 & 0 \\
		1 & 0 & 0 & 0 & 0 & 0 \\
		1 & 0 & 0 & 0 & 0 & 1 \\
		0 & 0 & 0 & 0 & 1 & 0
	\end{bmatrix}, \quad
	M(\Gamma) = 
	\begin{bmatrix}
		1 & 1 & 1 & 1 & 0 \\
		1 & 0 & 0 & 0 & 0 \\
		0 & 1 & 0 & 0 & 0 \\
		0 & 0 & 1 & 0 & 0 \\
		0 & 0 & 0 & 1 & 1 \\
		0 & 0 & 0 & 0 & 1
	\end{bmatrix},
	\]respectively, and $A(\Gamma)$ has an eigenvalue $0$ with multiplicity $2$. The null space $\mathcal{N}(\Gamma)$ is generated by the eigenvectors of $0$ eigenvalue
	\begin{flalign*}
		\left\{[0,1,-1,0,0,0]^\top,[0,1,0,-1,0,0]^\top\right\}.
	\end{flalign*}
	Thus \textbf{Supp}$(\Gamma)=\{2,3,4\}$ and \textbf{Core}$(\Gamma)=\{1\}$.
	According to Definition \ref{decom}, we have \[\mathcal{F}_S(\Gamma)=\{S:S=\Gamma\langle N[{2,3,4}]\rangle\},\; \mathcal{F}_N(\Gamma)=\{N:N=\Gamma\langle N[5,6]\rangle\},\;  \textbf{ConnE}(\Gamma)=\{e_{15}\}.\]
	In this case, the S-tree $S=\Gamma\langle N[2,3,4]\rangle$ is also an S-atom.
\end{example}
	\section{Main results}\label{main}
 In this section, we present the main results on the effect of graph structure on the strong Feller property and irreducibility for SPDE on a finite tree. Then  we show the existence and exponential ergodicity of a unique invariant measure for the considered SPDE with a dissipative nonlinearity.
	
To analyze the strong Feller property of the Markov semigroups $(\mathcal{S}_t)_{t\ge0}$ and $(\mathcal{R}_t)_{t\ge0}$ (see \eqref{nonlinear semigroup} and \eqref{linear semigroup}), we introduce the following conditions on the eigenvector of the nonzero eigenvalue of the matrix $\tilde{A}(\Gamma)$ and the nonlinearity $B$. Recall that $\sigma(\tilde{A}(\Gamma))$ is given in Lemma \ref{operatorproperty}. The matrix $\tilde{A}(\Gamma)$ is uniquely determined by the ordering of the vertices, and hence there exists a one-to-one correspondence between the entries of the eigenvectors of $\tilde{A}(\Gamma)$ and the vertices.
	\begin{assumption}\label{nonzero eigenvalue} \begin{enumerate}[label=(\roman*)]
			\item For every nonzero eigenvalue $\lambda \in \sigma(\tilde{A}(\Gamma))\cap(-1,1)$, the vertices corresponding to the zero entries of  eigenvector of $\lambda$ are non-adjacent;
			\item The nonlinearity $B:H\rightarrow H$ satisfies $B(H) \subseteq \text{Im}(Q)$, where $\text{Im}(Q)$ denotes the range of operator $Q$.
		\end{enumerate}
	\end{assumption}
The main results concerning quantifying the effect of graph structure on strong Feller property and irreducibility of Markov semigroups $(\mathcal{S}_t)_{t\ge0}$ and $(\mathcal{R}_t)_{t\ge0}$ are stated as follows.
\begin{theorem}\label{main theorem}
Let Assumptions \ref{b1} and \ref{nonzero eigenvalue}(i) hold. 
 The sharp upper bound on the number of noise‑free edges while maintaining the strong Feller property and irreducibility of semigroup $(\mathcal{R}_t)_{t\ge0}$ for SPDE \eqref{linear SEE} on tree $\Gamma$ is
		$$
		|\mathcal{Z}(\Gamma)|\le \min\{m -|\textbf{Supp}(\Gamma)| + |\textbf{Core}(\Gamma)|,m-1\}.
		$$
		
		Furthermore, let Assumption \ref{nonzero eigenvalue}(ii) hold, then the semigroup $(\mathcal{S}_t)_{t\ge0}$ for SPDE \eqref{SEE} fulfills the same upper bound to that of $(\mathcal{R}_t)_{t\ge0}$.
\end{theorem}

The results for the strong Feller and irreducibility properties are established for trees, as the underlying proofs rely on a null decomposition specific to tree structures (see Section \ref{proof} for the proof). The characterization of these two properties on general graphs presents a formidable challenge. Nevertheless, Assumption~\ref{nonzero eigenvalue}(i) holds for a broad class of trees, including chain graph and star graph. As a corollary, we can obtain the strong Feller property and irreducibility for SPDE on these graphs.

	\begin{figure}[htbp]
		\centering
		\begin{minipage}[c]{0.45\textwidth}
			\centering
			\vspace{1cm}
			\begin{tikzpicture}[scale=1.2,line width=0.8pt]
				\draw(0,0)--(1,0);
                \draw(1,0)--(2,0);
                \filldraw(0,0) circle(1.5pt);
                \filldraw(1,0) circle(1.5pt);
				\draw[dashed] (2,0) -- (3,0);
				\draw(3,0)--(4,0);
                \filldraw(3,0) circle(1.5pt);
                \filldraw(4,0) circle(1.5pt);
				\node[below] at (0,0) {$1$};
				\node[below] at (1,0) {$2$};
				\node[below] at (3,0) {$m$};
				\node[below] at (4,0) {$m+1$};
			\end{tikzpicture}
			\caption{Chain graph $L$ with $m$ edges} 
			\label{chain_graph}
		\end{minipage}
	    \hspace{0.2cm}
		\begin{minipage}[b]{0.45\textwidth}
			\centering
			\[
			\tilde{A}(L) = 
			\begin{bmatrix}
			0           & 1           & 0           & \cdots      & \cdots & 0 \\
			\frac{1}{2} & 0           & \frac{1}{2} & 0           & \cdots & 0 \\
			0           & \frac{1}{2} & 0           & \frac{1}{2} & \ddots &\vdots \\
			\vdots      & 0           & \ddots      & \ddots      & \ddots & 0 \\
			0           & \vdots      & \ddots      & \frac{1}{2} & 0      & \frac{1}{2} \\
			0           & 0           & \cdots      & 0           & 1      & 0
			\end{bmatrix}
			\]
		\end{minipage}
	\end{figure}
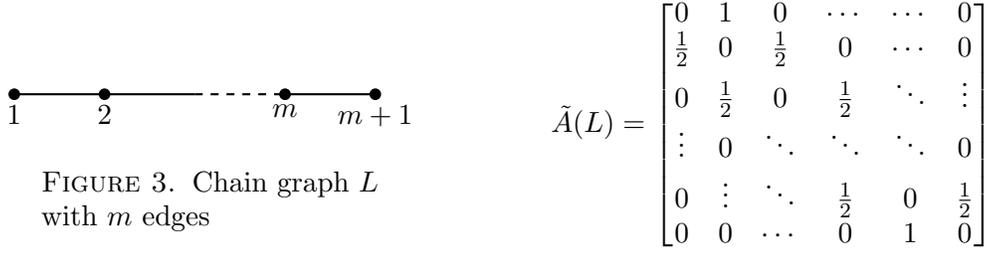

\begin{corollary}\label{bound_chain}
	 Let $L$ be the chain graph with $m$ edges, as shown in Fig.~\ref{chain_graph}. For the SPDE \eqref{SEE} on $L$, under Assumptions \ref{b1} and \ref{nonzero eigenvalue}(ii), Theorem \ref{main theorem} holds with $|\mathcal{Z}(L)| \le m-1$.
	\begin{proof}
		It suffices to verify that the chain graph satisfies Assumption \ref{nonzero eigenvalue}(i). One can calculate that the eigenvalues and eigenvectors of the normalized adjacency matrix $\tilde{A}(L)\in\mathbb{R}^{(m+1)\times (m+1)}$ are:
		\begin{flalign*}
			\lambda_k&=\cos\left(\tfrac{k\pi}{m}\right)=:\cos(\theta_k),\, k=0,1,\dots,m,\\
			U_k&=\bigl[1,\cos(\theta_k),\cos(2\theta_k),\dots,\cos(m\theta_k)\bigl]^\top.
		\end{flalign*}
	If $m$ is odd, $\tilde{A}(L)$ has no $0$ eigenvalue. For each $k=0,1,\ldots,m$, every component of the corresponding eigenvector $U_k$ is nonzero, thus all eigenvectors satisfy Assumption~\ref{nonzero eigenvalue}(i). The semigroups are strong Feller and irreducible, provided at least one edge is noisy, i.e. $|\mathcal{Z}(L)|\le m-1$. If $m$ is even, $\tilde{A}(L)$ admits a single $0$ eigenvalue (when $k=\frac{m}{2}$), and corresponding eigenvector has the form
	\begin{flalign*}
		U_{\frac{m}{2}}=\bigl[1,0,-1,0,\ldots,(-1)^\frac{m}{2}\bigr]^\top,
	\end{flalign*} which also satisfy Assumption~\ref{nonzero eigenvalue}(i).
	Hence applying Theorem \ref{main theorem} yields $|\mathcal{Z}(L)|\le m -|\textbf{Supp}(L)| + |\textbf{Core}(L)|=m-(\frac{m}{2}+1)+\frac{m}{2}=m-1$.
	\end{proof}
\end{corollary}

	\begin{figure}[htbp]
		\centering
		\begin{minipage}[c]{0.45\textwidth}
			\centering
			\vspace{0.6cm}
			\begin{tikzpicture}[scale=1.2,line width=0.8pt]
			\filldraw (0,0) circle (1.5pt) node[below=2pt]{c};
			
			
			\draw (0,0) -- (0,1.1) node[above] {$m$};
			\filldraw (0,1.1) circle (1.5pt);
			
			\draw (0,0) -- (-1,0.8) node[above] {1};
			\filldraw (-1,0.8) circle (1.5pt);

			\draw (0,0)--(-1,0.2) node[left]{2};
			\filldraw (-1,0.2) circle (1.5pt);
			
			\draw (0,0)--(-1,-0.7) node[left]{3};
			\filldraw (-1,-0.7) circle (1.5pt);
			
			\draw (0,0)--(-0.2,-1) node[left]{4};
			\filldraw (-0.2,-1) circle (1.5pt);
			
			\draw (0,0)--(1,-0.8) node[right]{5};
			\filldraw (1,-0.8) circle (1.5pt);
			
			\draw[dashed] (0,0) -- (1.5,0) ;
			\draw[dashed] (0,0) -- (0.9,0.9);
			\end{tikzpicture}
	     	\caption{Star graph $R$ with $m$ edges}
	     	\label{star_graph}
		\end{minipage}
		\hspace{0.1cm}
		\begin{minipage}[b]{0.45\textwidth}
			\centering
			\[
			\tilde{A}(R) = 
			\begin{bmatrix}
				0      & \frac1m       & \frac1m       & \cdots &  \frac1m \\
				1      & 0       & 0       & \cdots &  0 \\
				1      & 0       & \ddots  & \ddots & \vdots    \\
				\vdots & \vdots  & \ddots  & \ddots & 0 \\
				1      & 0       & \cdots  & 0      & 0
			\end{bmatrix}
			\]
		\end{minipage}
	\end{figure}
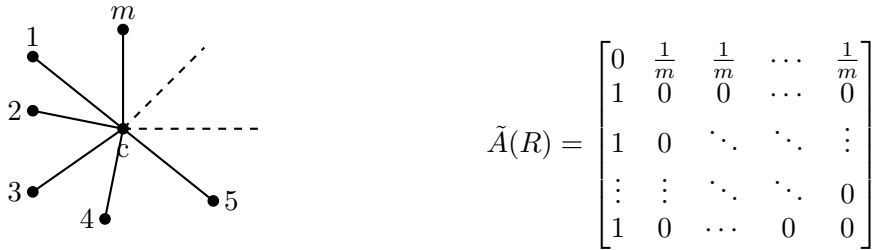

\begin{corollary}\label{bound_star}
	Let $R$ be a star graph with $m$ edges, one central vertex $\{c\}$ and $m$ boundary vertices $\{1,2,\ldots,m\}$, as shown in Fig.~\ref{star_graph}. For the SPDE \eqref{SEE} on $R$, under Assumptions \ref{b1} and \ref{nonzero eigenvalue}(ii), Theorem \ref{main theorem} holds with $|\mathcal{Z}(R)| \le 1$.
\end{corollary}
\begin{proof}
For normalized adjacency matrix $\tilde{A}(R)\in\mathbb{R}^{(m+1)\times(m+1)}$, we can calculate that the eigenvalues are $\{-1,1,0\}$, where $0$ eigenvalue has $m-1$ multiplicity. Let $e_j(1)=c$ for all $j=1,2,\ldots,m$. For each
	$l\in\mathbb{Z}$, $\mu_l=(\frac{1}{2}+l)^2\pi^2$ and there are $m-1$ linearly independent eigenfunctions of the form
	\begin{align*}
		\Psi^{1,l}(\mathbf{x})&=\bigl[\cos(\sqrt{\mu_l}x_1),-\cos(\sqrt{\mu_l}x_2),0,0,\dots,0\bigl]^\top,\\
		\Psi^{2,l}(\mathbf{x})&=\bigl[\cos(\sqrt{\mu_l}x_1),0,-\cos(\sqrt{\mu_l}x_3),0,\dots,0\bigl]^\top,\\
		&\vdots\\
		\Psi^{m-1,l}(\mathbf{x})&=\bigl[\cos(\sqrt{\mu_l}x_1),0,\dots,0,-\cos(\sqrt{\mu_l}x_m)\bigl]^\top.
	\end{align*}
Since  $\sigma(\tilde{A}(R))\cap(-1,1)=\{0\}$, Assumption \ref{nonzero eigenvalue}(i) holds naturally. To obtain the upper bound for $|\mathcal{Z}(R)|$, we notice that a star graph is an $S$‑atom satisfying $\mathbf{Supp}(R)=\{1,\dots,m\}$ and $\mathbf{Core}(R)=\{c\}$, which yields $|\mathcal{Z}(R)| \le 1$. 
\end{proof}
\begin{remark}\label{vertexnoise}
	 The above results illustrate that, for a chain graph, adding noise to a single edge
	is sufficient to guarantee the strong Feller property and irreducibility. In contrast, for a star
	graph, noise must be added to all edges except one in order to maintain these two properties.
	
	Our results for star graphs are consistent with the vertex-noise setting examined in  \cite[Example 4.8]{fkirine2025strong}. By an appropriate transformation induced by an admissible control operator, the linear SPDE driven by $\mathbb{R}^n$-valued vertex-noise  can be reformulated in an abstract framework, where the driving noise takes values in a negative Sobolev space (see \cite[Eq. (4.14)]{fkirine2025strong}).
\end{remark}

 It is well-known that a sufficient condition for the uniqueness of an invariant measure is that the semigroup possesses both the strong Feller property and irreducibility \cite[Proposition 11.13 and Theorem 11.14]{da2014stochastic}. Below we present the existence, uniqueness and exponential ergodicity of an invariant measure for the equation \eqref{SEE}.
\begin{theorem}\label{invariantmeasureth}
	Assume that the nonlinear mapping $B$ satisfies $K < \mu_1$, where $K$ is given in Assumption \ref{b1} and $\mu_1$ is presented in \eqref{eigenvalue}. Then there exists a unique invariant measure for $(\mathcal{S}_t)_{t\ge0}$, where $(\mathcal{S}_t)_{t\ge0}$ is the Markov semigroup of \eqref{SEE}. Moreover, this invariant measure is exponentially ergodic.
\end{theorem}
\begin{remark}
		\begin{enumerate}
\item[(i)]
	For stochastic systems in infinite-dimensional spaces, explicit expressions for invariant measures are generally difficult to obtain. However, in certain special cases, such expressions are available. For example, in the gradient case $B(u)=\nabla F(u)$ with some dissipative function $F:(\ker(A))^\perp\to\mathbb{R}$, the invariant measure admits the explicit representation of Gibbs type:
$	\pi(du)=\frac{1}{Z}e^{2F(u)}\pi_0(du)$,
where Gaussian measure $\pi_0=\mathscr{N}(0,\frac{1}{2}(A|_{(\ker(A))^\perp})^{-1})$ and  $Z=\int_{(\ker(A))^\perp}e^{2F(u)}\pi_0(du)$.

		\item[(ii)]  It is known that the strong Feller property may fail for SPDEs driven by a degenerate noise on Euclidean domains, which motivates the introduction of the asymptotically strong Feller property \cite{hairer2006ergodicity}. As a weaker alternative to the strong Feller property, this notion, when combined with suitable irreducibility, still yields uniqueness of invariant measure. Thus it is worth investigating whether SPDEs on graphs driven by such degenerate edge-wise noise still satisfy the asymptotic strong Feller property.
		\item[(iii)] We remark that for \eqref{eq:parabolic} with a superlinear dissipative nonlinearity, such as the stochastic Allen--Cahn equation on graph (i.e. $b_j(x)=x-x^3,\,x\in\mathbb{R}$), numerical experiments shown in Appendix \ref{nonlip} suggest that the corresponding Markov semigroup still possesses the strong Feller property, irreducibility, and the exponential ergodicity of a unique invariant measure. Nevertheless, standard tools like the Girsanov theorem are not directly applicable in this superlinear growth setting. A theoretical analysis of these cases will be reported in the future work.
	\end{enumerate}
\end{remark}
\section{Numerical verification}\label{numerical}
In this section, we provide numerical experiments on some concrete trees to verify our theoretical findings, where the solution is discretized by a full discretization whose spatial direction is the spectral Galerkin method and temporal direction is the accelerated exponential Euler method. 
\subsection{Chain graph}
We consider chain graph $L$ in Fig.~\ref{chain_graph} with $m=4$. Then $\sigma(\tilde{A}(L))\cap(-1,1)=\{\cos(\frac{\pi}{4}),\cos(\frac{\pi}{2}),\cos(\frac{3\pi}{4})\}=\{\frac{\sqrt{2}}{2},0,-\frac{\sqrt{2}}{2}\}$ and the corresponding eigenfunctions are denoted by $\{\{\Psi^{1,l}\},\{\Psi^{2,l}\},\{\Psi^{3,l}\}:l=1,2,\ldots,N\}$. Figs.~\ref{SF_chain} and \ref{IR_chain_nonlinear} provide numerical verification of Corollary  \ref{bound_chain}. The parameters are set as: dimension for spatial Galerkin method $N=2^6$, the time stepsize $\tau=2^{-5}$, the terminal time $T=2^{-1}$ and the number of trajectory $M_{traj}=500$.

\textbf{Verification of strong Feller property}. We set the  nonlinearity $b_j(x_j)=Q_j\sin(x_j),\, j=1,2,3,4$, with $Q_j$ defined in \eqref{Q}, which satisfies Assumption \ref{nonzero eigenvalue}(ii). We take $X_0=0$, $\epsilon \in \{ 10^{-\tfrac{4k}{7}} : k = 0,1,\dots,7\}$, $\mathcal{E}=\epsilon\sum_{l=1}^{N}\Psi^{1,l}$ and sign function $\phi(h)=\text{sgn}(\langle h,\Psi^{1,1}\rangle)\in\mathcal{B}_b(H)$, where $\Psi^{1,1}$ denotes the first eigenfunction of family $\{\Psi^{1,l}\}$. We use the Monte Carlo method to compute the difference
\begin{flalign}\label{nonlineardiffernece}
|\mathcal{S}_T\phi(X_0+\mathcal{E})-\mathcal{S}_T\phi(X_0)\big|&\approx\frac{1}{M_{traj}}\sum_{m=1}^{M_{traj}}|\phi\bigl(X^{\tau,N}(T,X_0+\mathcal{E})\bigl)\big|,
\end{flalign} where $X^{\tau,N}(T,X_0+\mathcal{E})$ denotes the numerical solution to nonlinear equation \eqref{SEE} at time $T$ with initial value $X_0+\mathcal{E}$, respectively. In Fig.~\ref{SF_chain}, all lines except the blue one decay to $0$ as $\epsilon\to0$. This indicates that $\mathcal{S}_T$ is strong Feller, provided noise acts on more than one edge, which is consistent with Corollary \ref{bound_chain}.
\begin{figure}[htbp]
	\centering
		\includegraphics[width=0.4\textwidth]{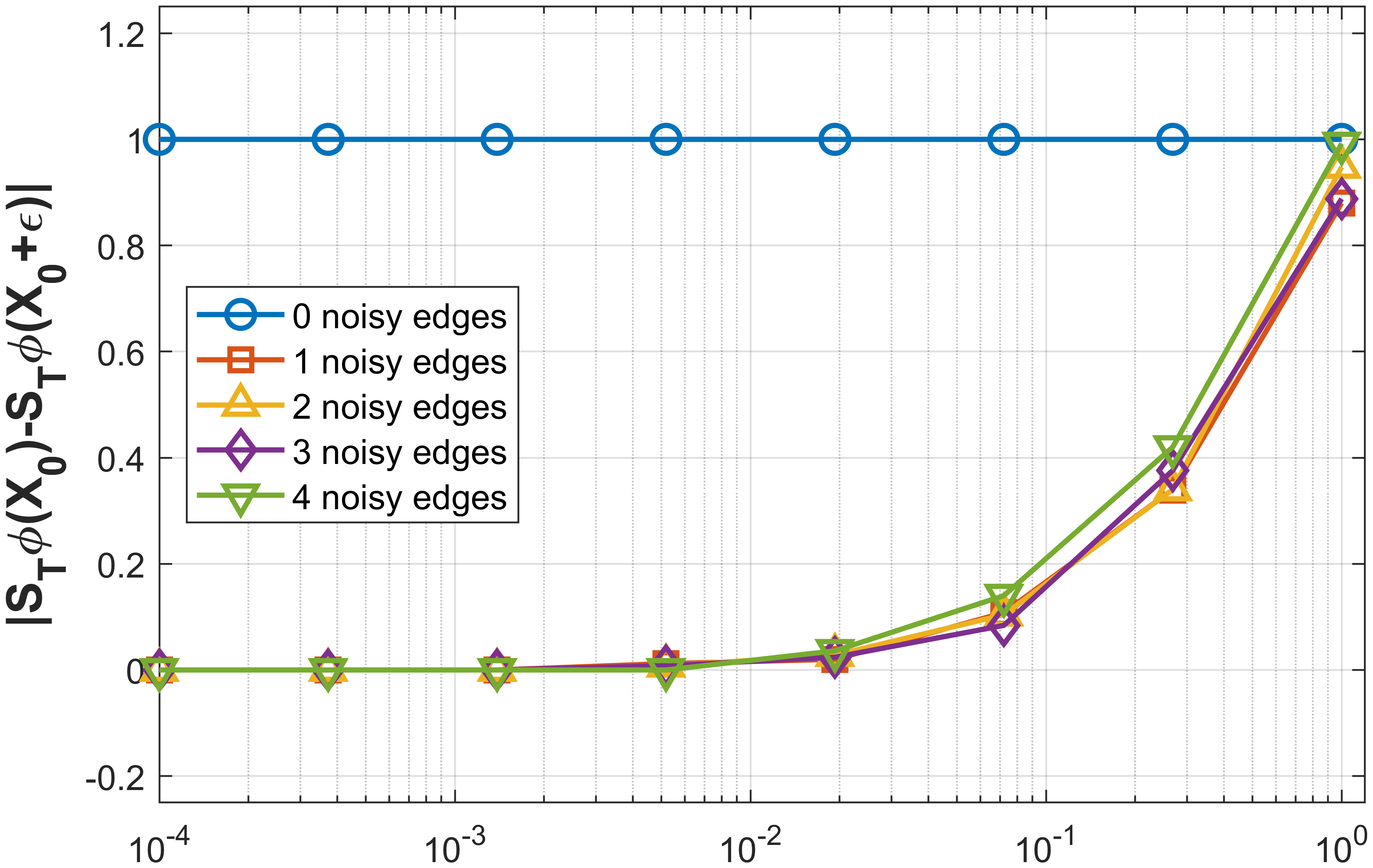}
	\caption{Strong Feller verification for $\mathcal{S}_T$ on chain graph with $4$ edges. $x$ axis: $\epsilon$, $y$ axis: the difference of \eqref{nonlineardiffernece}.  $N=2^6$, $\tau=2^{-5}$, $T=2^{-1}$, $M_{traj}=500$, $\epsilon \in \{ 10^{-\tfrac{4k}{7}} : k = 0,1,\dots,7\}$.}
	\label{SF_chain}
\end{figure}

	\textbf{Verification of irreducibility}. We set the nonlinearity $b_j(x_j)=Q_j\sin(x_j),\, j=1,2,3,4$, where $Q_j$ is defined in \eqref{Q}, which satisfies Assumption \ref{nonzero eigenvalue}(ii). We estimate reachability probabilities for every eigen-subspace via a counting procedure:
	\begin{flalign}\label{irreducible_nonlinear}
		\mathbb{P}\bigl( X(T,X_0)\in\text{span}(\Psi^{i,l})\bigl)\approx\frac{1}{M_{traj}}\sum_{m=1}^{M_{traj}}\mathbf{1}_{\{X(T,X_0)\in\text{span}(\Psi^{i,l})\}},\, i=1,2,3,\, l=1,2,\ldots,N.
	\end{flalign}
	 We set the initial value $X_0=0$. Result for $\mathcal{S}_T$ is shown in Fig.~\ref{IR_chain_nonlinear}. Columns 1-3 present the eigen‑directions of $\sigma_2$, and column 4 presents those of $\sigma_1$ (recall $\phi^{1,k}$ in Lemma \ref{operatorproperty}). We observe that the  probabilities across all eigen‑directions are positive, except for the zero noisy edge case (the blue line). This demonstrates that $(\mathcal{S}_t)_{t\ge0}$ is irreducible at time $T$, provided more than one edge is noisy. This verifies Corollary~\ref{bound_chain}.
\begin{figure}[htbp]
	\centering
	{
		\includegraphics[width=0.85\textwidth]{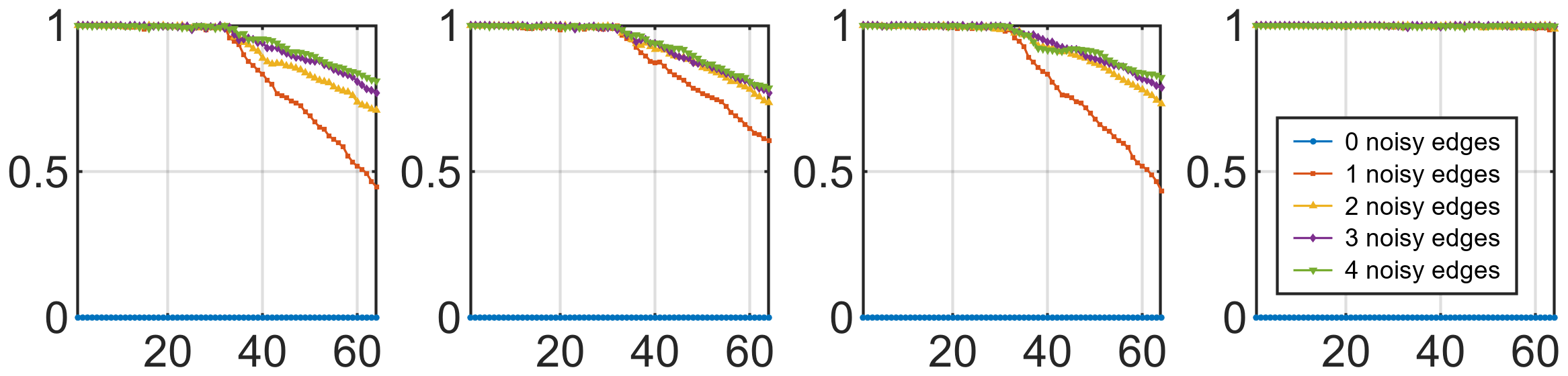}
	}
	\caption{Irreducibility verification for $\mathcal{S}_T$ on chain graph with 4 edges. Subfigures from left to right: $(\Psi^{1,l})_{1\le l\le N},\,(\Psi^{2,l})_{1\le l\le N},\,(\Psi^{3,l})_{1\le l\le N},\,(\phi^{1,k})_{0\le k\le N-1}$. $x$ axis: dimension $N$, $y$ axis: reachability probability. $N=2^6,\, \tau=2^{-5},\, T=2^{-1},\, M_{traj}=500$.}
	\label{IR_chain_nonlinear}
\end{figure}

\textbf{Verification of exponential ergodicity}. We take nonlinearity as $b_j(x_j)=\frac{x_j}{2\sqrt{1+x_j^2}}$ for $j=1,\dots,4$, which is globally Lipschitz continuous with constant $K=\frac{1}{2}<\mu_1=\bigl(\frac{\pi}{4}\bigl)^2$.  From Theorem~\ref{invariantmeasureth} we know that \eqref{SEE} on the chain graph with this nonlinearity admits a unique invariant measure that is exponentially ergodic. To numerically verify this property, we use Monte Carlo method to compute the empirical average \begin{flalign}\label{invariant}
	\mathbb{E}\bigl(\psi(X(t,X_0))\bigl)\approx\frac{1}{M_{traj}}\sum_{m=1}^{M_{traj}}\psi\left(X^{\tau,N}(t,X_0)\right),\, \psi\in\mathcal{B}_b(H)
\end{flalign}
at each time node. The parameters are set as: spatial Galerkin dimension $N=2^5$, the time stepsize $\tau=2^{-3}$, the terminal time $T=30$, the the number of trajectory $M_{traj}=1000$ and $\psi(h)=\sin(\|h\|_H),\,h\in H$.  
 Fig.~\ref{IM_chain} displays  the results for three different initial values and four different noise configurations. We observe that under the same noise configuration, all curves exponentially converge to the same limiting value, despite different initial values, which serves as numerical verification for the existence and exponential ergodicity of a unique invariant measure (see Theorem \ref{invariantmeasureth}). The initial values are chosen as: $X_0^{(1)}=0,\, X_0^{(2)}=\sum_{l=1}^{N}\phi^{1,l},\,X_0^{(3)}=\sum_{l=1}^{N}\xi^{1,l}\phi^{1,l}+\sum_{i=1}^{3}\sum_{l=1}^{N}\tilde{\xi}^{i,l}\Psi^{i,l}$, where $\{\xi^{1,l},\, l=1,2,\ldots,N\},\ \{\tilde{\xi}^{i,l},\, i=1,2,3,\, l=1,2,\ldots,N\}$ are all independently drawn from the standard normal distribution $\mathscr{N}(0,1)$.
\begin{figure}[htbp]
	\centering
	{
		\includegraphics[width=0.65\textwidth]{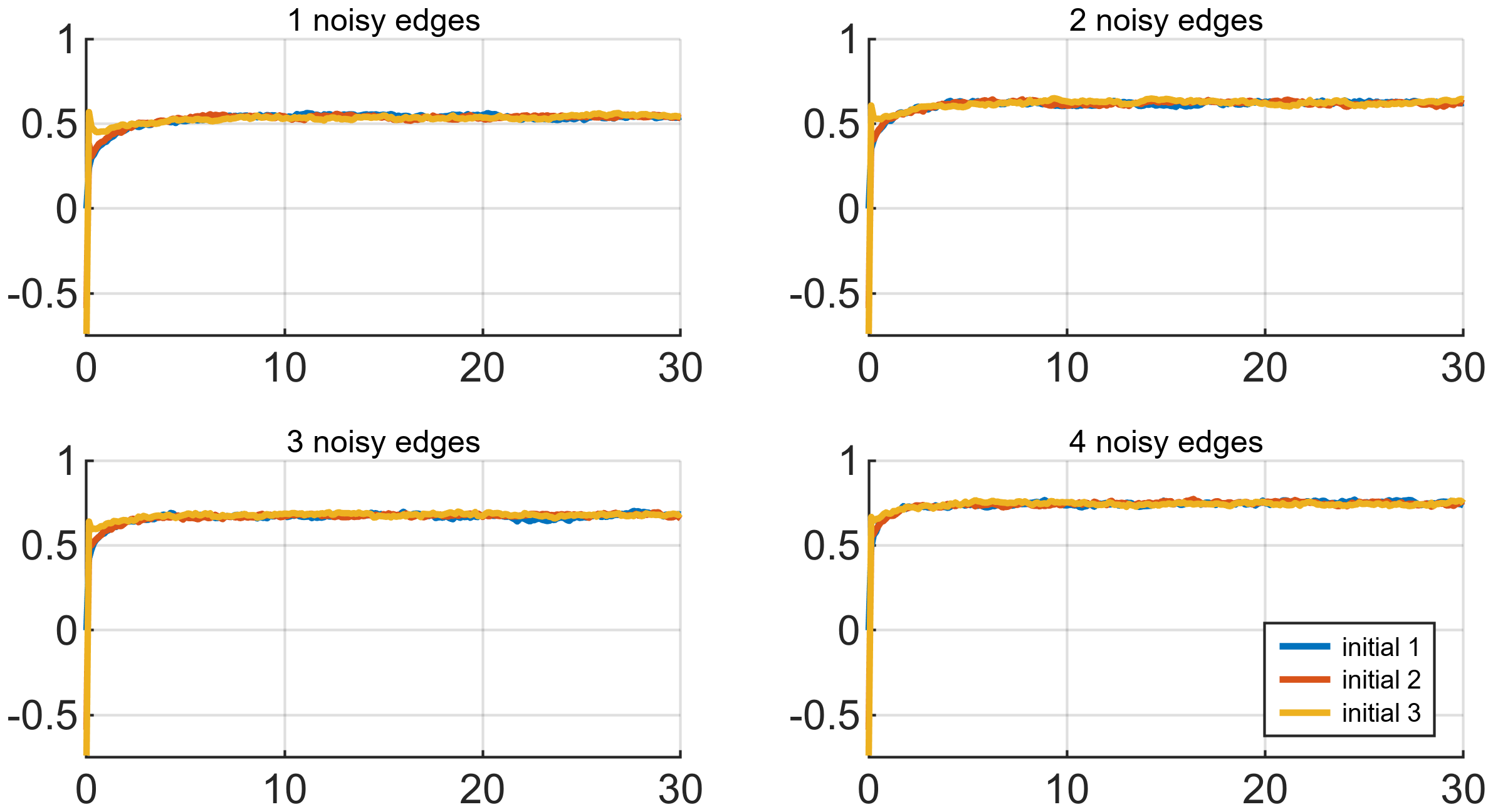}
	}
	\caption{Exponential ergodicity verification for \eqref{SEE} on chain graph with 4 edges. $x$ axis: time, $y$ axis: empirical average. $N=2^5,\, \tau=2^{-3},\, T=30,\, M_{traj}=1000.$ 3 initial values: $X_0^{(1)},\, X_0^{(2)},\, X_0^{(3)}$.}
	\label{IM_chain}
\end{figure}
\subsection{Star graph}

We consider star graph $R$ in Fig.~\ref{star_graph} with $m=4$. Then the eigenfunctions corresponding to $0\in\sigma(\tilde{A}(R))$ are denoted by $\{\{\Psi^{1,l}\},\{\Psi^{2,l}\},\{\Psi^{3,l}\}:l=1,2,\ldots,N\}$.
 Figs.~\ref{SF_star} and \ref{IR_star_nonlinear} provide a numerical verification of Corollary~\ref{bound_star}. The parameters are set the same as those of the chain graph.

\textbf{Verification of strong Feller property}. For the nonlinearity, we set $b_j(x_j)=Q_j\sin(x_j)$, $j=1,2,3,4$, where the coefficients $Q_j$ are defined in \eqref{Q}. Fig.~\ref{SF_star} shows that only for the cases of $3$ noisy edges (purple line) and $4$ noisy edges (green line), \eqref{nonlineardiffernece} decays to $0$ when $\epsilon$ converges to $0$. This means that the semigroup $(\mathcal{S}_t)_{t\ge0}$ on the star graph is strong Feller at time $T$ with at most one noise exception (see Corollary \ref{bound_star}).
\begin{figure}[htbp]
	\centering
	{
		\includegraphics[width=0.40\textwidth]{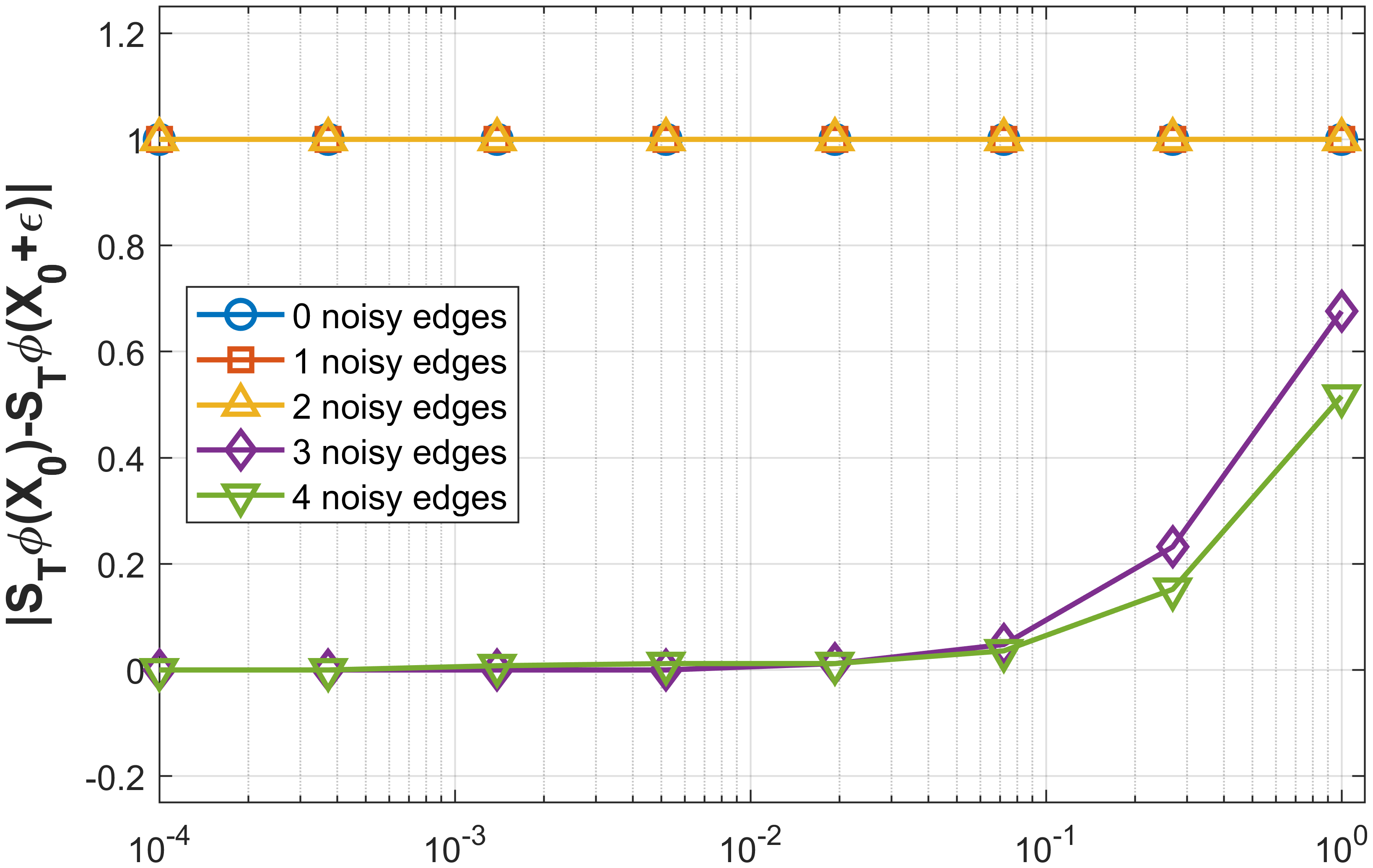}
	}
	\caption{Strong Feller verification for $\mathcal{S}_T$ on star graph with $4$ edges. $x$ axis: $\epsilon$, $y$ axis: the difference of \eqref{nonlineardiffernece}. $N=2^6,\, \tau=2^{-5},\, T=2^{-1},\,\epsilon \in \{ 10^{-\tfrac{4k}{7}} : k = 0,1,\dots,7\},\, M_{traj}=500$.}
	\label{SF_star}
\end{figure}

\textbf{Verification of irreducibility}.  We take  $b_j(x_j)=Q_j\sin(x_j),\, j=1,2,3,4$. In Fig.~\ref{IR_star_nonlinear}, columns 1-3 present the reachability probability \eqref{irreducible_nonlinear} for eigen‑directions corresponding to $-\mu_l=-(\tfrac{1}{2}+(l-1))^2\pi^2\in\sigma_2$, and column 4 presents those to $\sigma_1$. As shown in the first column, only the cases of 3 noisy edges (purple line) and 4 noisy edges (green line) make reachability probabilities of eigen-directions $\Psi^{1,l},\, l=1,2,\ldots,N$ positive. And the two cases also produce positive reachability probabilities for other eigen-directions, as seen in columns 2-4. Consequently, $(\mathcal{S}_t)_{t\ge0}$ is irreducible at time $T$, provided all edges are noisy with one edge exception, which verifies Corollary \ref{bound_star}.
\begin{figure}[htbp]
	\centering
	{
		\includegraphics[width=0.85\textwidth]{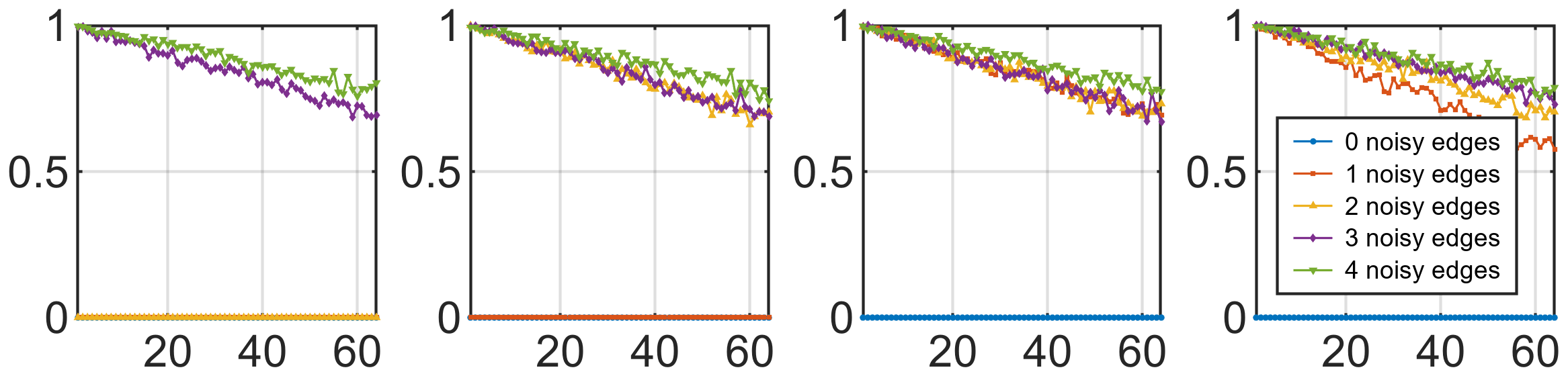}
	}
	\caption{Irreducibility verification for $\mathcal{S}_T$ on star graph with 4 edges.  Subfigures from left to right: $(\Psi^{1,l})_{1\le l\le N},\,(\Psi^{2,l})_{1\le l\le N},\,(\Psi^{3,l})_{1\le l\le N},\,(\phi^{1,k})_{0\le k\le N-1}.$ $x$ axis: dimension $N$, $y$ axis: reachability probability. $N=2^6,\, \tau=2^{-3},\, T=2^{-1},\, M_{traj}=500$.}
	\label{IR_star_nonlinear}
\end{figure}

\textbf{Verification of exponential ergodicity}. We take nonlinearity as $b_j(x_j)=\frac{x_j}{\sqrt{1+x_j^2}}$ for $j=1,\dots,4$, which is globally Lipschitz continuous with constant $K=1<\mu_1=\bigl(\frac{\pi}{2}\bigl)^2$. Theorem~\ref{invariantmeasureth} ensures that \eqref{SEE} on the star graph with this nonlinearity admits a unique invariant measure which is exponentially ergodic. The numerical experiment exhibits the similar result to the chain graph case, which is presented in Fig.~\ref{IM_star} in Appendix \ref{im_supp}.
\subsection{$S$-atom Tree}
Consider a more complex tree $T'$ in Fig.~\ref{T' graph} with 7 edges whose normalized adjacency matrix $\tilde{A}(T')\in\mathbb{R}^{8\times8}$ is given below:
\begin{figure}[htbp]
		\centering
		\begin{minipage}[c]{0.45\textwidth}
			\centering
			\vspace{0.8cm}
			\begin{tikzpicture}[scale=1.2,line width=0.8pt]
				\filldraw (0,0) circle (1.5pt) node[below right]{1};
				
				\draw (0,0) -- (-1,0) node[below] {3};
				\filldraw (-1,0) circle (1.5pt);
				
				\draw (0,0) -- (0,1) node[right] {2};
				\filldraw (0,1) circle (1.5pt);
				
				\draw (0,0) -- (0,-1) node[right] {4};
				\filldraw (0,-1) circle (1.5pt);
				
				\draw (0,0)--(1,0) node[below]{5};
				\filldraw (1,0) circle (1.5pt);
				
				\draw (1,0)--(2,0) node[right]{6};
				\filldraw (2,0) circle (1.5pt);
				
				\draw (2,0)--(2,1) node[right]{7};
				\filldraw (2,1) circle (1.5pt);
				
				\draw (2,0)--(2,-1) node[right]{8};
				\filldraw (2,-1) circle (1.5pt);
			\end{tikzpicture}
			\caption{$S$-atom $T'$}
			\label{T' graph}
		\end{minipage}
		\hspace{0.2cm}
		\begin{minipage}[b]{0.45\textwidth}
			\centering
			\[
			\tilde{A}(T') = 
			\begin{bmatrix}
				0    & \frac14           & \frac14  & \frac14 &\frac14&0&0 & 0 \\
				1    & 0           & 0   & 0 & 0&0&0 & 0 \\
				1    & 0           & 0   & 0 & 0&0&0 & 0 \\
				1    & 0           & 0   & 0 & 0&0&0 & 0 \\
				\frac12    & 0& 0        & 0 & 0&\frac12    & 0 &0\\
				0& 0& 0        & 0 & \frac13&  0  & \frac13 &\frac13\\
				0    & 0& 0        & 0 & 0&1    & 0 &0\\
				0   & 0& 0        & 0 & 0&1    & 0 &0\\
			\end{bmatrix}
			\]
		\end{minipage}
	\end{figure}
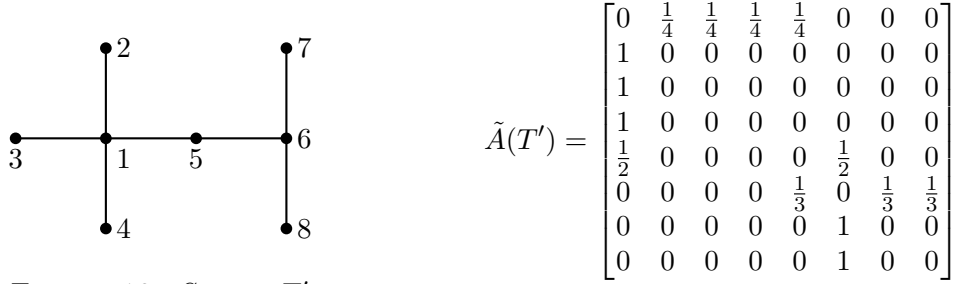
\begin{claim}
	For the SPDE \eqref{SEE} on the $S$-atom $T'$, under Assumptions \ref{b1} and \ref{nonzero eigenvalue}(ii), Theorem \ref{main theorem} holds with $|\mathcal{Z}(T')| \le3$.	
\end{claim}
\begin{proof}
	The spectrum of $\tilde{A}(T')$ is  $\sigma(\tilde{A}(T'))=\{0,\pm1,\pm\frac{\sqrt{102}}{12}\}$, where the eigenvalue $0$ has multiplicity $4$. The eigenvectors corresponding to $\frac{\sqrt{102}}{12}$ and $-\frac{\sqrt{102}}{12}$ are respectively
	\begin{flalign*}
		\bigl[-\tfrac{\sqrt{102}}{16},-\tfrac{3}{4},-\tfrac{3}{4},-\tfrac{3}{4},\tfrac{1}{8},\tfrac{\sqrt{102}}{12},1,1\bigl]^\top,\quad\bigl[\tfrac{\sqrt{102}}{16},-\tfrac{3}{4},-\tfrac{3}{4},-\tfrac{3}{4},\tfrac{1}{8},-\tfrac{\sqrt{102}}{12},1,1\bigl]^\top,
	\end{flalign*}
	which satisfy Assumption \ref{nonzero eigenvalue}(i). One can verify that the null space of $T'$ is given by
	\begin{flalign*}
		\mathcal{N}(T')=\text{span}\Bigl\{&[0,1,0,0,-1,0,0,1]^\top,[0,0,1,0,-1,0,0,1]^\top,\\
		&[0,0,0,1,-1,0,0,1]^\top,[0,0,0,0,0,0,1,-1]^\top\Bigl\}.
	\end{flalign*}
	Observe that $\mathbf{Supp}(T')=\{2,3,4,5,7,8\}$, $\mathbf{Core}(T')=\{1,6\}$, $\mathbf{Supp}(T')\cup\mathbf{Core}(T')=V(T')$ and the vertices $1$, $6$ are not adjacent. Hence $T'$ is an $S$‑atom. Then $|\mathcal{Z}(T')|\le 7-|\mathbf{Supp}(T')|+|\mathbf{Core}(T')|=7-6+2=3$.
\end{proof}
\begin{remark}
	Unlike the star graph and chain graph, the structure of a general $S$-atom lacks symmetry, and thus noise-free edges should be selected carefully. For $T'$ in Fig.~\ref{T' graph}, by the expression of null space $\mathcal{N}(T')$, the three noise‑free edges can be selected according to the following constraints: no more than one from the set $\{e_{12}, e_{13}, e_{14}\}$, no more than two from $\{e_{15}, e_{56}\}$, and no more than one from $\{e_{67}, e_{68}\}$.
\end{remark}
 In the numerical experiments, we examine five distinct noise configurations:
 $\mathcal{Z}=E(T')$ (blue);  $\mathcal{Z}=\{e_{67},e_{68}\}$ (red); $\mathcal{Z}=\{e_{12},e_{67}\}$ (yellow);  $\mathcal{Z}=\{e_{12},e_{15},e_{67}\}$ (purple); $\mathcal{Z}=\varnothing$ (green). Denote the eigenfunctions corresponding to the eigenvalues $\{-\frac{\sqrt{102}}{12},\frac{\sqrt{102}}{12}\} \text{ and }\{0\}$ of $\tilde{A}(T')$ by $$\{\{\Psi^{1,l}\},\{\Psi^{2,l}\}:l=1,2,\ldots,N\}\text{ and }\{\{\Psi^{3,l}\},\{\Psi^{4,l}\},\{\Psi^{5,l}\},\{\Psi^{6,l}\}:l=1,2,\ldots,N\}.$$
 
\textbf{Verification of strong Feller property}. For the nonlinearity $B$, we set $b_j(x_j)=Q_j\sin(x_j)$, $j=1,2,\ldots,7$. Let $X_0=0$, $\epsilon \in \{ 10^{-\tfrac{4k}{7}} : k = 0,1,\dots,7\}$,  $\mathcal{E}=\epsilon\sum_{l=1}^{N}\Psi^{6,l}$ and sign function $\phi(h)=\text{sgn}(\langle h,\Psi^{6,1}\rangle)\in\mathcal{B}_b(H)$, where $\Psi^{6,1}$ denotes the first eigenfunction of family $\{\Psi^{6,l}\}$. Fig.~\ref{SF_T} shows that the red and blue lines remain constant at $1$, whereas the other three lines decay to $0$. This indicates that the semigroup $(\mathcal{S}_t)_{t\ge0}$ on $T'$ is not strong Feller at time $T$ when $\mathcal{Z}=\{e_{67},e_{68}\}$ or $\mathcal{Z}=E(T')$, but is strong Feller when $\mathcal{Z}=\{e_{12},e_{67}\}$,  $\mathcal{Z}=\{e_{12},e_{15},e_{67}\}$ or $\mathcal{Z}=\varnothing$.
\begin{figure}[htbp]
	\centering
	{
		\includegraphics[width=0.40\textwidth]{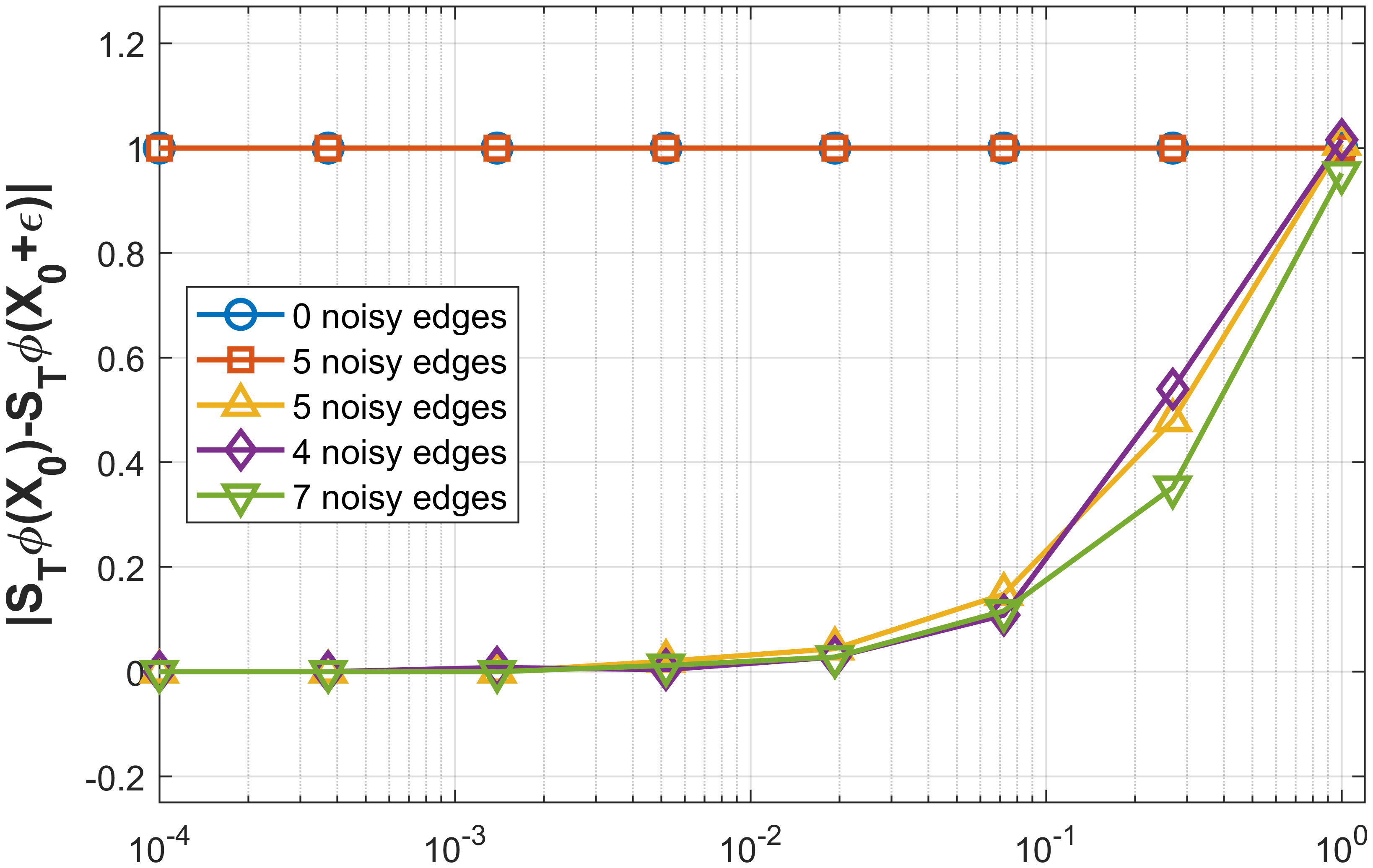}
	}
	\caption{Strong Feller verification for $\mathcal{S}_T$ on $S$-atom $T'$. $x$ axis: $\epsilon$, $y$ axis: the difference of \eqref{nonlineardiffernece}. $N=2^6,\, \tau=2^{-5},\, T=2^{-1},\, \epsilon \in \{ 10^{-\tfrac{4k}{7}} : k = 0,1,\dots,7\},\, M_{traj}=500$.}
	\label{SF_T}
\end{figure}

\textbf{Verification of irreducibility}. The nonlinearity is set as: $b_j(x_j)=Q_j\sin(x_j)$, $j=1,2,\ldots,7$. In
Fig.~\ref{IR_T_nonlinear}, the first six subfigures display six eigenfunctions corresponding to  $\sigma_2$ (the first two for $\frac{\sqrt{102}}{12} $ and $-\frac{\sqrt{102}}{12}$, the next four for $0$) and the last subfigure for eigenfunctions of $\sigma_1$. The red lines in the sixth subfigure of Fig.\ref{IR_T_nonlinear} as well as all the blue lines show zero reachability probability. Consequently, for $\mathcal{Z}=\{e_{67},e_{68}\}$ and $\mathcal{Z}=E(T')$, $(\mathcal{S}_t)_{t\ge0}$ is not irreducible at time $T$. But the semigroup is irreducible at $T$ for either the cases  $\mathcal{Z}=\{e_{12},e_{67}\}$, $\mathcal{Z}=\{e_{12},e_{15},e_{67}\}$ or $\mathcal{Z}=\varnothing$. 

\begin{figure}[htbp]
	\centering
	{
		\includegraphics[width=0.75\textwidth]{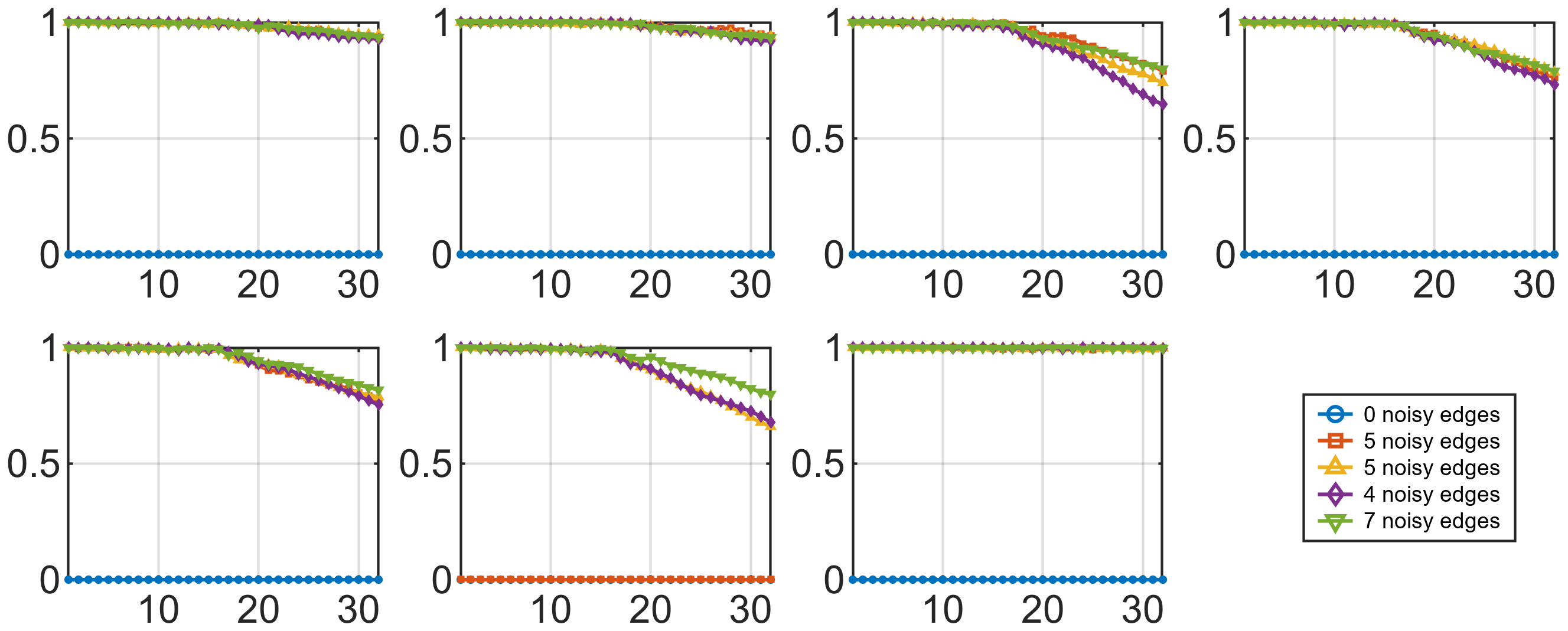}
	}
	\caption{Irreducibility verification for  $\mathcal{S}_T$ on $S$-atom  $T'$.  Seven subfigures from left to right: $(\Psi^{1,l})_{1\le l\le N},\ldots,(\Psi^{6,l})_{1\le l\le N},(\phi^{1,k})_{0\le k\le N-1}$. $x$ axis: dimension $N$, $y$ axis: reachability probability. $N=2^5,\, \tau=2^{-4},\, T=2^{-1},\,M_{traj}=500$.}
	\label{IR_T_nonlinear}
\end{figure}
\textbf{Verification of exponential ergodicity}. We take nonlinearity as $b_j(x_j)=\frac{x_j}{5\sqrt{1+x_j^2}}$ for $j=1,\dots,7$, which is dissipative and thus the corresponding equation \eqref{SEE} has a unique invariant measure and exhibits exponential ergodicity. This can be verified numerically; see Fig.~\ref{IM_T} in Appendix \ref{im_supp}.
\begin{remark}
 For general trees, the graph structures can be highly complex, making it difficult to compute the upper bound $|\mathcal{Z}(\Gamma)|$ directly from the sets $\textbf{Core}(\Gamma)$ and $\textbf{Supp}(\Gamma)$. To deal with this case, we use the following method based on the matching number $\nu(\Gamma)$ to compute this bound.
    From \cite[Corollary 4.15]{jaume2018null}, we have \[
		|\textbf{Core}(\Gamma)|  = \nu(\Gamma)- \frac{|V(\mathcal{F}_N(\Gamma))|}{2} \text{ and }  |\textbf{Supp}(\Gamma)|=\alpha(\Gamma)- \frac{|V(\mathcal{F}_N(\Gamma))|}{2},
		\] where  $\alpha(\Gamma)$ is the independence number. Due to the fact that $\alpha(\Gamma)+\nu(\Gamma)=n=m+1$ for trees, we derive that $$ m-|\textbf{Supp}(\Gamma)|+|\textbf{Core}(\Gamma)|=m-\alpha(\Gamma)+\nu(\Gamma)=2\nu(\Gamma)-1.$$ Thus the upper bound can be given by the matching number $$\mathcal{Z}(\Gamma)\le\min\{2\nu(\Gamma)-1,m-1\}.$$
		%
			%
		Here, the matching number $\nu(\Gamma)$ can be calculated via the classical Hopcroft--Karp--Karzanov algorithm \cite{hopcroft1973n}.
	\end{remark}
\section{Proofs of main results}\label{proof}
In this section, we present the proofs of Theorems \ref{main theorem} and \ref{invariantmeasureth}.
\subsection{Proofs of Theorem \ref{main theorem} and Theorem \ref{invariantmeasureth}} \label{proof of main}
\begin{proof}[Proof of Theorem \ref{main theorem}]
(\uppercase\expandafter{\romannumeral1}) The proof of strong Feller property is divided into two parts, $\mathcal{Z}=\varnothing$ and $\mathcal{Z}\neq\varnothing$.\\
\textbf{(\romannumeral1)} $\bm{\mathcal{Z} = \varnothing.}$
We first prove the strong Feller property of the semigroup $(\mathcal{R}_t)_{t\ge0}$ corresponding to  the linear equation \eqref{linear SEE}. The solution $v(t,X_0)$ is, for every $t>0$, a Gaussian random variable with mean value $P_t X_0$ and covariance operator
\begin{equation}\label{eq:Qt}
	Q_t = \int_0^t P_s P_s^* \, ds .
\end{equation}
For arbitrary element $v = \sum_{k=0}^{\infty} v_k \phi_k \in H,\, v_k\in\mathbb{R}$, we obtain
\begin{align*}
	Q_t v &= \int_0^t P_s P_s^* \Bigl( \sum_{k=0}^{\infty} v_k \phi_k \Bigr) ds 
	= \int_0^t \sum_{k=0}^{\infty} e^{-2\mu_k s} v_k \phi_k \, ds \\
	&= t v_0 + \sum_{k=1}^{\infty} \tfrac{1 - e^{-2\mu_k t}}{2\mu_k} v_k \phi_k 
	=: \sum_{k=0}^{\infty} \rho_k^2 v_k \phi_k ,
\end{align*}
where $\rho_0 = \sqrt{t}$ and $\rho_k = \sqrt{\frac{1 - e^{-2\mu_k t}}{2\mu_k}}$ for $k \ge 1$. Then $Q_t^\frac{1}{2}v=\sum_{k=0}^\infty\rho_kv_k\phi_k$. According to \cite[Theorem 7.2.1]{da1996ergodicity}, $\mathcal{R}_t$ is a strong Feller semigroup at time $t$ if and only if $\text{Im}(P_t) \subset \text{Im}(Q_t^{\frac{1}{2}})$. Note that for any $u = \sum_{k=0}^\infty u_k \phi_k \in H$, we have $P_t u = \sum_{k=0}^\infty e^{-\mu_k t} u_k \phi_k$. By defining
			\begin{flalign*}
				v_k=\tfrac{e^{-\mu_kt}}{\rho_k}u_k=
				\begin{cases}
					\tfrac{1}{\sqrt{t}}u_k,\ &k=0,\\
					e^{-\mu_kt}\sqrt{\tfrac{2\mu_k}{1-e^{-2\mu_kt}}}u_k,\,&k\geq1,
				\end{cases}
			\end{flalign*} 
		we have that $v=\sum_{k=0}^\infty v_k\phi_k$ satisfies $P_t u=Q_t^\frac{1}{2}v$ and $\|v\|_H^2=\sum_{k=0}^\infty v_k^2\le\tfrac{1}{t}\|u\|_H^2$ due to the fact that $\frac{xe^{-x}}{1-e^{-x}},\, x>0$ is uniformly bounded. This shows that $\text{Im}(P_t) \subset \text{Im}(Q_t^{\frac{1}{2}})$.
 	
We now turn to the strong Feller property of the semigroup $ (\mathcal{S}_t)_{t\ge0}$ corresponding to the nonlinear equation \eqref{SEE}. To this end, define the operator $\gamma(t) = Q_t^{-\frac{1}{2}} P_t$. Notice that for any $t>0$, $\sum_{k=0}^\infty\langle Q_t\phi_k,\phi_k\rangle=t+\sum_{k=1}^\infty\frac{1-e^{-2\mu_k t}}{2\mu_k}<\infty$, thus the operator $Q_t$ is of trace class. According to \cite[Theorem 7.2.4]{da1996ergodicity}, it suffices to show that $\|\gamma(\cdot)\|_{\mathcal{L}(H)} \in L^1_{\text{loc}}(0,+\infty)$. Indeed, since $\gamma(t)\phi_k = e^{-\mu_k t}\rho_k^{-1}\phi_k$ for all $k\ge0$, we have
\[
\|\gamma(t)\|_{\mathcal{L}(H)} \le \sup_{k\ge0}\tfrac{e^{-\mu_k t}}{\rho_k} \le \tfrac{1}{\sqrt{t}},
\]
and the function $t\mapsto 1/\sqrt{t}$ belongs to $L^1_{\mathrm{loc}}(0,+\infty)$. Therefore, $(\mathcal{S}_t)_{t\ge0}$ is strong Feller.

\textbf{(\romannumeral2)}$\bm{\mathcal{Z} \neq \varnothing.}$ We split the proof into five steps.

\textbf{Step 1: Equivalent description for the strong Feller property of $\mathcal{R}_t$.} 
	Define \begin{flalign*}
		Q_t=\int_0^tP_sQQ^*P_s^*ds.
	\end{flalign*}
Notice the equivalent relation:
	\begin{flalign*}
	 \text{Im}(P_t)\subset \text{Im}(Q_t^\frac{1}{2})\text{ if and only if } \text{ker}(P_t)\supset \text{ker}(Q_t^\frac{1}{2}).
	\end{flalign*}
	Since for any $h \in H$, the spectral expansion yields	$P_t h = \sum_{k=0}^{\infty} e^{-\mu_k t} \langle h, \phi_k \rangle \phi_k.$ If $P_t h = 0$, then $e^{-\mu_k t} \langle h, \phi_k \rangle = 0$ for each $k$. Since $e^{-\mu_k t} > 0$, it follows that $\langle h, \phi_k \rangle = 0$ for all $k$. Hence $h = 0$, and therefore $\ker(P_t) = \{0\}$. By \cite[Theorem 7.2.1]{da1996ergodicity}, we obtain that $\text{ker}(Q_t^\frac{1}{2})=\{0\}$ is a sufficient and necessary condition for the strong Feller property of the semigroup $(\mathcal{R}_t)_{t\ge0}$. For any $f=\sum_{k=0}^{\infty}f_k\phi_k\in H,\, f_k\in\mathbb{R}$,
	\begin{flalign*}
		\|Q_t^\frac{1}{2}f\|_H^2=\langle Q_tf,f\rangle=\int_0^t\langle P_sQQ^*P_s^*f,f\rangle ds=\int_0^t\|Q^*P_s^*f\|_{H}^2ds=\int_{0}^{t}\|QP_sf\|_{H}^2ds.
	\end{flalign*}
	Then by the continuity of $(P_s)_{s\ge0}$, 
	\begin{flalign*}
		\text{ker}(Q_t^\frac{1}{2})=\{f\in H:QP_sf=0,\,\forall s\in[0,t]\}.
	\end{flalign*}
	Since $QP_sf=\sum_{k=0}^{\infty}e^{-\mu_ks}f_k(Q\phi_k)$, we deduce \begin{flalign}
		\text{ker}(Q_t^\frac{1}{2})=\overline{\text{span}\{\phi_k:Q\phi_k=0,\,k\ge0\}}.
	\end{flalign}

For all eigenfunctions $\phi_k$ of $(\Delta,\mathcal{D}(\Delta))$, if \begin{align*}
	\exists\, \phi_k\ne0 \text{ such that } Q\phi_k=0,\;\; \text{i.e.},\  \phi_k|_{\mathcal{Z}}\ne0 \text{ and } \phi_k|_{\mathcal{Y}}=0,
\end{align*}then $\ker(Q_t^\frac{1}{2})\neq\{0\}$ and $(\mathcal{R}_t)_{t\ge0}$ is not strong Feller. And 
if
\begin{align}\label{eigenfunction}
Q\phi_k\ne0 \text{ for all }\phi_k,
\end{align}  
then ker$(Q_t^\frac{1}{2})=\{0\}$ and $(\mathcal{R}_t)_{t\ge0}$ is a strong Feller semigroup.

\textbf{Step 2: Structure of eigenfunctions in sets $\sigma_1$ and $\sigma_2$.}
	In this step, we analyze the structure of eigenfunctions in sets $\sigma_1$ and $\sigma_2$, respectively, where $\sigma_1$ and $\sigma_2$ are given in Lemma \ref{operatorproperty}. The following claim shows that the eigenfunctions of $\sigma_1$ do not affect the strong Feller property, whose proof is postponed in Section \ref{Proof of Claim}.
	
	\textit{Claim 1:
		No eigenfunction corresponding to any eigenvalue $\lambda\in\sigma_1$ is identically zero on every edge.}
	
The key to establishing the strong Feller property therefore lies in the properties of eigenfunctions corresponding to eigenvalues in $\sigma_2$. For $-\mu_l\in\sigma_2,\,l\in\mathbb{Z}$, by Lemma \ref{operatorproperty},  
\begin{flalign*}
	\phi^{2,l}_{j}(x_j)=\tfrac{1}{\sin\sqrt{\mu_l}}\left(\phi^{2,l}_{j}(0)\sin(\sqrt{\mu_l}(1-x_j))+\phi^{2,l}_{j}(1)\sin(\sqrt{\mu_l}x_j)\right).
\end{flalign*}
For its derivative we have
	\begin{flalign*}
		(\phi^{2,l}_{j})'(x_j)&=\tfrac{\sqrt{\mu_l}}{\sin(\sqrt{\mu_l})}\left(-\phi_j^{2,l}(0)\cos\left(\sqrt{\mu_l}(1-x_j)\right)+\phi^{2,l}_{j}(1)\cos(\sqrt{\mu_l}x_j)\right),\\
		(\phi^{2,l}_{j})'(0)&=\tfrac{\sqrt{\mu_l}}{\sin(\sqrt{\mu_l})}\left(-\phi_j^{2,l}(0)\cos\left(\sqrt{\mu_l}\right)+\phi^{2,l}_{j}(1)\right),\\
		(\phi^{2,l}_{j})'(1)&=\tfrac{\sqrt{\mu_l}}{\sin(\sqrt{\mu_l})}\left(-\phi_j^l(0)+\phi^{2,l}_{j}(1)\cos(\sqrt{\mu_l})\right).
	\end{flalign*}
 We define the set of vectors $$\mathbf{U}:=\bigl\{U^l=[U^l(v_1),U^l(v_2),\dots,U^l(v_n)]^\top\in\mathbb{R}^n: U^l(v_i)=\phi^{2,l}(v_i),i=1,2,\dots,n,l\in\mathbb{Z}\bigl\},$$ where $U^l$ can be proved as an eigenvector of the matrix $\tilde{A}(\Gamma)$ corresponding to eigenvalue $\cos(\sqrt{\mu_l})$ for each $l\in\mathbb{Z}$. To illustrate this, consider a vertex $v\in V(\Gamma)$ and, for simplicity, suppose that $e_j(0)=v$ holds for every edge $ e_j\in E_v$, where $E_v$ denotes the set of edges incident to $v$. Then $U^l(v)=U^l(e_j(0))=\phi^{2,l}_j(0)$ and $U^l(e_j(1))=\phi^{2,l}_j(1)$ for each $e_j\in E_v$. By the Neumann--Kirchhoff condition \eqref{eq:kirchhoff} at vertex $v$, we have
	\begin{flalign*}
		\sum_{e_j\in E_v}(\phi^{2,l}_{j})'(0)=\tfrac{\sqrt{\mu_l}}{\sin(\sqrt{\mu_l})}\sum_{e_j\in E_v}\left(-\phi^{2,l}_j(0)\cos\left(\sqrt{\mu_l}\right)+\phi^{2,l}_{j}(1)\right)=0,
	\end{flalign*}
which implies
\begin{flalign*}
		\sum_{e_j\in E_v}\phi^{2,l}_{j}(1)=\sum_{e_j\in E_v}\cos(\sqrt{\mu_l})\phi^{2,l}_{j}(0)=\cos(\sqrt{\mu_l})U^l(v)\sum_{e_j\in E_v}1=\text{deg}(v)\cos(\sqrt{\mu_l})U^l(v).
	\end{flalign*}
Hence
\begin{flalign*}
	 \cos(\sqrt{\mu_l})U^l(v)=\frac{1}{\text{deg}(v)}\sum_{e_j\in E_v}\phi_j^{2,l}(1)=\frac{1}{\text{deg}(v)}\sum_{e_j\in E_v}U^l(e_j(1)).
	 \end{flalign*}
 Since each $e_j(1),\,e_j\in E_v$ denotes an adjacent  vertex of $v$, then
 \begin{flalign}\label{e0}
\frac{1}{\text{deg}(v)}\sum_{v'\sim v}U^l(v')=\cos(\sqrt{\mu_l})U^l(v).
	\end{flalign}
	Similarly, if $e_j(1)=v$ for all $e_j\in E_v$, we can also get \begin{flalign}\label{e1}
		\frac{1}{\text{deg}(v)}\sum_{v'\sim v}U^l(v')=
		\cos(\sqrt{\mu_l})U^l(v).
	\end{flalign}
It follows from \eqref{e0}, \eqref{e1} and the definition of $\tilde{A}(\Gamma)$ (see Lemma \ref{operatorproperty}) that
	\begin{flalign*}
		\tilde{A}(\Gamma)U^l=\cos(\sqrt{\mu_l})U^l.
	\end{flalign*}
Hence $U^l$ is the eigenvector of $\tilde{A}(\Gamma)$ corresponding to eigenvalue $\cos(\sqrt{\mu_l})$.

    If there exists such a vector $U\in\mathbf{U}$ satisfying  $U(e_j(0))=U(e_j(1))=0$ for some edge $e_j\in E(\Gamma)$, then the corresponding  $\phi^{2,l}$ satisfies $\phi^{2,l}_j(0)=\phi^{2,l}_j(1)=0$ and further $\phi^{2,l}_j\equiv0$. Therefore, gaining the strong Feller property, according to \eqref{eigenfunction}, reduces to locating $0$ entries of  each vector $\mathbf{U}$. Determining these zero positions allows us to check whether the eigenfunction $\phi^{2,l}$ equals to $0$ on some edge. Under Assumption \ref{nonzero eigenvalue}(i), eigenfunctions corresponding to nonzero eigenvalue belong to $\sigma(\tilde{A}(\Gamma))\cap(-1,1)$ are nonzero on all edges, thus we only need to study $U\in\mathbf{U}$ satisfying $\tilde{A}(\Gamma)U=0$ (when $\cos(\sqrt{\mu_l})=0$). The relation  $\tilde{A}(\Gamma)=D^{-1}(\Gamma)A(\Gamma)$ implies that $A(\Gamma)$ and $\tilde{A}(\Gamma)$ share the same eigenvectors for eigenvalue $0$. Therefore it suffices to analyze $A(\Gamma)U=0$.
    	
    \textbf{Step 3: Quantifying effect of $S$-atom on $\mathcal{R}_t$.}
     Recall that the set \textbf{Core}($\Gamma$) introduced in Section \ref{decomposition} consists of vertices $v\in V(\Gamma)$ such that $U(v)\equiv0$ for all eigenvectors $U\in\mathbf{U}$ associated with eigenvalue $0$. Via the null decomposition, given any tree, we can decompose it into the $S$-set and $N$-set, which are collections of $S$-trees and $N$-trees, respectively. Since $N$-trees are non-singular, it suffices to focus on $S$-trees. Furthermore, every $S$-tree can be decomposed into an $A$-set, which is a set of  $S$-atoms. Here, an $S$-atom is the minimal structural unit required in our work. Consider such an $S$-atom $S$ with $m$ edges, $n = m+1$ vertices, and adjacency matrix $A(S)$. From \cite[Theorem 3.15]{jaume2017s}, it holds that
	\begin{flalign*}
		\text{dim}(\mathcal{N}(S))=\text{dim}(\text{ker}(A(S)))=|\textbf{Supp}(S)|-|\textbf{Core}(S)|=:d.
	\end{flalign*}
Let $\{g_1,g_2,\dots,g_d\}$ be a basis of $\ker(A(S))$, and define $G:=[g_1,g_2,\dots,g_d]\in\mathbb{R}^{n\times d}$, which satisfies $\text{rank}(G)=d$. Let $M(S)\in\mathbb{R}^{n\times m}$ denote the incidence matrix of $S$. Then $\text{rank}(M(S))=m$. The following claim shows that multiplying by $M(S)$ does not reduce the rank of $G$, whose proof is postponed in Section \ref{Proof of Claim}.

\textit{Claim 2: $\text{rank}(M(S)^\top G)=d$.}

From \textit{Claim 2}, the matrix $M(S)^\top G$ can be expressed as $M(S)^\top G=[\beta_1,\beta_2,\dots,\beta_d]$, where $\beta_1,\beta_2,\dots,\beta_d\in\mathbb{R}^n$ are linearly independent. The following claim states that support of any nonzero eigenvector for $0$ eigenvalue is an independent set (i.e., consists of isolated vertices), whose proof is postponed in Section \ref{Proof of Claim}.

\textit{Claim 3: For $S$-atom $S$, the support $\mathbf{Supp}(S)$, that is the support of its null space, consists of isolated vertices. In other words, no edge connects two vertices both lying in $\mathbf{Supp}(S)$.}
	
Define the vector \[
\varphi^l=\left[\text{sgn}\bigl(\phi^{2,l}_1(0)\bigl)+\text{sgn}\bigl(\phi^{2,l}_1(1)\bigl),\text{sgn}\bigl(\phi^{2,l}_2(0)\bigl)+\text{sgn}\bigl(\phi^{2,l}_2(1)\bigl),\dots,\text{sgn}\bigl(\phi^{2,l}_m(0)\bigl)+\text{sgn}\bigl(\phi^{2,l}_m(1)\bigl)\right]^\top.
\]\textit{Claim 3} implies that, for each component  $\varphi^l_j,\, j=1,2,\ldots,m$ of vector $\varphi^l$, at most one of the values $\phi_j^{2,l}(0)$ and $\phi_j^{2,l}(1)$ is nonzero. Moreover, for $M(S)^\top\in\mathbb{R}^{m\times n},\,j=1,2,\dots,m,\text{ and }i=1,2,\dots,n$, the $(j,i)$-entry $m_{j,i}=1$ if $v_i$ is an endpoint of $e_j$ and $m_{j,i}=0$ if not. Due to this fact, for each fixed $\mu_l=\left(l+\tfrac{1}{2}\right)^2\pi^2,\,l\in\mathbb{Z}$, the columns of $M(S)^\top G$ are mutually linearly independent and hence $[\beta_1,\beta_2,\dots,\beta_d]$ can be regarded as the complete basis of vector $\varphi^l$, i.e.
	for $\forall\, l\in\mathbb{Z},\,\varphi^l=\sum_{i=1}^d\alpha_i\beta_i,\, \alpha_i\in\mathbb{R}$. Therefore
	\begin{flalign*}
		Q\varphi^l=\sum_{i=1}^d\alpha_i(Q\beta_i).
	\end{flalign*}
	Our goal is to deduce $\varphi^l = 0$ from $Q\varphi^l = 0$. For this purpose, the vectors $[Q\beta_1, Q\beta_2, \dots, Q\beta_d]$ must be linearly independent. This forces $\text{rank}(Q)=|\mathcal{Y}|\ge d$, since $[\beta_1,\beta_2, \dots, \beta_d]$ are linearly independent.
	Hence, it follows from $\text{dim}(\text{ker}(Q))=|\mathcal{Z}|$ and $|\mathcal{Y}|+|\mathcal{Z}|=m$ that 
	\begin{flalign*}
		|\mathcal{Z}|\le \min\{m-d,m-1\} =\min\{m-|\textbf{Supp}(S)|+|\textbf{Core}(S)|,m-1\}.
	\end{flalign*}

\textbf{Step 4: Quantifying effect of general tree on $\mathcal{R}_t$.}
	For an arbitrary tree $\Gamma$ with $m$ edges, the noise can also be removed on any edge belonging to $\textbf{ConnE}(\Gamma)\cup\{\textbf{BondE}(S):S\in\mathcal{F}_S(\Gamma)\}$. Consequently, we arrive at the final conclusion: the sharp upper bound on the number of edges without noises while maintaining the strong Feller property of $(\mathcal{R}_t)_{t\ge0}$ is
	\begin{flalign*}
		|\mathcal{Z}|&\le\min\Big\{ m-\Bigl(\sum_{S\in\mathcal{F}_S(\Gamma)}\sum_{A\in\mathcal{F}_A(S)}\bigl({|\textbf{Supp}(A)|-|\textbf{Core}(A)|}\bigl)\Bigl),m-1\Big\}\\
		&=\min\{m-|\textbf{Supp}(\Gamma)|+|\textbf{Core}(\Gamma)|,m-1\}.
	\end{flalign*}

\textbf{Step 5: Quantifying effect of general tree on $\mathcal{S}_t$.}
 For semigroup $(\mathcal{S}_t)_{t\ge0}$ of nonlinear equation \eqref{SEE}, a procedure based on the Girsanov theorem applied as in \cite[Example 2.3]{maslowski2000probabilistic} yields the strong Feller property under Assumption \ref{b1} and \ref{nonzero eigenvalue}(\romannumeral2).
 
	(\uppercase\expandafter{\romannumeral2}) Proof of irreducibility. Regarding the semigroup $(\mathcal{R}_t)_{t\ge0}$ in \eqref{linear semigroup}, we have the chain of equivalences: \begin{align*}
	\mathcal{R}_t \text{ is\  irreducible} \Leftrightarrow\text{ supp}(\mathcal{R}_t(x,\cdot)) = H \Leftrightarrow\ker(Q_t^{\frac{1}{2}}) = \{0\},
\end{align*}
where $\Leftrightarrow$ denotes ``if and only if'' and $\text{supp}(\mathcal{R}_t(x,\cdot))$ denotes the support of probability measure $\mathcal{R}_t(x,\cdot)$.
Hence, the conclusion coincides with that obtained in part \textup{(\uppercase\expandafter{\romannumeral1})}. The irreducibility of the semigroup $(\mathcal{S}_t)_{t\ge0}$ in \eqref{nonlinear semigroup} follows from the analogous analysis of the strong Feller property for $(\mathcal{S}_t)_{t\ge0}$.

 The proof of Theorem \ref{main theorem} is thus finished.
\end{proof}
 
 \begin{proof}[Proof of Theorem \ref{invariantmeasureth}]
The existence of an invariant measure can be obtained by applying the Krylov--Bogoliubov theorem, where the tightness of the distribution of solution comes from the compact embedding $D((-\Delta)^\alpha) \subset\subset H,\,\alpha > 0$ together with the regularity estimate: for any $\alpha\in(0,\tfrac{1}{4})$ and $X_0\in H$, there exists a constant $C>0$ independent of $t$ such that
\begin{flalign}\label{spatial regularity}
	\sup_{t\geq1}\mathbb{E}\left[\| X(t,X_0)\|_{\mathbb{H}^{2\alpha}}\right]\leq C.
\end{flalign} 
The uniqueness and exponential ergodicity of the invariant measure are ensured by: there exist constants $c,\, C>0$ independent of $t$, such that for any $X_0,X_1\in H$,
\begin{flalign}\label{attractive}
	\left(\mathbb{E}\left[\| X(t,X_0)-X(t,X_1)\|_H^2\right]\right)^{\frac{1}{2}}\leq C\| X_0-X_1\|_He^{-ct}.
\end{flalign}
The proofs of \eqref{spatial regularity} and \eqref{attractive} are postponed to Appendix \ref{im}.
\end{proof}

\subsection{Proofs of claims}\label{Proof of Claim}
\begin{proof}[Proof of Claim 1]
	The proof is split into two cases, $\lambda=0$ and $\lambda\neq0$.
	
	(a) For $\lambda=0$, $\phi^{1,0}_j(x_j)=a_jx_j+b_j,\ j=1,2,\ldots,m.$ Then for all $j,\ (\phi^{1,0}_j)'(x_j)=a_j,\ \forall x_j\in[0,1].$
	Define the vector $a=[a_1,a_2,\ldots,a_m]^\top$ and by the Neumann--Kirchhoff condition \eqref{eq:kirchhoff},
	\begin{flalign*}
		\Phi^+(\phi^{1,0})'(\mathbf{0})-\Phi^-(\phi^{1,1})'(\mathbf{1})=\Phi a=0,
	\end{flalign*} 
	where $\mathbf{0}=[0,0,\ldots,0]^\top,\, \mathbf{1}=[1,1,\ldots,1]^\top\in\mathbb{R}^m$. From $\Phi^\top y=0\ ( y\in\mathbb{R}^{m+1})$, each edge  $e=(v_i,v_j)\in E$ yields $y_i-y_j=0$, hence  $y_1=y_2=\dots=y_{m+1}$. These conditions give $m$ linearly independent equations, illustrating $\text{rank}(\Phi)=m$. By a similar argument, the rank of incidence matrix $M(\Gamma)$ is  $\text{rank}(M(\Gamma))=m$. Thus both matrices $\Phi$ and $M(\Gamma)$ have full column rank. Then $a_1=a_2=\cdots=a_m=0$. By the continuity condition and $\|\phi^{1,0}\|_H=1$, we derive  $b_1=b_2=\cdots=b_m=\frac{1}{\sqrt{m}}$.  Therefore $\phi^{1,0}_j\equiv \frac{1}{\sqrt{m}}$ for all $j$.
	
	(b) Since the eigenvalue $\lambda = -k^2\pi^2,\ k\in\mathbb{N}^+$ has multiplicity one, by Lemma \ref{operatorproperty}, the corresponding eigenfunction satisfies
	\begin{flalign*}
		\phi^{1,k}_j(x_j)&=c^k_{1,j}\cos(k\pi x_j)+c^k_{2,j}\sin(k\pi x_j),\, \forall\, x_j\in[0,1],\\
		\phi^{1,k}_j(0)&=c^k_{1,j},\ \phi^{1,k}_j(1)=(-1)^kc^k_{1,j},\, j=1,2,\ldots,m.
	\end{flalign*}And for its derivative,
	\begin{flalign*}
		(\phi^{1,k}_j)'(x_j)&=-k\pi c^k_{1,j}\sin(k\pi x_j)+k\pi c^k_{2,j}\cos(k\pi x_j),\\
		(\phi^{1,k}_j)'(0)&=k\pi c_{2,j}^k,\ (\phi^{1,k}_j)'(1)=(-1)^kk\pi c^k_{2,j},\, j=1,2,\ldots,m.
	\end{flalign*}
	Define the vector \begin{flalign*}
		C_1^k:=\bigl[c^k_{1,1},c^k_{1,2},\ldots,c^k_{1,m}\bigr]^\top,\ 
		C_2^k:=\bigl[c^k_{2,1},c^k_{2,2},\ldots,c^k_{2,m}\bigr]^\top.
	\end{flalign*} By the Neumann--Kirchhoff condition \eqref{eq:kirchhoff}, for $ k\in\mathbb{N}^+$ and $i\in\mathbb{N}$, 
	\begin{flalign*}
		\Phi^+(\phi^{1,k})'(\mathbf{0})-\Phi^-(\phi^{1,k})'(\mathbf{1})=\begin{cases}
			k\pi\Phi C_2^k=0,\ &k=2i,\\
			k\pi\left(\Phi^++\Phi^-\right)C_2^k=k\pi M(\Gamma)C_2^k=0,\ &k=2i+1.
		\end{cases}
	\end{flalign*}
	Therefore $c^k_{2,1}=c^k_{2,2}=\cdots=c^k_{2,m}=0.$ The continuity condition implies that if for some $j_0$, the coefficient $c^k_{1,j_0}$ vanishes, then $c^k_{1,j}=0$ for every $j=1,\dots,m$. Consequently, none of the coefficients of $\cos(k\pi x_j)$ can be zero, which yields $\phi^{1,k}_j \not\equiv 0$ on every edge $e_j$. Thus we finish the proof of \textit{Claim 1}.
\end{proof}
\begin{proof}[Proof of Claim 2]
By the fundamental rank inequality,
	\begin{align}\label{bound1}
		d-1&=\text{rank}(M(S))+\text{rank}(G)-(m+1)\notag\\
		&\le \text{rank}(M(S)^\top G)\le \min\{\text{rank}(M(S)),\text{rank}(G)\}=d.
	\end{align}
	Moreover, from \cite[Lemma 8.2.3]{godsil2013algebraic} we have the relation $M(S)M(S)^\top=D(S)+A(S)$ and thus
	\begin{align*}
		M(S)M(S)^\top G=D(S)G+A(S)G.
	\end{align*}
	Multiplying both sides by the matrix $D^{-1}$ yields
	\begin{align*}
		D^{-1}(S)M(S)M(S)^\top G=G+D^{-1}(S)A(S)G=G+\tilde{A}(S)G=G.
	\end{align*}
	Consequently
	\begin{align*}
		d &= \text{rank}(D^{-1}(S)M(S)M(S)^\top G)
		= \text{rank}(M(S)M(S)^\top G) \\
		&\le \min\{\text{rank}(M(S)), \text{rank}(M(S)^\top G)\}
		\le \text{rank}(M(S)^\top G).
	\end{align*}
	Combining this and \eqref{bound1} yields  $\text{rank}(M(S)^\top G)=d$, which proves the \textit{Claim 2}.
\end{proof}
\begin{proof}[Proof of Claim 3]
	Let $P\neq0$ be any one eigenvector of $A(S)$ satisfying $A(S)P=0$. Then every component of $P$ corresponds to a vertex of $S$. Suppose, to the contrary, that there exists an edge $e\in E(S)$ such that $P(e(0)) \neq 0$ and $P(e(1)) \neq 0$. Then the vertex-induced subgraph $\langle \mathbf{Supp}(P) \rangle$ contains a connected component, which is necessarily a connected non‑trivial tree.
	It contains a vertex $r \in V(S)$ with deg$(r)=1$ whose unique neighbor in $S$ is denoted by $w \in V(S)$. Moreover, because $r, w \in \mathbf{Supp}(P)$, we have $P(r) \neq 0$ and $P(w) \neq 0$. The condition $A(S)P = 0$ implies that
	\begin{flalign*}
		0=\left(A(S)P\right)_r=\sum_{v'\sim r}P(v')=P(w)+\sum_{v'\sim r,v'\neq w}P(v')=P(w)\neq0,
	\end{flalign*}
	which causes a contradiction. Therefore, the support of $P$ must consist of isolated vertices. 
\end{proof}
\bibliography{reference}
\section{Appendix}\label{appendix}
\subsection{Additional numerical tests}
\subsubsection{Some supplementary figures for Section \ref{numerical}}\label{im_supp}
This section supplies the numerical verification for exponential ergodicity for SPDEs on star graph and $S$-atom $T'$. We plot the evolution of empirical average \eqref{invariant} at each time node. For star graph, we set the initial values to be: $X_0^{(1)}=0,\, X_0^{(2)}=\sum_{l=1}^{N}\phi^{1,l},\,X_0^{(3)}=\sum_{l=1}^{N}\xi^{1,l}\phi^{1,l}+\sum_{i=1}^{3}\sum_{l=1}^{N}\tilde{\xi}^{i,l}\Psi^{i,l}$, where $\{\xi^{1,l},\, l=1,2,\ldots,N\},\ \{\tilde{\xi}^{i,l},\, i=1,2,3,\, l=1,2,\ldots,N\}$ are all independently drawn from the standard normal distribution $\mathscr{N}(0,1)$. Fig.~\ref{IM_star} displays that all curves exhibit exponential convergence to the same limiting value, under three different initial values for each noise configuration, verifying the existence of an invariant measure and its exponential ergodicity for \eqref{SEE} on the star graph (see Theorem \ref{invariantmeasureth}).
\begin{figure}[htbp]
	\centering
	\includegraphics[width=0.65\textwidth]{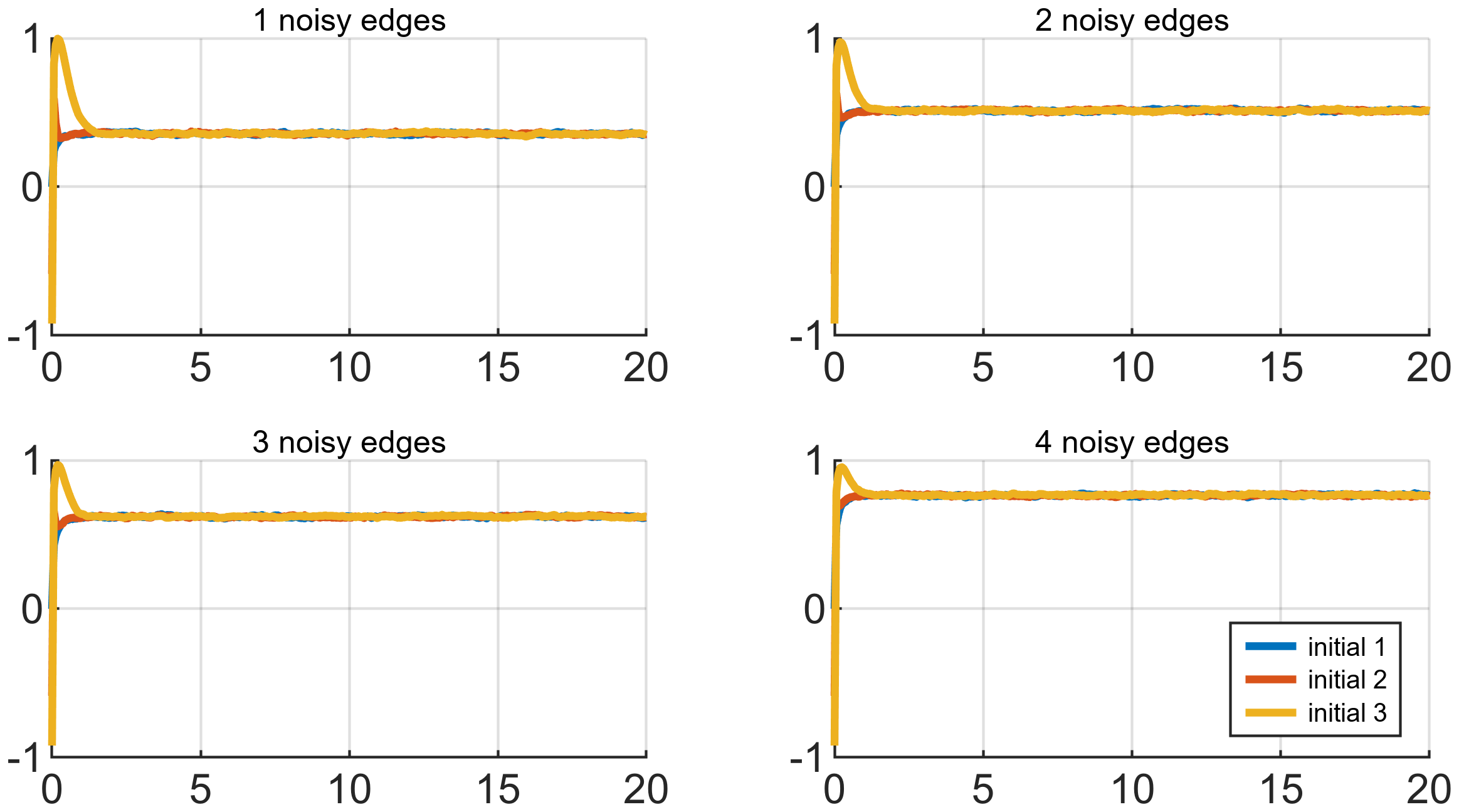}
	\caption{Exponential ergodicity verification for \eqref{SEE} on star graph with 4 edges. $x$ axis: time, $y$ axis: empirical average. $N=2^6,\, \tau=2^{-3},\,  T=20,\,M_{traj}=500,\, \psi(h)=\sin(\|h\|_H),\,h\in H$. 3 initial values: $X_0^{(1)},\, X_0^{(2)},\, X_0^{(3)}$.}
	\label{IM_star}
\end{figure}

For SPDE \eqref{SEE} on $S$-atom $T'$, we consider three types of initial data: $X_0^{(1)}=0,\, X_0^{(2)}=\sum_{l=1}^{N}\phi^{1,l},\,X_0^{(3)}=\sum_{l=1}^{N}\xi^{1,l}\phi^{1,l}+\sum_{i=1}^{6}\sum_{l=1}^{N}\tilde{\xi}^{i,l}\Psi^{i,l}$, where $\{\xi^{1,l},\, l=1,2,\ldots,N\},\ \{\tilde{\xi}^{i,l},\, i=1,2,\dots,6,\, l=1,2,\ldots,N\}$ are all independently drawn from the standard normal distribution $\mathscr{N}(0,1)$. Fig.~\ref{IM_T} displays that empirical averages \eqref{im} converge to the same limiting value in an exponential rate, under three different initial values for each noise configuration. This numerically verifies that the SPDE \eqref{SEE} on the $S$-atom $T'$ has a unique invariant measure which is exponentially ergodic. 

\begin{figure}[htbp]
	\centering
	{
		\includegraphics[width=0.7\textwidth]{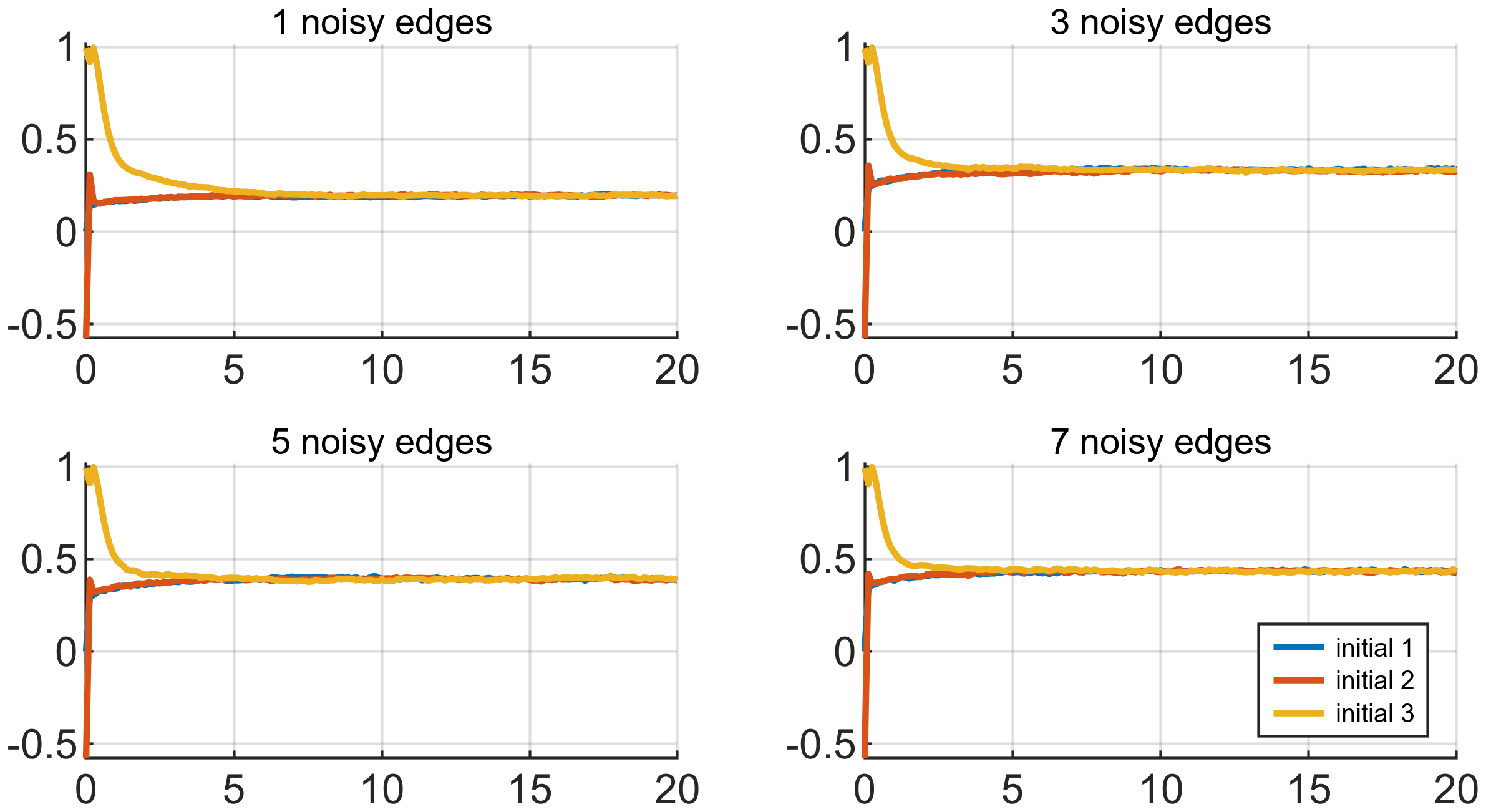}
	}
	\caption{Exponential ergodicity verification for \eqref{SEE} on $S$-atom $T'$. $x$ axis: time, $y$ axis: empirical average. $N=2^5,\,\tau=2^{-3},\, T=20,\, M_{traj}=500,\, \psi(h)=\sin(\|h\|_H),\,h\in H$, 3 initial values: $X_0^{(1)},\, X_0^{(2)},\, X_0^{(3)}$.}
	\label{IM_T}
\end{figure}
\subsubsection{Linear SPDE on graph}\label{linear verify}
This section gives the numerical verification of strong Feller property and irreducibility for linear SPDE on graph. Parameters are the same as those of the nonlinear case. 

\textbf{(I) Chain graph.}

\textbf{Verification of strong Feller property}.

 We use the Monte Carlo method to compute the difference \eqref{lineardiffernece}. In Fig.~\ref{SF_linear}(A), all lines except the blue one decay to $0$ as $\epsilon\to0$. This means that $\mathcal{R}_T$ is strong Feller, provided noise acts on at least one edge.
\begin{figure}[htbp]
	\centering
	\subfloat[chain graph]
	{
		\includegraphics[width=0.30\textwidth]{SF_chain_linear}
	}
	\quad
	\subfloat[star graph]
	{
		\includegraphics[width=0.30\textwidth]{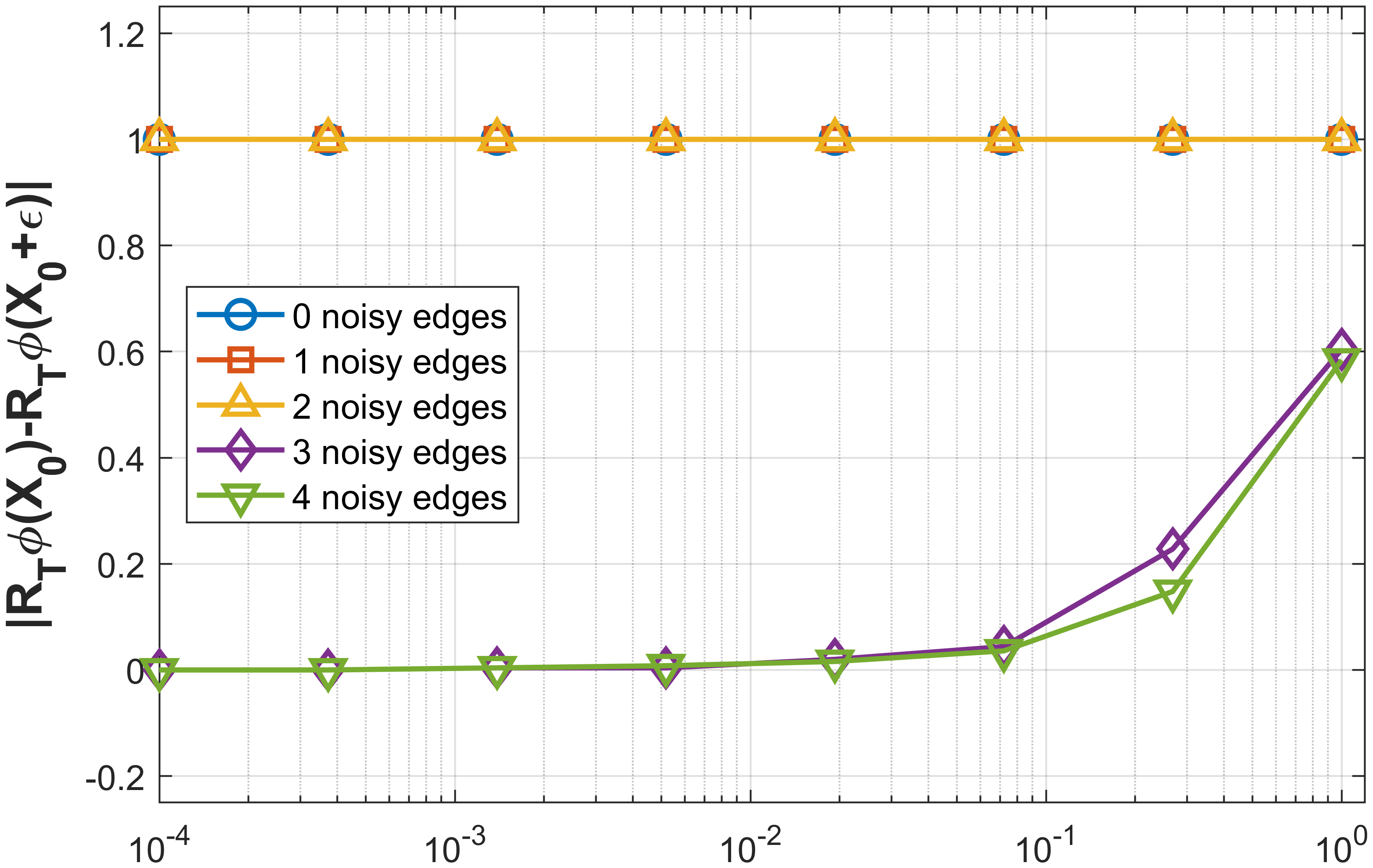}
	}
\quad
\subfloat[$S$-atom $T'$]
{
	\includegraphics[width=0.30\textwidth]{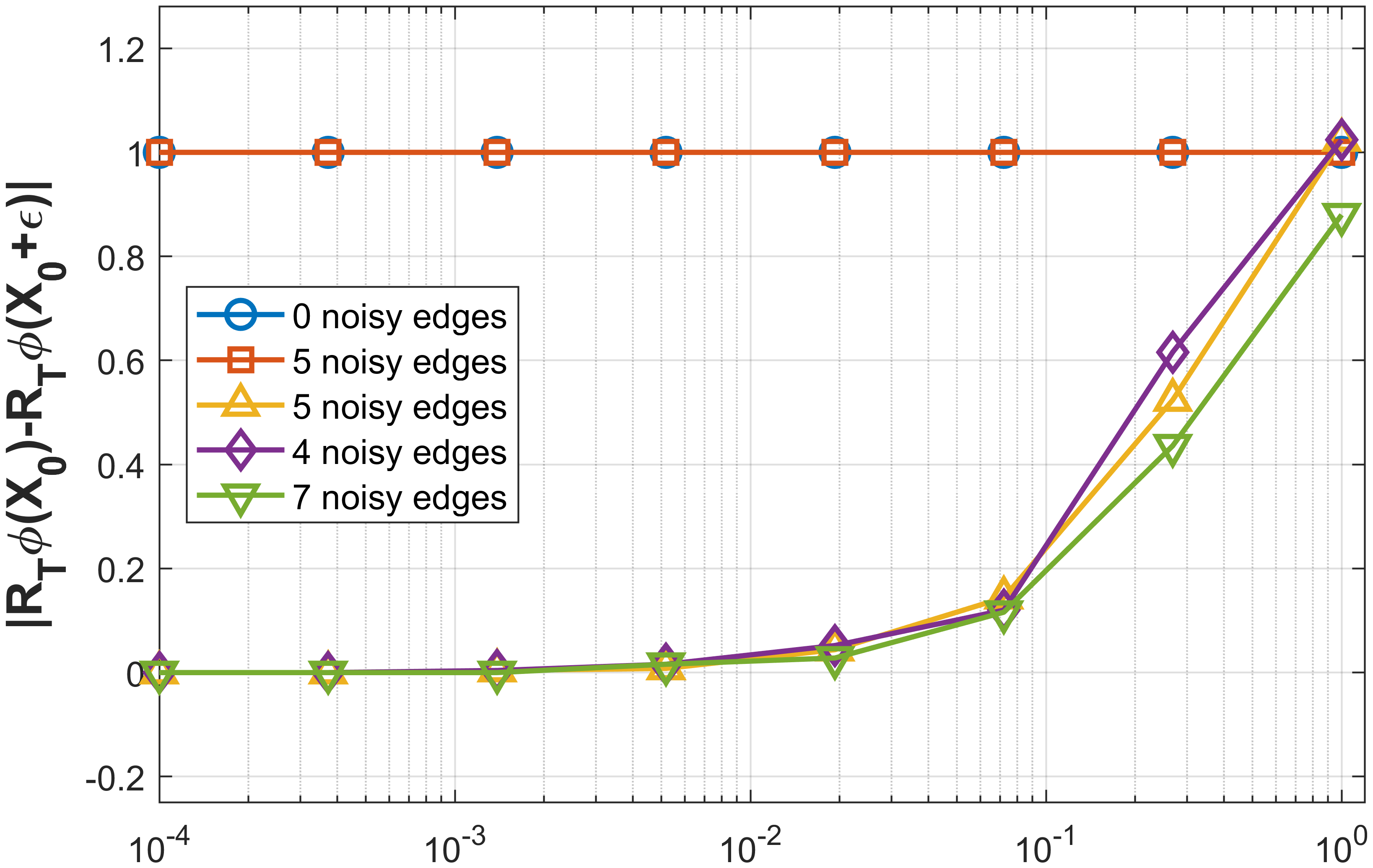}
}
	\caption{Strong Feller verification for $\mathcal{R}_T$ on different trees. $x$ axis: $\epsilon$, $y$ axis: the difference~\eqref{lineardiffernece}.  $N=2^6,\, \tau=2^{-5},\, T=2^{-1},\, \epsilon \in \{ 10^{-\tfrac{4k}{7}} : k = 0,1,\dots,7\},\, M_{traj}=500$.}
	\label{SF_linear}
\end{figure}

\textbf{Verification of irreducibility}. We estimate reachability probabilities for every eigen-subspace via \eqref{irreducible_nonlinear}. Results for $\mathcal{R}_T$ is shown in Fig.~\ref{IR_chain_linear}. Columns 1-3 present the eigen‑directions of $\sigma_2$, and column 4 presents those of $\sigma_1$. We observe that the probabilities across all eigen‑directions are positive, except for the zero noisy edge case (the blue line). This demonstrates that $\mathcal{R}_T$ is irreducible at time $T$, provided more than one edge is noisy.  
\begin{figure}[htbp]
	\centering
	{
		\includegraphics[width=0.85\textwidth]{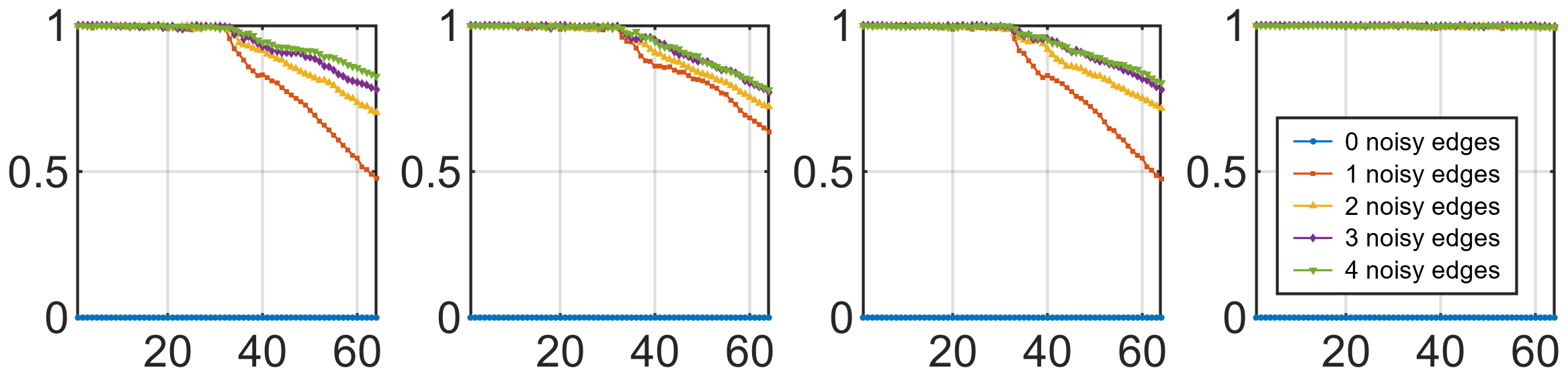}
	}
	\caption{Irreducibility verification for $\mathcal{R}_T$ on chain graph with 4 edges. Subfigures from left to right: $(\Psi^{1,l})_{1\le l\le N},\,(\Psi^{2,l})_{1\le l\le N},\,(\Psi^{3,l})_{1\le l\le N},\,(\phi^{1,k})_{0\le k\le N-1}$. $x$ axis: dimension $N$, $y$ axis: reachability probability. $N=2^6,\, \tau=2^{-5},\, T=2^{-1},\, M_{traj}=500$.}
	\label{IR_chain_linear}
\end{figure}

\textbf{(II) Star graph.}

\textbf{Verification of strong Feller property}. Fig.~\ref{SF_linear} (B) shows that only for the cases of $3$ noisy edges (purple line) and $4$ noisy edges (green line), \eqref{lineardiffernece} decays to $0$ when $\epsilon$ converges to $0$. This means that the semigroup $(\mathcal{R}_t)_{t\ge0}$ on the star graph is strong Feller at time $T$ with at most one noise exception.

\textbf{Verification of irreducibility}. In
Fig.~\ref{IR_star_linear}, columns 1-3 present the reachability probability \eqref{irreducible_nonlinear} for eigen‑directions corresponding to $\sigma_2$, and column 4 presents those to $\sigma_1$. As shown in the first column, only the cases of 3 noisy edges (purple line) and 4 noisy edges (green line) make reachability probabilities of eigen-directions $\Psi^{1,l},\,l=1,2,\ldots,N$ positive. And the two cases also produce positive reachability probabilities for other eigen-directions, as seen in columns 2-4. Consequently, $(\mathcal{R}_t)_{t\ge0}$ is irreducible at time $T$, provided all edges are noisy with one edge exception.  
\begin{figure}[htbp]
	\centering
	{
		\includegraphics[width=0.85\textwidth]{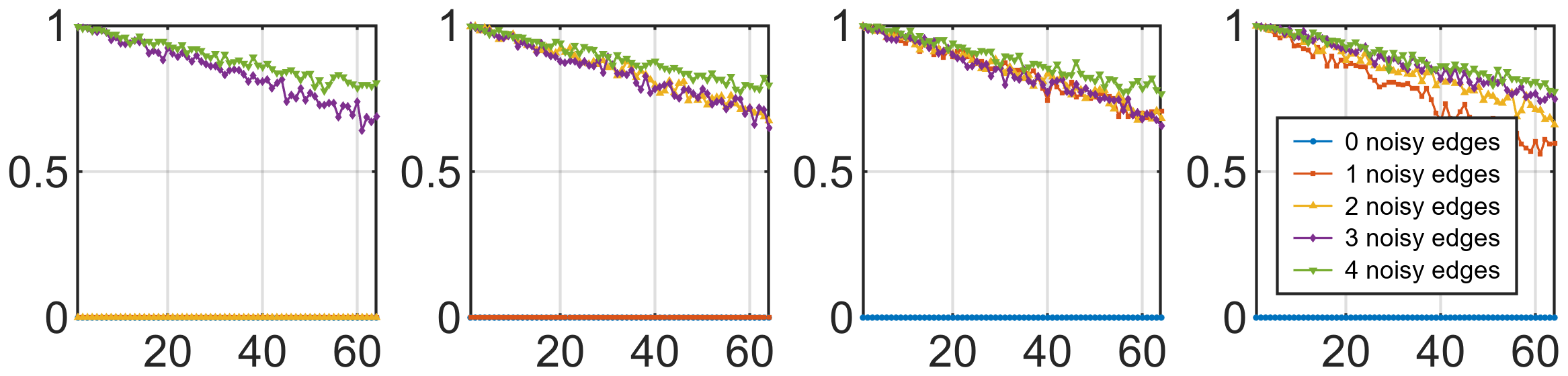}
	}
	\caption{Irreducibility verification for $\mathcal{R}_T$ on star graph with 4 edges.  Subfigures from left to right: $(\Psi^{1,l})_{1\le l\le N},\,(\Psi^{2,l})_{1\le l\le N},\,(\Psi^{3,l})_{1\le l\le N},\,(\phi^{1,k})_{0\le k\le N-1}.$ $x$ axis: dimension $N$, $y$ axis: reachability probability. $N=2^6,\, \tau=2^{-3},\, T=2^{-1},\, M_{traj}=500$.}
	\label{IR_star_linear}
\end{figure}

\textbf{(III) $S$-atom $T'$.}

\textbf{Verification of strong Feller property}.
Fig.~\ref{SF_linear} (C) shows that the red and blue lines remain constant at $1$, whereas the other three lines decay to $0$. This indicates that the semigroup $(\mathcal{R}_t)_{t\ge0}$ on $T'$ is not strong Feller at time $T$ when $\mathcal{Z}=\{e_{67},e_{68}\}$ or $\mathcal{Z}=E(T')$, but is strong Feller when $\mathcal{Z}=\{e_{12},e_{67}\}$,  $\mathcal{Z}=\{e_{12},e_{15},e_{67}\}$ or $\mathcal{Z}=\varnothing$.

\textbf{Verification of irreducibility}. In
Fig.~\ref{IR_T_linear}, the first six subfigures display six eigenfunctions corresponding to  $\sigma_2$ (the first two for $\frac{\sqrt{102}}{12} $ and $-\frac{\sqrt{102}}{12}$, the next four for $0$) and the last subfigure for eigenfunctions of $\sigma_1$. The red lines in the sixth subfigure of Fig.~\ref{IR_T_linear} as well as all the blue lines show zero reachability probability. Consequently, for $\mathcal{Z}=\{e_{67},e_{68}\}$ and $\mathcal{Z}=E(T')$, $(\mathcal{R}_t)_{t\ge0}$ is not irreducible at time $T$. But the semigroup is irreducible at $T$ for either the cases $\mathcal{Z}=\{e_{12},e_{67}\}$, $\mathcal{Z}=\{e_{12},e_{15},e_{67}\}$ or $\mathcal{Z}=\varnothing$. 

\begin{figure}[htbp]
	\centering
	{
		\includegraphics[width=0.75\textwidth]{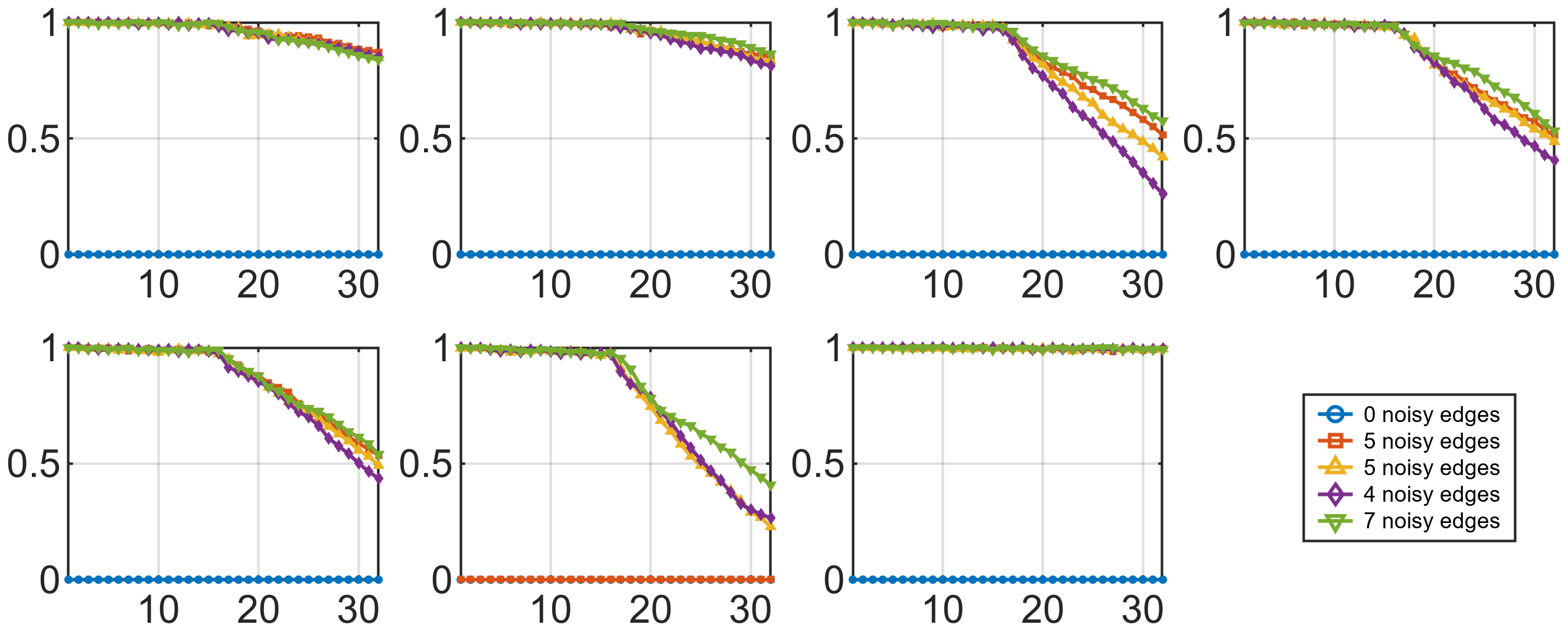}
	}
	\caption{Irreducibility verification for  $\mathcal{S}_T$ on $S$-atom  $T'$.  Seven subfigures from left to right: $(\Psi^{1,l})_{1\le l\le N},\ldots,(\Psi^{6,l})_{1\le l\le N},(\phi^{1,k})_{0\le k\le N-1}$. $x$ axis: dimension $N$, $y$ axis: reachability probability. $N=2^5,\, \tau=2^{-4},\, T=2^{-1},\,M_{traj}=500$.}
	\label{IR_T_linear}
\end{figure}
\subsubsection{SPDE with non-Lipschitz nonlinearity on graph}\label{nonlip}
For SPDE \eqref{SEE} on chain graph and star graph with the nonlinearity $b_j(x)=x-x^3,\, j=1,2,\ldots,m,\, x\in\mathbb{R}$, we present some numerical observations concerning the strong Feller property, irreducibility and exponential ergodicity of a unique invariant measure in Figs.~\ref{SF_poly}, \ref{IR_poly} \ref{IM_chain_poly} and \ref{IM_star_poly}, which show the similar results to the Lipschitz nonlinearity case (see Corollaries \ref{bound_chain}, \ref{bound_star} and Theorem \ref{invariantmeasureth}).
\begin{figure}[htbp]
	\centering
	\subfloat[Chain graph with $4$ edges]
	{
		\includegraphics[width=0.4\textwidth]{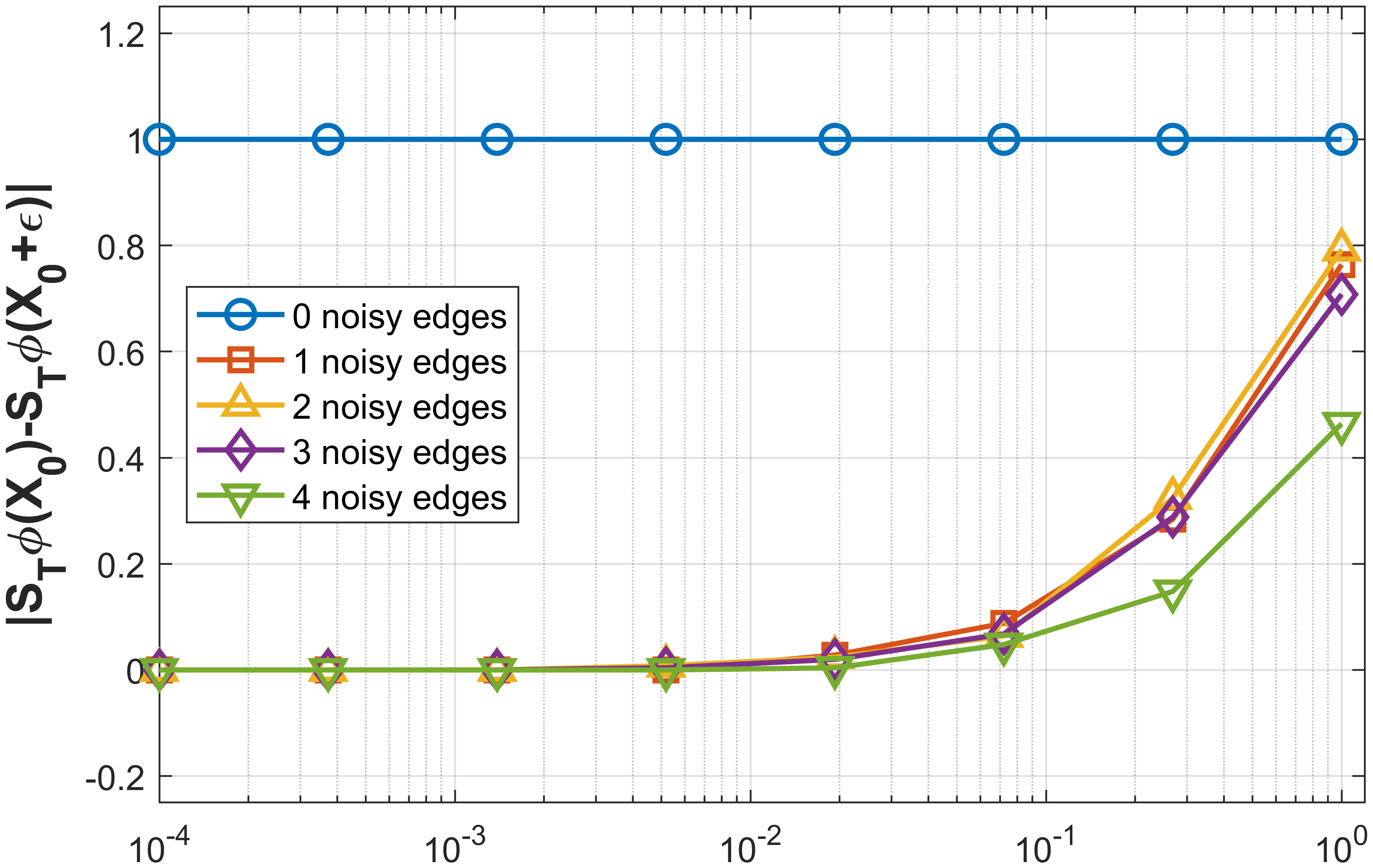}
	}
	\quad
	\subfloat[Star graph with $4$ edges]
	{
		\includegraphics[width=0.4\textwidth]{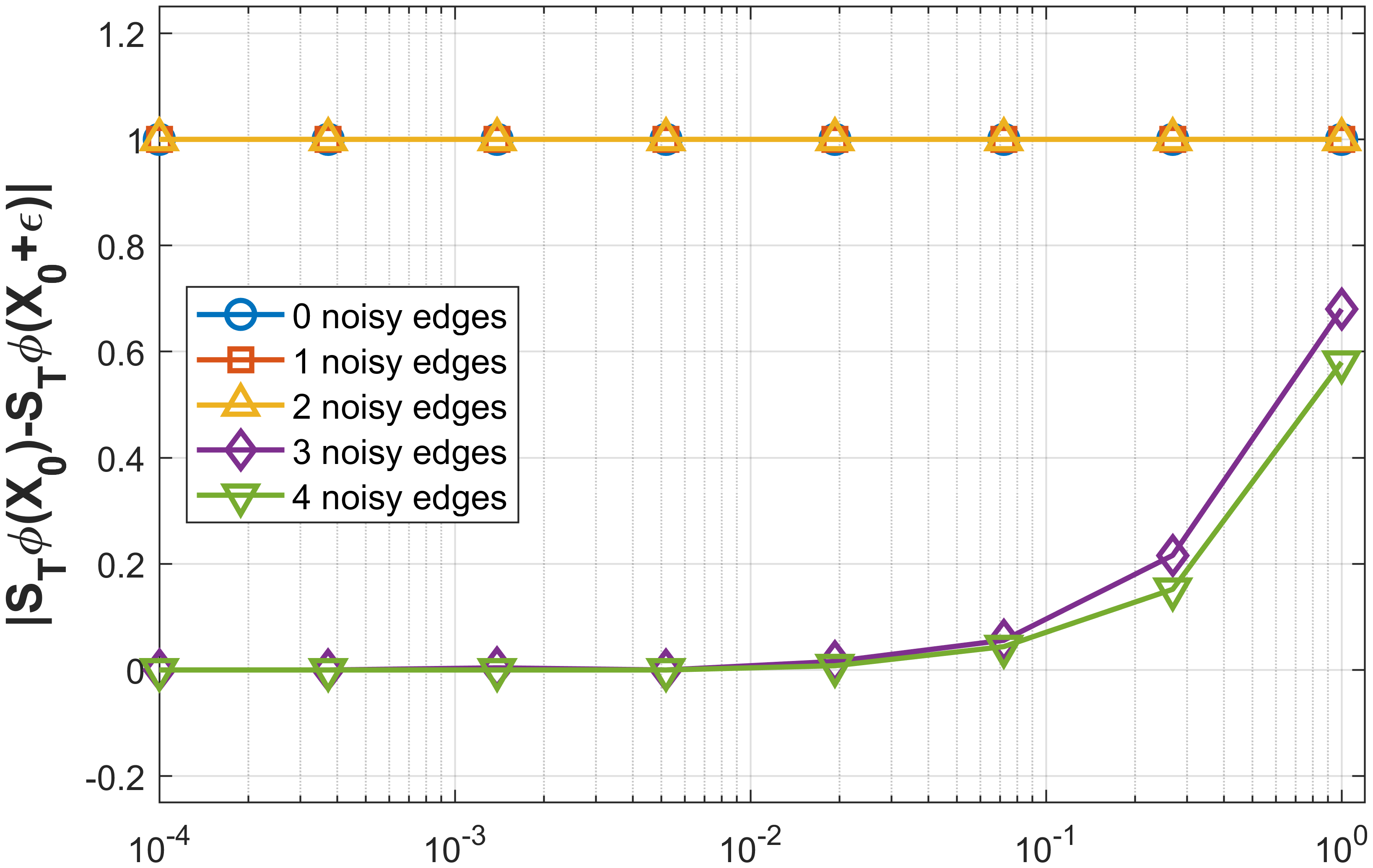}
	}
	\caption{Strong Feller verification SPDE \eqref{SEE} with  $b_j(x)=x-x^3$.}
	\label{SF_poly}
\end{figure}
\begin{figure}[htbp]
	\centering
	{
		\includegraphics[width=0.75\textwidth]{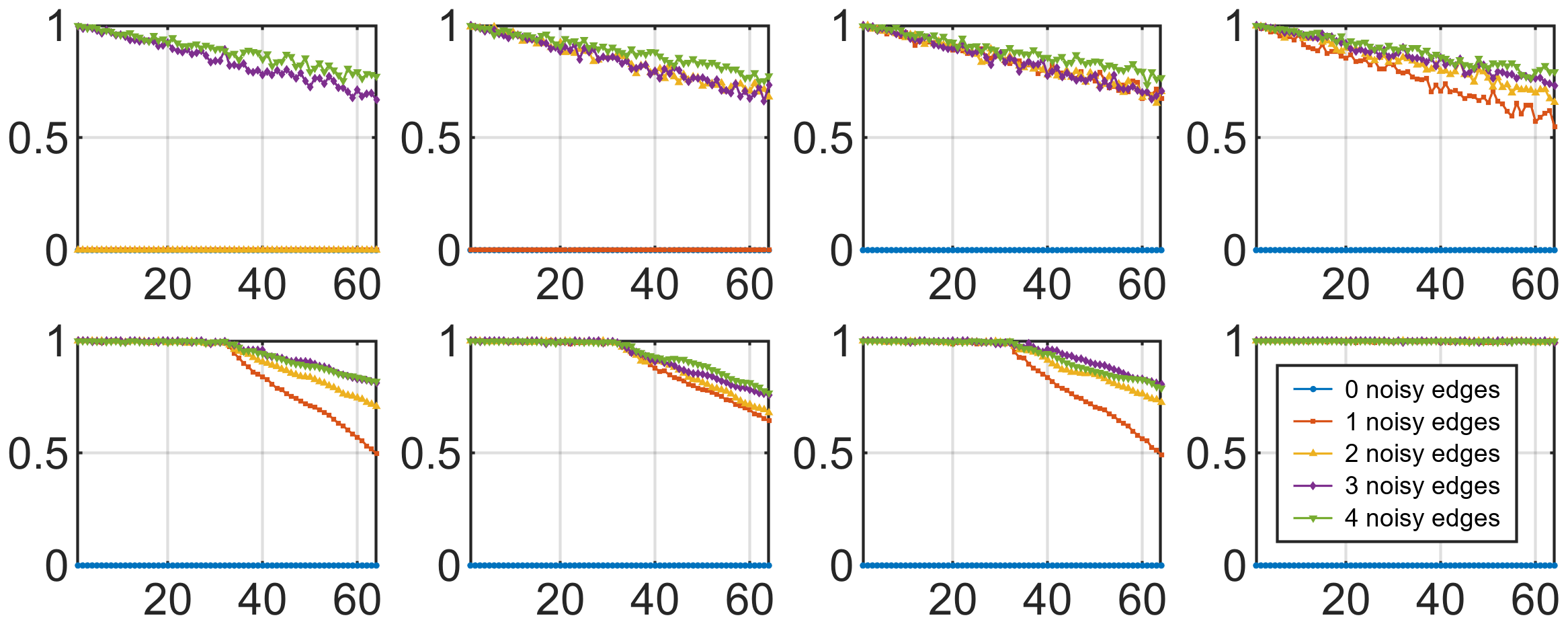}
	}
	\caption{Irreducibility verification SPDE \eqref{SEE} with  $b_j(x)=x-x^3$ (top for chain graph and bottom for star graph).}
	\label{IR_poly}
\end{figure}
\begin{figure}
    \centering
    {
    \includegraphics[width=0.7\linewidth]{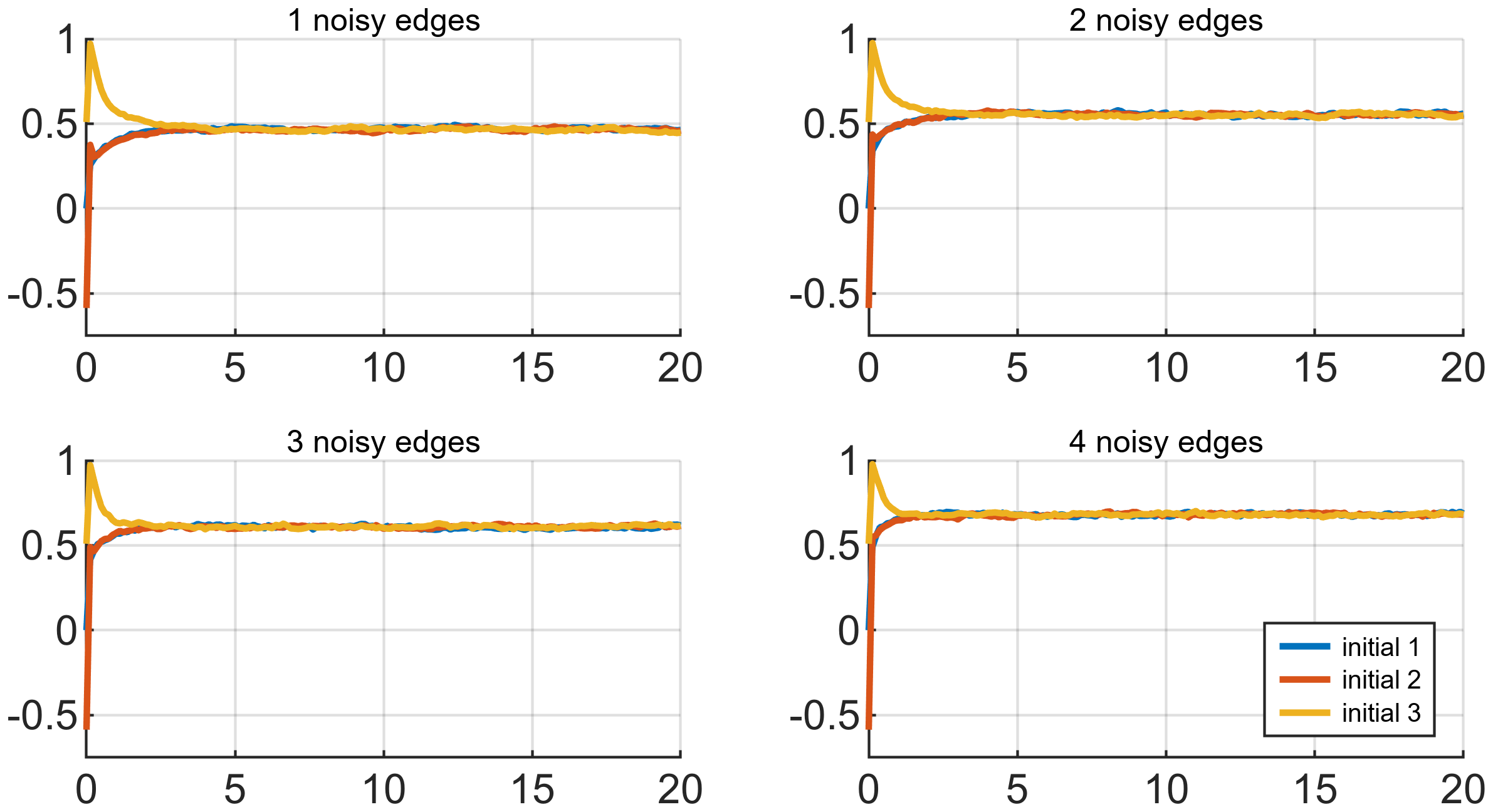}
    }
    \caption{Exponential ergodicity verification for SPDE \eqref{SEE} with  $b_j(x)=x-x^3$ on chain graph with 4 edges. }
    \label{IM_chain_poly}
    \end{figure}
\begin{figure}
    \centering
    {
    \includegraphics[width=0.7\linewidth]{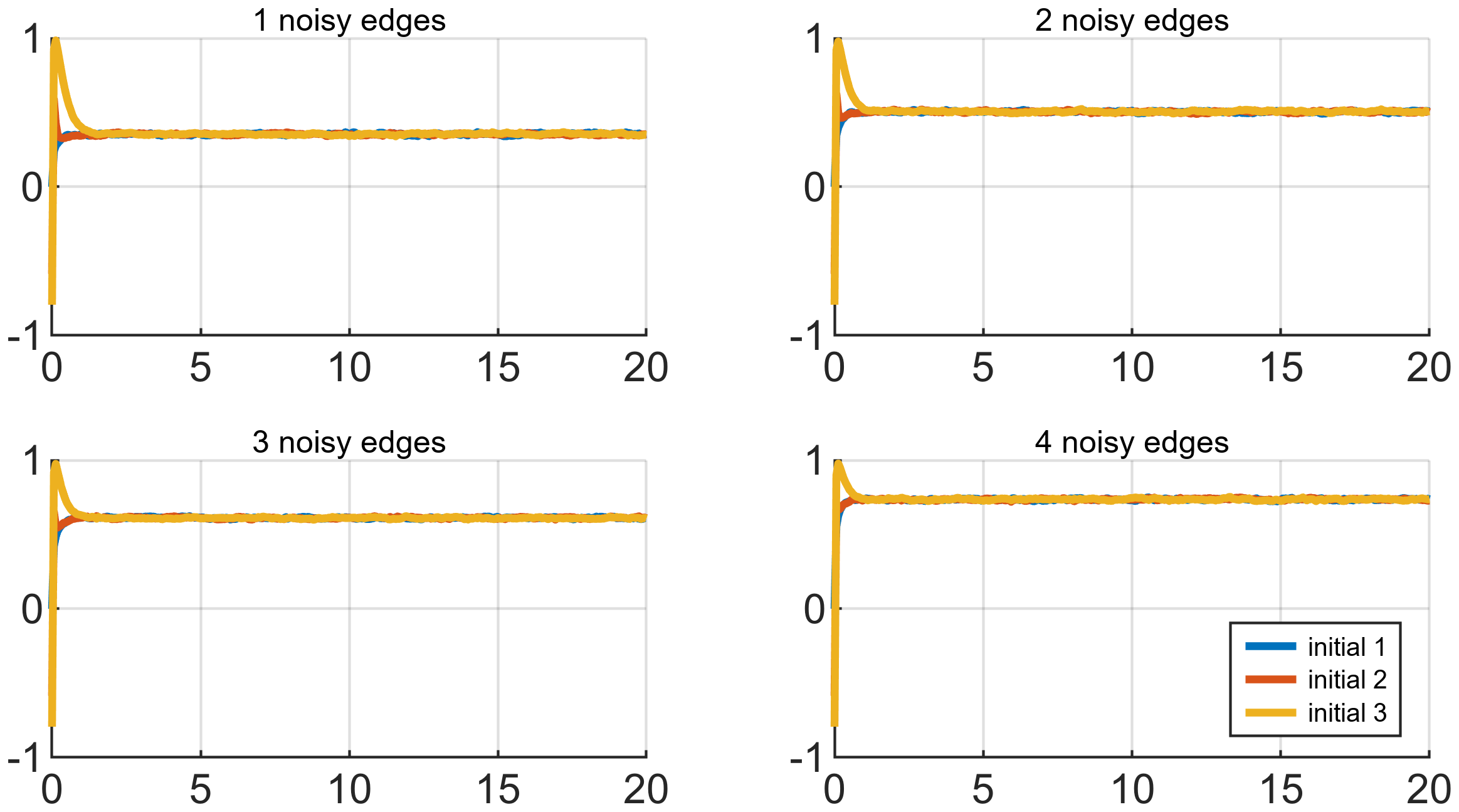}
    }
    \caption{Exponential ergodicity verification for SPDE \eqref{SEE} with  $b_j(x)=x-x^3$ on star graph with 4 edges. }
    \label{IM_star_poly}
    \end{figure}
\subsection{Proofs of \eqref{spatial regularity}  and \eqref{attractive}.}\label{proof of imth}
This section gives the proofs of estimates \eqref{spatial regularity} and \eqref{attractive}. To prove uniform time regularity estimates of the mild solution, we present an estimate for the stochastic convolution.
\begin{lemma}\label{stochasticconvolution}
	Denote by  $W_A(t):=\int_{0}^{t}P_{t-s}QdW(s)$ the stochastic convolution. For any $\alpha\in[0,\frac{1}{4})$, 
	\begin{flalign*}
		\sup_{t>0}\mathbb{E}\left[\|W_A(t)\|_{\mathbb{H}^{2\alpha}}^2\right]<+\infty.
	\end{flalign*}
	\begin{proof}
		Since $\|Q\|_{\mathcal{L}(H)} = 1$, for any $t > 0$ we have
		\begin{align*}
			\mathbb{E} \left[\| W_A(t) \|_{\mathbb{H}^{2\alpha}}^2\right]
			&= \mathbb{E} \Bigl[ \Bigl\| A^\alpha \int_0^t P_{t-s} Q \, dW(s) \Bigr\|_H^2 \Bigr]
			= \int_0^t \| A^\alpha P_{t-s} Q \|_{L_2^0(H)}^2 \, ds \\
			&\le \int_0^t \| A^\alpha P_{t-s} \|_{L_2^0(H)}^2 \, \| Q \|_{\mathcal{L}(H)}^2 \, ds= \int_0^t \| A^\alpha P_{t-s} \|_{L_2^0(H)}^2 \, ds\\
			&= \int_0^t \sum_{k=1}^{\infty} \mu_k^{2\alpha} e^{-2\mu_k s} \, ds= \sum_{k=1}^{\infty} \mu_k^{2\alpha} \frac{1 - e^{-2\mu_k t}}{2\mu_k}= \frac12 \sum_{k=1}^{\infty} \mu_k^{2\alpha-1} \bigl( 1 - e^{-2\mu_k t} \bigr).
		\end{align*}
		By $\mu_k \sim k^2$, the series converges when $\alpha \in [0,\tfrac14)$, which finishes the proof.
	\end{proof}
\end{lemma}
\begin{prop}\label{im}
	Under Assumption \ref{b1} and the condition in Theorem \ref{invariantmeasureth}, the mild solution $X(t,X_0)$ of \eqref{SEE} satisfies \eqref{spatial regularity}  and \eqref{attractive}.
\end{prop}
\begin{proof}
	We first show that  $\sup_{t>0}\mathbb{E}\left[\|X(t,X_0)\|_H^2\right]\le C(X_0,\mu_1,m,K)$.
	Define $Z(t)=X(t,X_0)-W_A(t)$, which satisfies
	\begin{flalign*}
		\begin{cases}
			dZ(t)=-AZ(t)dt+B(Z(t)+W_A(t))dt,\\
			Z(0)=X_0.
		\end{cases}
	\end{flalign*}
	Then
	\begin{flalign*}
		\frac{1}{2}\frac{d\|Z(t)\|_H^2}{dt}&=\langle -A Z(t)+B(Z(t)+W_A(t)),Z(t)\rangle\\
		&=\langle -AZ(t),Z(t)\rangle+\langle B(Z(t)),Z(t)\rangle+\langle B(Z(t)+W_A(t))-B(Z(t)),Z(t)\rangle\\
		&\le -\mu_1\|Z(t)\|_H^2+K\left(\|Z(t)\|_H^2+\sqrt{m}\|Z(t)\|_H\right)+K\|W_A(t)\|_H\|Z(t)\|_H\\
		&\le -\mu_1\|Z(t)\|_H^2+K\left(\|Z(t)\|_H^2+\sqrt{m}\|Z(t)\|_H\right)+\tfrac{\epsilon}{2}\|Z(t)\|_H^2+C(\epsilon,K)\|W_A(t)\|_H^2\\
		&\le -\mu_1\|Z(t)\|_H^2+K\|Z(t)\|_H^2+\epsilon \|Z(t)\|_H^2+ C(\epsilon,m,K)+C(\epsilon,K)\|W_A(t)\|_H^2\, a.s.,
	\end{flalign*}
	where we used $
	\langle -A u,u\rangle=\sum_{k=1}^{+\infty}-\mu_ku_k^2\le -\mu_1\|u\|_H^2
	$, 
	\begin{flalign*}
		\langle B(Z(t)),Z(t)\rangle&=\sum_{i=1}^{m}\langle b_i((Z(t))_i),(Z(t))_i\rangle\le\sum_{i=1}^{m}\|b_i((Z(t))_i)\|_{L^2(0,1)}\|(Z(t))_i\|_{L^2(0,1)}\\
		&\le K\sum_{i=1}^{m}\left(\|(Z(t))_i\|_{L^2(0,1)}+\|(Z(t))_i\|_{L^2(0,1)}^2\right)\\
		&\le K\Bigl(\|Z(t)\|_H^2+\sqrt{m}\bigl(\sum_{i=1}^{m}\|(Z(t))_i\|_{L^2(0,1)}^2\bigl)^\frac12\Bigl)= K\left(\|Z(t)\|_H^2+\sqrt{m}\|Z(t)\|_H\right),
	\end{flalign*}
	and Young's inequality.
	By virtue of
	\begin{flalign*}
		\left(e^{2(\mu_1-K-\epsilon)t}\|Z(t)\|_H^2\right)'&=e^{2(\mu_1-K-\epsilon)t}\left(\frac{d\|Z(t)\|_H^2}{dt}+2(\mu_1-K-\epsilon)\|Z(t)\|^2_H\right)\\
		&\le 2\bigl(C(\epsilon,m,K)+C(\epsilon,K)\|W_A(t)\|_H^2\bigr)e^{2(\mu_1-K-\epsilon)t},
	\end{flalign*}
	we have
	\begin{flalign*}		
		\|Z(t)\|_H^2&\le e^{-2(\mu_1-K-\epsilon)t}\|Z(0)\|_H^2+2\bigl(C(\epsilon,m,K)+C(\epsilon,K)\|W_A(t)\|_H^2\bigr)\int_0^te^{-2(\mu_1-K-\epsilon)(t-s)}ds\\
		&=e^{-2(\mu_1-K-\epsilon)t}\|X_0\|_H^2+\frac{\bigl(C(\epsilon,m,K)+C(\epsilon,K)\|W_A(t)\|_H^2\bigr)}{(\mu_1-K-\epsilon)}(1-e^{-2(\mu_1-K-\epsilon)t}).
	\end{flalign*}
	By taking $\epsilon=\frac{\mu_1-K}{2}$ we derive
	$
	\sup_{t>0}\mathbb{E}\bigl[\|X(t,X_0)\|_H^2\bigr]\le  C(X_0,\mu_1,m,K).
	$
	
	Next we prove $\sup_{t\geq1}\mathbb{E}\left[\| X(t,X_0)\|_{\mathbb{H}^{2\alpha}}\right]\leq C$ for $\alpha\in(0,\frac{1}{4})$. By the relation of $\mathcal{D}(A^\frac{\alpha}{2})$ and $\mathbb{H}^\alpha$, 	\begin{flalign*}
		\mathbb{E}\left[\|X(t,X_0)\|_{\mathbb{H}^{2\alpha}}\right]&=\mathbb{E}\left[\Big\|A^\alpha\Bigl(P_tX_0+\int_0^tP_{t-s}B(X(s,X_0))ds+W_A(t)\Bigl)\Big\|_H\right]\\
		&\le\mathbb{E}\left[\left\|A^\alpha P_tX_0\right\|_H+\Big\|A^\alpha\int_0^tP_{t-s}B(X(s,X_0))ds\Big\|_H+\|A^\alpha W_A(t)\|_H\right].
	\end{flalign*}
	For the first term,
	\begin{flalign}\label{firstterm}
		\left\|A^\alpha P_tX_0\right\|_H^2&=\sum_{k=1}^{+\infty}\mu_k^{2\alpha}e^{-2\mu_kt}u_k^2\le \sup_{x>0}x^{2\alpha}e^{-2xt}\sum_{k=1}^{+\infty}u_k^2\le C(\alpha)t^{-2\alpha}\|X_0\|_H^2.
	\end{flalign}
	For the second term, we point out that for any $h\in H$,
	\begin{flalign*}
		\|B(h)\|_H&=\Bigl(\sum_{i=1}^m\|b_i(h_i)\|_{L^2(0,1)}^2\Bigl)^\frac{1}{2}\le\sum_{i=1}^m\|b_i(h_i)\|_{L^2(0,1)}\\
		&\le  K\sum_{i=1}^{m} \left(1+\|h_i\|_{L^2(0,1)}\right)\le  K\left(m+\sqrt{m}\|h\|_H\right)=\sqrt{m}K(\sqrt{m}+\|h\|_H).
	\end{flalign*}
	Since $t>1$, 
	\begin{flalign*}
		&\quad\mathbb{E}\left[\Big\|A^\alpha\int_0^tP_{t-s}B(X(s,X_0))ds\Big\|_H\right]\le \mathbb{E}\left[\int_0^t\left\|A^\alpha P_{t-s}B(X(s,X_0))\right\|_Hds\right]\\
		&\le\mathbb{E}\left[\int_0^t\left\|A^\alpha P_{t-s}\right\|_{\mathcal{L}(H)}\left\|B(X(s,X_0))\right\|_Hds\right]\le \sqrt{m}K \mathbb{E}\left[\int_0^t\left\|A^\alpha P_{t-s}\right\|_{\mathcal{L}(H)}\left(\sqrt{m}+\left\|X(s,X_0)\right\|_H\right)ds\right]\\
		&\le \sqrt{m}K\bigl(\sqrt{m}+\sup_{s>0}\mathbb{E}\bigl[\|X(s,X_0)\|_H\bigr]\bigr)\int_0^t\|A^\alpha P_{t-s}\|_{\mathcal{L}(H)}ds\\
		&\le \sqrt{m}K\bigl(\sqrt{m}+\sup_{s>0}\mathbb{E}\bigl[\|X(s,X_0)\|_H\bigr]\bigr) \int_0^\infty\|A^\alpha P_{s}\|_{\mathcal{L}(H)}ds.
	\end{flalign*}
Notice that
\begin{flalign*}
	\|A^\alpha P_{s}\|_{\mathcal{L}(H)}=\sup_{k\geq1}\mu_k^\alpha e^{-\mu_ks}=\sup_{k\geq1}(\mu_ks)^\alpha s^{-\alpha}e^{-\tfrac{\mu_ks}{2}}e^{-\tfrac{\mu_ks}{2}}\le C(\alpha)s^{-\alpha}e^{-\tfrac{\mu_1s}{2}},
\end{flalign*}
	since $\sup_{x>0}x^\alpha e^{-\tfrac{x}{2}}=(2\alpha)^\alpha e^{-\alpha}$ (attained at $x=2\alpha$). Hence
	\begin{flalign*}
	\int_0^\infty\|A^\alpha P_{s}\|_{\mathcal{L}(H)}ds&\le C(\alpha)\int_0^\infty s^{-\alpha}e^{-\tfrac{\mu_1s}{2}}ds\\
	&\le\int_0^1s^{-\alpha}ds+\int_1^\infty e^{-\tfrac{\mu_1s}{2}}ds\le C(\alpha,\mu_1).
	\end{flalign*}
So
 \begin{flalign}\label{secondterm}
		\sup_{t>0}\mathbb{E}\left[\Big\|A^\alpha\int_0^tP_{t-s}B(X(s,X_0))ds\Big\|_H\right]
		\le C(\alpha,X_0,\mu_1,m,K).
	\end{flalign}
	For the last term, from Lemma \ref{stochasticconvolution},
	\begin{flalign}\label{lastterm}
		\sup_{t>0}\mathbb{E}\bigl[\|A^\alpha W_A(t)\|_H\bigr]\le\sup_{t>0}\left(\mathbb{E}\bigl[\|A^\alpha W_A(t)\|_H^2\bigr]\right)^\frac{1}{2}<\infty.
	\end{flalign}
	Combining \eqref{firstterm}, \eqref{secondterm} and \eqref{lastterm}, we finish the proof of \eqref{spatial regularity}.
	
	We now prove \eqref{attractive}. Define $E(t)=X(t,X_0)-X(t,X_1)$. Then $E(t)$ satisfies the equation
	\begin{flalign*}
		dE(t)&=-AE(t)\, dt+\bigl(B(X(t,X_0))-B(X(t,X_1))\bigl)dt.
	\end{flalign*}
	Then we have
	\begin{align*}
		\frac{1}{2}\frac{d}{dt}\|E(t)\|_H^2&=\langle E(t),-AE(t)\rangle+\langle E(t),B(X(t,X_0))-B(X(t,X_1))\rangle\\&\le -\mu_1\|E(t)\|_H^2+\|E(t)\|_H\|B(X(t,X_0))-B(X(t,X_1))\|_H\le -(\mu_1-K)\|E(t)\|_H^2.
	\end{align*}
Consequently, applying the Gronwall inequality and taking the expectation finish the proof.
\end{proof}
\end{document}